\documentclass[10pt,reqno]{amsart}
\usepackage{amsmath,amsthm,amssymb,mathrsfs,stmaryrd,xcolor}
\usepackage[all]{xy}
\usepackage{url}
\usepackage{geometry}
\usepackage[utf8]{inputenc}
\usepackage[T1]{fontenc}

\usepackage{enumitem}
\usepackage[colorlinks=true,hyperindex, linkcolor=magenta, pagebackref=false, citecolor=cyan,pdfpagelabels]{hyperref}
\usepackage[capitalize]{cleveref}
\setlist[enumerate]{itemsep=2pt,topsep=5pt,parsep=2pt,partopsep=2pt}

\newcommand{\cosimp}[3]{\xymatrix@1{#1 \ar@<.4ex>[r] \ar@<-.4ex>[r] & {\ }#2 \ar@<0.8ex>[r] \ar[r] \ar@<-.8ex>[r] & {\ } #3 \ar@<1.2ex>[r] \ar@<.4ex>[r] \ar@<-.4ex>[r] \ar@<-1.2ex>[r] & \cdots }}

\newcommand{\colim}{\mathop{\mathrm{colim}}}
\newcommand{\et}{\mathop{\mathrm{\acute{e}t}}}

\newcommand{\ssp}{\mathop{\mathrm{sp}}}
\newcommand{\Spa}{\mathop{\mathrm{Spa}}}
\newcommand{\Spec}{\mathop{\mathrm{Spec}}}
\newcommand{\Perv}{\mathop{\mathrm{Perv}}}
\newcommand{\Sh}{\mathop{\mathrm{Sh}}}
\newcommand{\inthom}{\mathop{R\mathscr{H}\mathrm{om}}}
\newcommand{\pervleqzero}{\mathop{\phantom{}^{\mathfrak{p}}\!D_{zc}^{\leq 0}}}
\newcommand{\pervgeqzero}{\mathop{\phantom{}^{\mathfrak{p}}\!D_{zc}^{\geq 0}}}
\newcommand{\pervcoh}{\mathop{\phantom{}^{\mathfrak{p}}\mathcal{H}}}

\newcommand{\adjunction}[4]{\xymatrix@1{#1{\ } \ar@<0.3ex>[r]^{ {\scriptstyle #2}} & {\ } #3 \ar@<0.3ex>[l]^{ {\scriptstyle #4}}}}

\begin{document}

\newtheorem{theorem}{Theorem}[section]
\newtheorem*{theorem*}{Theorem}
\newtheorem*{definition*}{Definition}
\newtheorem{proposition}[theorem]{Proposition}
\newtheorem{lemma}[theorem]{Lemma}
\newtheorem{corollary}[theorem]{Corollary}
\newtheorem{conjecture}[theorem]{Conjecture}

\theoremstyle{definition}
\newtheorem{definition}[theorem]{Definition}
\newtheorem{question}[theorem]{Question}
\newtheorem{remark}[theorem]{Remark}
\newtheorem{warning}[theorem]{Warning}
\newtheorem{example}[theorem]{Example}
\newtheorem{notation}[theorem]{Notation}
\newtheorem{convention}[theorem]{Convention}
\newtheorem{construction}[theorem]{Construction}
\newtheorem{claim}[theorem]{Claim}
\newtheorem{assumption}[theorem]{Assumption}

\title{The six functors for Zariski-constructible sheaves in rigid geometry}

\author{Bhargav Bhatt}
\address{Bhargav Bhatt\\Department of Mathematics, University of Michigan\\Ann Arbor, MI 48109, USA}
\email{bhattb@umich.edu}

\author{David Hansen} 
\address{David Hansen\\Max Planck Institute for Mathematics\\Vivatsgasse 7, Bonn 53111, Germany}
\email{dhansen@mpim-bonn.mpg.de}

\keywords{Rigid analytic spaces, \'etale cohomology, generic smoothness, six functors}
\subjclass{14G22, 14F20}

\begin{abstract} 

We prove a generic smoothness result in rigid analytic geometry over a characteristic zero nonarchimedean field. The proof relies on a novel notion of generic points in rigid analytic geometry which are well-adapted to ``spreading out'' arguments, in analogy with the use of generic points in scheme theory. As an application, we develop a six functor formalism for Zariski-constructible \'etale sheaves on characteristic zero rigid spaces. Among other things, this implies that characteristic zero rigid spaces support a well-behaved theory of perverse sheaves.
\end{abstract}

\maketitle

\setcounter{tocdepth}{2}
\tableofcontents

\section{Introduction}

In this paper, we prove a new generic smoothness result for morphisms of rigid analytic spaces (regarded as adic spaces, always), and apply it to set up a $6$ functor formalism for \'etale cohomology of rigid analytic spaces with coefficients in Zariski-constructible sheaves. Our first geometric result is the following (see also Remark~\ref{Ducrosrmk} for an alternative approach through \cite{Ducros}).

\begin{theorem}[Generic smoothness, Theorems \ref{genericflatness}, ~\ref{genericsmoothness} and \ref{genericsmoothnessproper}]
Fix a nonarchimedean field $K$, and let $f:X \to Y$ be a quasicompact map of rigid analytic spaces over $\Spa K$ with $Y$ reduced.

\begin{enumerate}\label{genericmaintheorem}
\item If $Y$ is geometrically reduced (which is automatic from reducedness if $K$ has characteristic $0$, \cite[Lemma 3.3.1]{ConradIrr}), there is a dense open subset $U \subset Y$ such that $f^{-1}(U) \to U$ is flat.
\item If $\mathrm{char}\,K=0$ and $X$ is smooth, there is a dense open subset $U \subset Y$ such that $f^{-1}(U) \to U$ is smooth. If moreover $f$ is proper, then the maximal such $U$ is Zariski-open. 
\end{enumerate}
\end{theorem}

In classical algebraic geometry, results like this are easily proved by spreading out from generic points in $Y$. In non-archimedean geometry, at least from the point of view of topology, there are far too many generic points: all rank one points of $Y$, and in particular all classical rigid points, are generic in the sense of locally spectral spaces. Moreover, at most of these points, spreading out cannot work naively, due to the subtle mixture of completions and integral closures which arise when computing the stalks and residue fields of $\mathcal{O}_Y$ and $\mathcal{O}_Y^+$. Our main new observation in the proof of Theorem~\ref{genericmaintheorem} is that there is nevertheless a reasonable rigid analytic analog of generic points from algebraic geometry, given as follows.

\begin{definition}[Weakly Shilov points, \S \ref{ss:Shilov}]

Fix a nonarchimedean field $K$ with a pseudouniformizer $t \in K^\circ$ and residue field $k$. A rank one point $x$ in a rigid space $X/K$ is \emph{weakly Shilov} if any one of the following equivalent conditions is satisfied:

\begin{enumerate}
\item There is an open affinoid subset $\Spa(A,A^\circ) \subset X$ such that $x$ lies in the Shilov boundary of $\Spa(A,A^\circ)$.

\item There is an open affinoid subset $\Spa(A,A^\circ) \subset X$ containing $x$ such that the map $A^\circ \to K_x^+$ identifies\footnote{As $x$ is a rank $1$ point, the ring $K_x^+$ is a rank $1$ valuation ring, and identifies with the subring $\mathcal{O}_{K_x} \subset K_x$ of power bounded elements of the valued field $K_x$.} $K_x^+$ with a $t$-completed localization of $A^\circ$.

\item The transcendence degree of the secondary residue field $K_x^+ / \mathfrak{m}$ over $k$ equals the local dimension of $X$ at $x$.

\item (Applicable only if $X$ is quasiseparated and quasi-paracompact.) There exists a formal model $\mathfrak{X}$ of $X$ such that the specialization map $\ssp: |X| \to |\mathfrak{X}_k|$ carries $x$ to the generic point of an irreducible component of $\mathfrak{X}_k$.

\end{enumerate}
\end{definition}

\begin{example}
If $X = \Spa K\! \left\langle T \right\rangle$ is the closed unit disc, the weakly Shilov points are exactly the points of Type 2 in the usual nomenclature, i.e. the points defined by the Gauss norms on closed subdisks. 
\end{example}

Weakly Shilov points are closely related to divisorial valuations as considered in birational geometry. For our purposes, the utility of these points arises by combining the characterizations (1) and (2) above: the former implies such points are dense in $X$, while the latter (roughly) makes these points amenable to the same commutative algebra arguments as generic points in algebraic geometry. The proof of Theorem \ref{genericmaintheorem} proceeds by making this idea precise: indeed, our arguments show that the subsets $U$ in Theorem \ref{genericmaintheorem} can be chosen to contain all the weakly Shilov points of $Y$.

\begin{remark}[Obtaining Theorem~\ref{genericmaintheorem} from Ducros' work]
\label{Ducrosrmk}
Theorem~\ref{genericmaintheorem} can also be deduced from Ducros' \cite{Ducros} if one switches to Berkovich spaces; we indicate the argument for Theorem~\ref{genericmaintheorem} (2) in the proper case (which is the most essential one for this paper) in Remark~\ref{rmk:DucrosProof}. Note that the idea of using Abhyankar points in \cite{Ducros} seems functionally equivalent to our idea of weakly Shilov points. On the other hand, our differential approach to a key step (see Theorem~\ref{keystep}) differs from the function theoretic approach of the corresponding \cite[Theorem 6.3.7]{Ducros}. We were unaware of \cite{Ducros} when working on Theorem~\ref{genericmaintheorem}, and thank Brian Conrad for bringing \cite{Ducros} to our attention.
\end{remark}

Let us now turn to the application of this result to \'etale cohomology of rigid spaces. Recall that for any rigid space $X/K$, work of Huber \cite{Hub96} and Berkovich \cite{BerEt} shows that the derived category $D(X,\mathbf{Z}/n)$ of \'etale $\mathbf{Z}/n$-sheaves admits a reasonable 6-functor formalism (at least for $n$ invertible on $K$). However, unlike in the case of schemes, it is much more subtle to isolate a reasonable subcategory of ``constructible'' complexes which are stable under the 6 operations.\footnote{For instance, constructible sheaves in the sense of Huber's work \cite{Hub96}, while having many wonderful categorical properties, do not capture the same geometric intuition as the corresponding notion in algebraic or complex geometry. Indeed, even the skyscraper sheaf at a classical point, perhaps the simplest example of a proper pushforward, is not Huber-constructible. Relatedly, analytifications of algebraically constructible sheaves on algebraic varieties are almost never Huber-constructible.}  In \cite{H}, the second author proposed that the following notion should yield the desired theory.

\begin{definition}[Zariski-constructible sheaves, Definition~\ref{zcdefinition}]
Fix a rigid space $X/K$ and $n > 0$. An \'etale sheaf $\mathscr{F}$ of $\mathbf{Z}/n$-modules is called \emph{Zariski-constructible} if $X$ admits a locally finite stratification $X = \coprod_{i \in I } X_i$ into Zariski locally closed subsets $X_i$ such that $\mathscr{F}|X_i$ is locally constant with finite stalks for all $i$. Write $D^{(b)}_{zc}(X,\mathbf{Z}/n) \subset D(X_{\et},\mathbf{Z}/n)$ for the full subcategory of the derived category of $\mathbf{Z}/n$-module sheaves on $X_{\et}$ spanned by complexes that have Zariski-constructible cohomology sheaves and are locally bounded on $X$. 
\end{definition}

The paper \cite{H} only showed the stability of $D_{zc}(X,\mathbf{Z}/n)$ by the 6 operations in some very limited situations.  The techniques in the present paper yield this stability in satisfactory generality over characteristic zero fields. The resulting formalism, which can be regarded as the rigid analytic analog of the classical theory of analytically constructible sheaves on complex analytic spaces (see, e.g., \cite[\S 2]{VerdierHomology}), is summarized as follows:

\begin{theorem}[The $6$ functor formalism for Zariski-constructible sheaves]
\label{ZCIntro}
Let $K$ be a characteristic zero nonarchimedean field of residue characteristic $p \geq 0$. For an integer $n \geq 1$, the assignment $X \mapsto D^{(b)}_{zc}(X,\mathbf{Z}/n)$ enjoys the following properties:
\begin{enumerate}
\item {\em Pullback}: For any map $f:X \to Y$, the pullback $f^*$ preserves $D^{(b)}_{zc}$ (Proposition~\ref{ZCStability}). 
\item {\em Proper pushforward:} For a proper map $f:X \to Y$, the pushforward $Rf_*$ preserves $D^{(b)}_{zc}$ (Theorem~\ref{properdirectimage}).
\item {\em More pushforwards:} For a Zariski-compactifiable map $f:X \to Y$, the pushforwards $Rf_*$ and $Rf_!$ carry lisse objects in $D^{(b)}_{zc}$ into $D^{(b)}_{zc}$ (Corollary~\ref{openpushforward}).
\item {\em $!$-pullback:} Given a map $f:X \to Y$, if either $f$ is finite or $(n,p) = 1$, the pullback $Rf^!$  preserves $D^{(b)}_{zc}$ (Corollary~\ref{finiteshriekpullback}).
\item {\em Verdier duality:} There is a natural dualizing complex $\omega_X \in D^{(b)}_{zc}(X,\mathbf{Z}/n)$ such that the functor $\mathbf{D}_X(-) := \inthom(-,\omega_X)$ induces an anti-equivalence on $D^{(b)}_{zc}$ satisfying biduality (Theorem~\ref{dualizingcomplex}).
\item {\em $\otimes$ and $\inthom$:} Given $\mathscr{F} \in D^{(b)}_{zc}(X,\mathbf{Z}/n)$ with locally finite Tor dimension (e.g., if $n$ is a prime), the functors $\mathscr{F} \otimes^L_{\mathbf{Z}/n} (-)$ and $\inthom(\mathscr{F},-)$ preserves $D^{(b)}_{zc}$ (Corollary~\ref{inthomexists}).
\end{enumerate}
Moreover, proper base change holds (Theorem~\ref{properbc}), and all of these operations are compatible with extensions of the nonarchimedean base field and with analytification of algebraic varieties (Proposition~\ref{schemescomparison2}, Theorem~\ref{dualizingcomplex}, Proposition~\ref{OperationsBC}).
Finally, all of these results admit extensions to $\mathbf{Z}_\ell$-coefficients (Theorem~\ref{operationsZell}).
\end{theorem}

Let us make a couple of remarks. First, we do not assume any conditions on $p$ relative to the coefficient ring (except in the case of $Rf^!$, but see Remark \ref{dualityproper}), so this result generalizes some previously known finiteness theorems in $p$-adic Hodge theory (and uses them as input). Secondly, due to the poor behaviour of Zariski closures in rigid geometry, it is unreasonable to expect arbitrary Zariski-constructible sheaves to be stable under pushforward (Warning~\ref{ZCWarning} (1)), so one cannot do much better than (3) above (although see Proposition \ref{morepushforwards}); similar issues also occur in complex geometry.

Next, we briefly comment on the proofs. Preservation under $f^\ast$ and $\otimes$ is straightforward and is stated for completeness. The first key new result is the preservation of Zariski-constructibility under $Rf_\ast$ for proper $f$. This was raised as a conjecture in \cite[Conjecture 1.14]{H}. Here we reduce it to the known statement that $Rf_\ast$ preserves locally constant constructible complexes when $f$ is both smooth and proper. This reduction relies on  Temkin's embedded resolution of singularities for quasi-excellent $\mathbf{Q}$-schemes, results of the second author \cite[Theorem 1.6]{H} on extending branched covers across Zariski-open immersions (building on previous work of Bartenwerfer \cite{Bar76} and L\"utkebohmert \cite{Lut93}), and (most crucially) Theorem \ref{genericmaintheorem}.(2).  

For the remaining stabilities in Theorem~\ref{ZCIntro}, we largely reduce them to analogous results for schemes (e.g., by replacing an affinoid $\mathrm{Spa}(A)$ with the scheme $\mathrm{Spec}(A)$). The classical results from SGA4 treat only finite type objects, and are not sufficient for our purposes. However, Gabber's work on \'etale cohomology of excellent schemes \cite{ILO14} is presented in exactly the right amount of generality, provided we are allowed to localize our questions to affinoids. The latter is possible thanks to our second key new result (which, in particular, settles \cite[Conjecture 1.12]{H}).

\begin{theorem}[Locality of Zariski-constructibility, Theorem~\ref{ZCLocal}]
 For abelian sheaves on characteristic zero rigid spaces, the property of being Zariski-constructible is an \'etale-local property.
\end{theorem}

Using this toolkit, one can imitate many standard constructions with constructible sheaves found in complex or algebraic geometry. As an example, we show that characteristic zero rigid spaces support a theory of perverse sheaves which has the same pleasant formal properties as its algebraic counterpart \cite{BBDG} (except that we need to restrict to qcqs spaces when working with $\mathbf{Q}_\ell$-coefficients).

\begin{theorem}[Perverse sheaves in rigid geometry, Theorem~\ref{perverseproperties} and Theorem~\ref{PervQell}]
Let $K$ be a characteristic zero nonarchimedean field with residue characteristic $p \geq 0$ and let $X/K$ be a rigid space. Fix a prime $\ell$ and a coefficient ring $\Lambda \in \{ \mathbf{Z}/\ell^n\mathbf{Z}, \mathbf{Q}_{\ell} \}$. If $\Lambda=\mathbf{Q}_\ell$, then assume that $X$ is qcqs and define $D^{(b)}_{zc}(X,\mathbf{Q}_\ell) := D^{(b)}_{zc}(X,\mathbf{Z}_\ell) \otimes_{\mathbf{Z}_\ell} \mathbf{Q}_\ell$.

There is a naturally defined perverse $t$-structure on $D^{(b)}_{zc}(X,\Lambda)$, with abelian heart denoted $\Perv(X,\Lambda)$. This construction has the following stability properties:
\begin{enumerate}
\item {\em Duality:} $\mathrm{Perv}(X,\Lambda)$ is Verdier self-dual inside $D^{(b)}_{zc}(X,\Lambda)$.

\item {\em Finite pushforward:} $\mathrm{Perv}(-,\Lambda)$ is stable under $f_*$ for $f:Y \to X$ a finite map.

\item {\em Intermediate extensions:} There is a notion of intermediate extension of lisse sheaves defined on Zariski-locally closed subsets of $X$.

\item {\em Finiteness:} If $X$ is quasi-compact, then $\Perv(X,\Lambda)$ is noetherian and artinian.

\item {\em Nearby cycles:}  The nearby cycles functor associated with a formal model of $X$ is perverse $t$-exact when $p \neq \ell$. 

\end{enumerate}
Moreover, this construction is compatible with the usual constructions in algebraic geometry under analytification (in the proper case for $\Lambda=\mathbf{Q}_\ell$), and is compatible with extensions of the ground field.
\end{theorem}

As an application, we can define an intersection cohomology complex on any qcqs characteristic $0$ rigid space, and the resulting intersection cohomology groups have reasonable properties:

\begin{corollary}[Intersection cohomology of rigid spaces, Theorem~\ref{ICProperties}]
Let $K$ be a characteristic zero nonarchimedean field of residue characteristic $p \geq 0$; let $C/K$ be a completed algebraic closure and let $\ell$ be any prime. Let $X/K$ be a qcqs rigid space $X/K$.

\begin{enumerate}
\item {\em Existence of intersection cohomology}: There are naturally defined $\ell$-adic intersection cohomology groups $IH^{n}(X_{C},\mathbf{Q}_{\ell})$.  These are finitely generated $\mathbf{Q}_{\ell}$-modules if $\ell \neq p$ or if $X$ is proper.

\item {\em GAGA}: If $X = \mathcal{X}^{an}$ for a proper $K$-scheme $\mathcal{X}$, then $IH^*(X_C,\mathbf{Q}_\ell) \simeq IH^*(\mathcal{X}_C,\mathbf{Q}_\ell)$.

\item {\em Poincar\'e duality}: If $X$ is proper and equidimensional of dimension $d$ and $\ell \neq p$, there is a natural Poincar\'e duality isomorphism \[ IH^{n}(X_C,\mathbf{Q}_{\ell})^{\ast} \cong IH^{-n}(X_C,\mathbf{Q}_{\ell})(d). \]
\end{enumerate}
\end{corollary}

We end this paper by formulating some conjectures in \S \ref{ss:conj}, roughly predicting that deep known results on the intersection cohomology of algebraic varieties over $K$ carry forth to the rigid context.

\subsection*{Conventions}

We follow the convention that the term ``nonarchimedean field'' is reserved for valued fields carrying a rank $1$ valuation.

If $K$ is a nonarchimedean field, we use the terms $K$-affinoid algebra and topologically of finite type (henceforth abbreviated tft) $K$-algebra synonymously; recall that these are exactly the Banach $K$-algebras that can receive a continuous surjection from a Tate algebra $K\langle x_1,...,x_n \rangle$. Likewise, we say ``rigid space over $K$'' and ``adic space locally of tft over $\Spa K$'' interchangeably. If $A$ is a tft $K$-algebra, we write $\tilde{A} := A^\circ / A^{\circ \circ} = (A^\circ / t)^{red}$; this is a $k$-algebra of finite type \cite[Corollary 3, \S 6.3.4]{BGR}, where $k$ is the residue field of $K$ and $t$ is a pseudouniformizer (i.e., any nonzero element of $K^{\circ \circ} - \{0\}$).

We warn the reader that $A^\circ$ need not be tfp over $\mathcal{O}_K$ even when $A$ is reduced; this pathology does not occur if either $K$ is discrete, or if $K$ is stable field with $|K^*|$ divisible (if $K$ is algebraically closed); see \cite[\S 3.6]{BGR} for the definition of stability, and \cite[\S 6.4]{BGR} for the finiteness properties.

We write $\Spa A = \Spa(A,A^\circ)$ for any Huber ring $A$.

Say $X$ is an analytic adic space and $x \in X$. We write $K_x$ for the completed residue field of $X$ at $x$. This is a nonarchimedean field. We shall write $|\cdot|_x$ for the associated valuation on functions, though note that is only well-defined up to equivalence. However, if $A$ is a tft $K$-algebra and $x \in \Spa A$ is a rank one point, there is a unique $\mathbf{R}_{\geq 0}$-valued representative of the associated equivalence class which extends the fixed norm on the base field $K$, i.e. a unique representative such that $|t|_x = |t|_K$. We always choose this representative.

Given $X$ and $x$ as above, the secondary residue field of $X$ at a point $x$ is the quotient $ \tilde{K}_x = K_x^{+} / \mathfrak{m}_x$. Here $K_x^+$ is the valuation subring of the residue field $K_x$ defined as the completed image of the map $\mathcal{O}_{X,x}^+ \to K_x$; when $x$ is a rank $1$ point, the ring $K_x^+$ coincides with the subring $\mathcal{O}_{K_x} \subset K_x$ of power bounded elements. Recall that if $y \prec x$ is any specialization, then $K_x \cong K_y$ and $K_y^+ \subset K_x^+$ under this identification. 

For any tft $K$-algebra $A$, we shall write $\ssp: \Spa A \to \Spec \tilde{A}$ for the specialization map, given by taking the center of the valuation. This is a continuous, closed, and spectral map of spectral spaces. 

If $f:X \to Y$ is a map of rigid spaces over $K$, we say $f$ is \emph{Zariski-compactifiable} if it admits a factorization $f=\overline{f} \circ j$, where $j:X \to X'$ is a Zariski-open immmersion and $\overline{f}:X' \to Y$ is a proper morphism. We say a rigid space $X$ over $K$ is Zariski-compactifiable if the structure map $f:X \to \Spa K$ is so.

\subsection*{Acknowledgements} In April 2020, Bogdan Zavyalov asked DH whether \cite[Conjecture 1.14]{H} was within reach, and the present paper grew directly out of that conversation. We thank Bogdan heartily for this crucial initial stimulus, and for some helpful comments on previous drafts of this paper. We are also grateful to Brian Conrad, Johan de Jong, Haoyang Guo, Mattias Jonsson, Shizhang Li, and Jacob Lurie for useful conversations and exchanges, and to the anonymous referees for a number of comments that helped improve the exposition. Bhatt was partially supported by the NSF (\#1801689, \#1952399, \#1840234), a Packard fellowship, and the Simons Foundation (\#622511).

\section{Generic flatness and generic smoothness}
\label{sec:gensmooth}

In this section, we introduce the notion of weakly Shilov points, and  prove our main geometric results on generic smoothness.

\subsection{Shilov and weakly Shilov points}
\label{ss:Shilov}

Fix a complete nonarchimedean field $K$ and a pseudouniformizer $t \in K^\circ$. Recall the following basic example of a ``non-classical'' point of a standard $K$-affinoid.

\begin{example} 
\label{GaussDisc}
If $X=\Spa K\! \left \langle T \right \rangle $ is the one-dimensional affinoid ball over $K$, then the Gauss point $X$ is given by the $t$-adic norm on $K \langle T \rangle$. To describe this norm ring theoretically, recall that the standard formal for $X$ is given by the formal closed disc $\mathfrak{X} := \mathrm{Spf}(\mathcal{O}_K \langle T \rangle)$. The special fibre $\mathfrak{X}_s = \mathrm{Spec}(\mathcal{O}_K/t[T])$ has a unique a generic point $\overline{\eta}$; the $t$-completed localization of $\mathcal{O}_K \langle T \rangle$ at $\overline{\eta}$ is a rank one $t$-complete and $t$-torsionfree valuation ring $V$ equipped with a map $\mathcal{O}_K \langle T \rangle \to V$. The resulting valuation on $K \langle T \rangle$ is the Gauss point $\eta$. Moreover, the canonical map $X \to \mathfrak{X}$ given by taking the center of the valuation carries $\eta$ to $\overline{\eta}$.
\end{example}

We now isolate a general class of points with properties similar to the Gauss point from Example~\ref{GaussDisc}.

\begin{proposition} 
\label{ShilovAffd}
Fix a tft $K$-algebra $A$ and a rank one point $x \in \mathrm{Spa}(A,A^\circ)$. The following are equivalent:
\begin{enumerate}

\item $\ssp(x)$ is the generic point of an irreducible component of $\Spec \tilde{A}$.

\item $\{ x \} = \ssp^{-1}(\ssp(x))$ as subsets of $\Spa A$.

\item $K_{x}^{+}$ is a $t$-completed ind-Zariski localization of $A^\circ$. More precisely, $K_{x}^{+}$ is the $t$-completed local ring of $\Spec A^\circ$ at the point $\ssp(x) \in \Spec \tilde{A} \subset \Spec A^\circ$.

\item The seminorm $|\cdot |_x$ belongs to the Shilov boundary of $A$ in the sense of $K$-Banach algebras.

\end{enumerate}

\end{proposition}

\begin{proof} 
We begin with two (well-known) observations on the affinoid adic space $\mathrm{Spa}(A,A^\circ)$; the first concerns finding suitable rings of definition, while the second concerns a description via formal models.

First, we claim that there exists an open and bounded tfp $\mathcal{O}_K$-subalgebra $A_0 \subset A^0$ such that $\mathrm{Spec}(\tilde{A}) \to \mathrm{Spec}(A_0/tA_0)$ is a universal homeomorphism and such that $A^\circ$ is the integral closure of $A_0$ in $A$. In fact, the second part follows from \cite[Remark following 6.3.4/1]{BGR}. For the first, using Noether normalization, we can choose a surjective map $T \to A$, where $T  = K \langle x_1,...,x_n \rangle$ is a Tate algebra. By \cite[6.3.4/2]{BGR}, the map $\tilde{T} \to \tilde{A}$ is module finite. Setting $A'_0 = \mathrm{im}(T^\circ \to A^\circ)$, we learn that $A_0'/(A^{\circ\circ} \cap A_0') \to \tilde{A}$ is module finite and injective, with $A_0'$ being open, bounded, and tfp over $\mathcal{O}_K$. Enlarging $A_0'$ inside $A^\circ$ by adding finitely topological generators of $\tilde{A}$ over $A_0'$, we find an open bounded tfp $\mathcal{O}_K$-subalgebra $A_0 \subset A^\circ$ such that $A_0/(A^{\circ\circ} \cap A_0) \to \tilde{A}$ is bijective. But we also know that $\sqrt{tA_0} = A^{\circ\circ} \cap A_0$: any element of the right side is topologically nilpotent and in $A_0$, and must thus have a large enough power inside $tA_0$. Thus, we have found the subalgebra $A_0$ indicated at the start of this paragraph. 

Next, we also recall an alternative description of the locally ringed space $(\mathrm{Spa}(A,A^\circ), \mathcal{O}^+)$. Consider the category of all proper maps $f_i:X_i \to \mathrm{Spec}(A_0)$ of schemes which are isomorphisms after inverting $t$. For each $f_i$, let $X_{i,t=0} \subset X_i$ be the special fibre (regarded merely as a closed subset). Set $Z = \lim_i X_{i,t=0}$, so $Z$ is a spectral space.  Let $\pi_i:Z \to X_i$ be the structure map, and define the structure sheaf $\mathcal{O}_Z$ of $Z$ via $\mathcal{O}_Z := \colim_i \pi_i^{-1} \mathcal{O}_{X_i}$. Then it is a basic fact that $\mathrm{Spa}(A,A^\circ) = Z$ as topological spaces, and $\mathcal{O}^+$ identifies with the $t$-adic completion of $\mathcal{O}_Z$. For future use, we remark that, by passing to a cofinal subsystem, we may (and do) assume that each $X_i$ is $\mathcal{O}_K$-flat, i.e., $\mathcal{O}_{X_i}$ is $t$-torsionfree. This condition implies that the generic fibre $X_i[1/t]$ is dense in each $X_i$, and thus all the transition maps $X_i \to X_j$ in the system are surjective: their image is a closed set containing a dense open. In particular, $\mathcal{O}_Z$ is $t$-torsionfree. By spectrality, the maps $\pi_i:Z \to X_{i,t=0}$ are also surjective  for all $i$. Finally, we also remark that $\mathcal{O}_Z$ is integrally closed in $\mathcal{O}_Z[1/t]$ by generalities on blowups.

Using the preceding two observations, we prove the equivalences. 

$(1) \Rightarrow (2)$: Fix $x \in \mathrm{Spa}(A)$ with $\ssp(x) \in \Spec \tilde{A}$ being a generic point. Using a suitable Noether normalization, we then learn that $\tilde{A}_{\ssp(x)}$ is a rank $1$ valuation ring with pseudouniformizer $t$: this ring is the integral closure of a rank $1$ $t$-complete valuation ring in a finite extension of its fraction field. Any point $y \in \ssp^{-1}(\ssp(x))$ is represented by an equivalence class of maps $A_0 \to V$ where $V$ is a $t$-complete $t$-torsionfree valuation ring with the property that the closed point of $\mathrm{Spec}(V)$ is carried to $\ssp(x)$. But any such map factors uniquely as $A_0 \to \tilde{A}_{\ssp(x)} \to V$. Thus, taking $V$ to be the $t$-completion of $\tilde{A}_{\ssp(x)}$ gives the unique such map up to equivalence, showing that $y=x$. 

$(2) \Rightarrow (3)$: Write $x_i \in X_i$ for the image of $x$, so $\mathcal{O}^{+}_{\mathrm{Spa}(A,A^+),x}$ and $\colim \mathcal{O}_{X_i,x_i}$ identify after $t$-completion. Now if $x = \ssp^{-1}(\ssp(x))$, then $x_i = f_i^{-1}(\ssp(x))$ is the unique preimage of $\ssp(x)$. The map $A_{0,\ssp(x)} \to \mathcal{O}_{X_i,x_i}$ is then integral (by properness of $X_i \to \mathrm{Spec}(A_0)$, base changed to $\mathrm{Spec}(A_{0,\ssp(x)}$) and an isomorphism after inverting $t$. Taking a colimit, we learn that $A_{0,\ssp(x)} \to \colim_i \mathcal{O}_{X_i,x_i}$ is integral and an isomorphism after inverting $t$. But the target is also integrally closed in its $t$-localization, so it must coincide with $A^{\circ}_{\ssp(x)}$; here we implicitly use that $\mathrm{Spec}(\tilde{A}) \simeq \mathrm{Spec}(A_0/t)$ to identify $\ssp(x)$ with a point of $\mathrm{Spec}(A^\circ)$, as well as the fact that $A^\circ$ is the integral closure of $A_0$. Thus, we have shown that the $t$-completion of $A^{\circ}_{\ssp(x)} \to \mathcal{O}^{+}_{\mathrm{Spa}(A,A^+),x}$ is an isomorphism. As the $t$-completion of the target is $K_x^+$, the claim follows. 

$(3) \Rightarrow (2)$: This is clear from the description of points of $\mathrm{Spa}(A,A^\circ)$ as equivalence classes of maps $A^\circ \to V$ to $t$-complete and $t$-torsionfree valuation rings. 

$(2) \Rightarrow (1)$: Assume that $x$ is the unique preimage of $\ssp(x)$ and yet that $\ssp(x) \in \mathrm{Spec}(\tilde{A})$ is not a generic point; we shall obtain a contradiction. Let $\overline{y} \in \mathrm{Spec}(\tilde{A})$ be a generic point of an irreducible component $Y \subset \mathrm{Spec}(\tilde{A})$ containing $\ssp(x)$. Applying the implications $(1) \Rightarrow (3)$ to a lift $y \in \mathrm{Spa}(A,A^+)$ of $\overline{y}$, we learn that $\overline{y}$ has a unique lift $y \in \mathrm{Spa}(A,A^+)$, and that $K_{y}^+$ is the $t$-completed local ring of $A^{\circ}_y$. In particular, the residue field of $K_{y}^+$ identifies with function field $K(Y)$ of the irreducible component $Y$. Choose a valuation subring $\overline{V} \subset K(Y)$ that is a $\tilde{A}$-algebra and has center $\ssp(x) \in \mathrm{Spec}(\tilde{A})$; this valuation has rank $\geq 1$ as $\overline{y} \neq \ssp(x)$.  Let $V \subset K_{y}^+$ be the preimage of $\overline{V}$. Then $V$ is a $t$-complete valuation ring of rank $\geq 2$, is an $A^\circ$-algebra, and has center $\ssp(x)$ on $\mathrm{Spec}(\tilde{A})$. The resulting map $A^\circ \to V$ then gives a point $x' \in \mathrm{Spa}(A,A^\circ)$ with rank $\geq 2$ and image $\ssp(x)$ in $\mathrm{Spec}(\tilde{A})$. But $x$ is the unique preimage of $\ssp(x)$, so $x = x'$. We now obtain a contradiction as $x$ had rank $1$ by assumption, while $x'$ has rank $\geq 2$ by construction. 

$(1) \Longleftrightarrow (4)$: This is proven in \cite[2.4.4]{BerSpectral}.
\end{proof}

\begin{remark}
One may ask if the equivalences proven in Proposition~\ref{ShilovAffd} continue to hold true for the affinoid adic space $\mathrm{Spa}(A,A^+)$ attached to any complete Tate ring $(A,A^+)$. Inspection of the proof shows that the equivalence of $(2)$ and $(3)$ and the implication $(2) \Rightarrow (1)$ hold true in general.  On the other hand, there is a perfectoid affinoid algebra where $(1) \Rightarrow (2)$ and $(1) \Rightarrow (3)$ in Proposition~\ref{ShilovAffd} fail, as we explain next. We are not aware of a broader class of algebras (than the $K$-affinoid ones) where the Proposition~\ref{ShilovAffd} holds true.

Take a complete and algebraically closed extension $C/\mathbf{Q}_p$ whose algebraically closed residue field $k$ has transcendence degree $\geq 1$ over $\mathbf{F}_p$. Write $V = \mathcal{O}_C$, and let $W \subset V$ be the preimage of $\overline{\mathbf{F}_p} \subset k$. Then $W$ is a $p$-complete and $p$-torsionfree local ring that is integrally closed in $W[1/p] = V[1/p]$. Moreover, we have $\sqrt{pW} = \sqrt{pV}$ with $W/\sqrt{pW} \to V/\sqrt{pV}$ identifying with the map $\overline{\mathbf{F}_p} \subset k$. A lift $x \in V \subset W[1/p]$ of any element of $k - \overline{\mathbf{F}_p}$ has the property that $x,x^{-1} \notin W$, so $W$ is not a valuation ring.  Endowing $W$ with the $p$-adic topology, we obtain a complete uniform Tate ring $(W[1/p],W)$ with $W$ perfectoid. The special fibre $\widetilde{W[1/p]} = W/\sqrt{pW}$ identifies with $\overline{\mathbf{F}_p}$, so $\Spec(\widetilde{W[1/p]})$ has a unique point which is thus a generic point. On the other hand, the local ring of $\mathrm{Spec}(W)$ at this point is simply $W$, which is not a valuation ring. This shows that the implication $(1) \Rightarrow (3)$ in Proposition~\ref{ShilovAffd} fails in this example.  To show that $(1) \Rightarrow (2)$ also fails in this example, one calculates that $\mathrm{Spa}(W[1/p],W)$ identifies with the Riemann-Zariski space of $k$, which has more than $1$ element; we omit the argument. 
\end{remark}

\begin{definition}Let $A$ be a tft $K$-algebra. A \emph{Shilov point} $x \in \Spa A$ is any point satisfying the equivalent conditions of Proposition~\ref{ShilovAffd}.
\end{definition}

Recall that in an analytic adic space, every rank one point is generic in the sense of spectral spaces. For algebraic purposes, the following class of points is more relevant.

\begin{definition}
\label{def:WeaklyShilov}
Let $X$ be any rigid space over a non-archimedean field $K$. A point $x \in X$ is \emph{weakly Shilov} if there is an open affinoid subset $ x \in W \subset X$ such that $x$ is a Shilov point of $W$. In particular, such a point has rank one.
\end{definition}

\begin{example}
For $X = \mathrm{Spa}(K \langle T \rangle)$ the closed unit disc, the weakly Shilov points are exactly the points of type $2$ (see \cite[Example 2.20]{ScholzePerfectoid} for the classification of points on $X$). This follows from the characterization in Proposition~\ref{WeakShilovChar}, since the secondary residue field of $x$ is algebraic over $K$ when $x$ has type $1$, $3$ or $4$ (and $x$'s of type 5 have rank $2$, so they are not Shilov). 
\end{example}

\begin{remark}
The valuations isolated in Definition~\ref{def:WeaklyShilov} are sometimes also called {\em divisorial valuations} in birational geometry (see Proposition~\ref{WeakShilovChar} (3)). 
\end{remark}

\begin{lemma} Let $X$ be a  rigid space over a nonarchimedean field $K$. Then weakly Shilov points are dense in $X$.
\end{lemma}
\begin{proof} By definition, any open affinoid $U \subset X$ contains a weakly Shilov point.
\end{proof}

We also have the following alternative characterization of weakly Shilov points.

\begin{proposition}
\label{WeakShilovChar}
Let $X$ be a rigid space over a nonarchimedean field $K$ and let $x \in X$ be a rank one point. The following are equivalent:
\begin{enumerate}
\item $x$ is weakly Shilov.
\item The transcendence degree of the secondary residue field $\tilde{K}_x$ over $K^\circ / \mathfrak{m}$ equals the local dimension $\dim_{x} X$ of $X$ at $x$.
\item (Applicable only when $X$ is quasiseparated and quasi-paracompact) There exists a formal model $\mathfrak{X}$ such that $x$ maps to the generic point of an irreducible component of $\mathfrak{X}_s$. 
\end{enumerate}
\end{proposition}
\begin{proof}
$(1) \Longleftrightarrow (2)$ follows by combining Lemme 4.4 and Corollaire 4.15 in \cite{P}.

$(3) \Rightarrow (1)$: if there exists a formal model $\mathfrak{X}$ as in (3), then taking $W \subset X$ be the preimage of any formal affine open $\mathfrak{X}$ containing the image of $x$ shows that $x$ is weakly Shilov.

$(1) \Rightarrow (3)$: Assume $x \in X$ is weakly Shilov. We can then find an affinoid open $\mathrm{Spa}(R) \subset X$ containing $x$ such that $x \in \mathrm{Spa}(R)$ is Shilov. By Raynaud, there is a formal model $\mathfrak{X}$ of $X$ such that $\mathrm{Spa}(R) \subset X$ is the preimage of a formal affine open $\mathrm{Spf}(R_0) \subset \mathfrak{X}$, cf. \cite[Theorem 8.4.3]{Bosch}. By passing to a refinement, we can assume that the map $R_0/t \to \tilde{R}$ gives a homeomorphism of spectra; see first paragraph of the proof of Proposition~\ref{ShilovAffd}. As $x$ is weakly Shilov, its image in $\mathrm{Spec}(\tilde{R})$ is a generic point. But then its image in $\mathrm{Spec}(R_0/t)$ is also a generic point since $\mathrm{Spec}(\tilde{R}) \to \mathrm{Spec}(R_0/t)$ is a homeomorphism by construction. As $\mathrm{Spf}(R_0) \subset \mathfrak{X}$ is a formal open immersion, it follows that $x$ gives a generic point of $\mathfrak{X}_s$ as well
\end{proof}

\begin{corollary}
\label{ShilovReg}
Say $X$ is a rigid space over a nonarchimedean field $K$, and $Z \subset X$ is a nowhere dense Zariski closed set. Then $Z$ does not contain any weakly Shilov point of $X$. In particular, if $X$ is reduced, then any weakly Shilov point of $X$ lies in the regular locus of $X$.
\end{corollary}
\begin{proof}
As $Z \subset X$ is a nowhere dense Zariski closed set, we have $\dim_x(Z) < \dim_x(X)$ for all $x \in X$. The first statement now follows from the characterization of weak Shilov points in Proposition~\ref{WeakShilovChar} (2). The second statement follows from the first statement applied to the locus $Z \subset X$ of points that are not regular, which is a nowhere dense Zariski closed set by excellence considerations.
\end{proof}

\subsection{Tools involving the cotangent complex}

To prove our generic smoothness result, it will be convenient to use the analytic cotangent complex as this provides a homologically well-behaved object detecting smoothness in non-noetherian situations (such as topologically finitely presented algebras over non-discrete valuation rings). In this subsection, we recall some results on this object.

\begin{notation}
Fix a complete nonarchimedean field $K$ with valuation ring $V \subset K$ and a pseudouniformizer $t \in V$.  For any map $A \to B$ of $V$-algebras, write $L^{an}_{B/A}$ for the derived $t$-completion of $L_{B/A}$. A {\em tft (or topologically finite type)} $V$-algebra is a $V$-algebra $A$ of the form $V[x_1,...,x_n]^{\wedge}/I$ (where the completion is $t$-adic). If moreover $I$ is finitely generated, we say that $A$ is {\em tfp (or topologically finitely presented)}. The class of tfp $V$-algebras has good properties such as coherence and classical $t$-adic completenesss, see \cite[Proposition 7.1.1]{GR}. Moreover, any tft $V$-algebra which is $t$-torsion-free is in fact tfp, see \cite[Corollary 7.3.6]{FGK}.
\end{notation}

\begin{theorem}[Gabber-Ramero]
\label{GRCC}
Say $A \to B$ is a map of tfp $V$-algebras. Then the following hold true:
\begin{enumerate}
\item $L^{an}_{B/A}$ is a pseudocoherent $B$-complex.

\item Assume that $A \to B$ induces a smooth map of relative dimension $n$ on taking adic generic fibres. Then $L^{an}_{B/A}[1/t]$ is a finite projective $B[1/t]$-module of rank $n$.

\item If $A \to B$ is surjective, then $L_{B/A} \simeq L^{an}_{B/A}$. 

\end{enumerate}
 \end{theorem}
 \begin{proof}
 (1) is \cite[Theorem 7.1.31]{GR}.
 
(2) is the key assertion checked in the proof of \cite[Theorem 7.2.39]{GR}.

 (3) is \cite[Theorem 7.1.29]{GR}.

\end{proof}

\begin{lemma}
\label{CCValFinite}
Let $R$ be a finitely presented flat $V$-algebra. 
\begin{enumerate}
\item $L_{R/V}$ is a pseudocoherent $R$-complex.
\item If $R$ is also $V$-finite, then $L_{R/V}$ is derived $t$-complete. 
\end{enumerate}
\end{lemma}
\begin{proof}
For (1), by noetherian approximation, we can write $V \to R$ as the base change of a finitely presented flat map $V_0 \to R_0$ of finitely generated $\mathbf{Z}$-algebras along some map $V_0 \to V$. Then $L_{R_0/V_0}$ is pseudocoherent, and $L_{R_0/V_0} \otimes_{R_0}^L R \simeq L_{R/V}$ by Tor independent base change, so $L_{R/V}$ is also pseudocoherent.  

For (2), observe that if $R$ is $V$-finite, then $R$ is a finite free $V$-module (as any finitely presented flat $V$-module is finite free). In particular, a pseudocoherent $R$-complex is also pseudocoherent as a $V$-complex. The claim now follows as any pseudocoherent $V$-complex is derived $t$-complete (since $V$ itself is so).
\end{proof}

\begin{corollary}
\label{NACC}
Let $E/K$ be a finite extension, and let $W_0 \subset E^\circ$ be an open tfp $V$-subalgebra.  
\begin{enumerate}
\item $L_{W_0/V}$ is pseudocoherent so $L_{W_0/V} \simeq L^{an}_{W_0/V}$, whence $L_{E/K} \simeq L^{an}_{W_0/V}[1/t]$.
\item $H_i(L^{an}_{W_0/V}[1/t]) = 0$ for $i \neq 0,1$. 
\end{enumerate}
\end{corollary}
\begin{proof}
For (1):  since $W_0 \subset E^\circ$, choosing monic equations defining generators of $W_0$ shows that any such $W_0$ is a finitely presented flat $V$-algebra. Lemma~\ref{CCValFinite} then implies that $L_{W_0/V}$ is pseudocoherent and thus already derived $t$-complete, so $L_{W_0/V} \simeq L^{an}_{W_0/V}$. The last part of (1) follows by inverting $t$ and noting that formation of cotangent complexes commutes with localization.

(2) follows from (1) and a general fact: the cotangent complex of {\em any} extension of fields only has homology in degrees $1$ and $0$. Indeed, this follows from transitivity triangles and the fact that any field is ind-smooth over its prime subfield (which is perfect) by generic smoothness in algebraic geometry.
\end{proof}

We need the following result later.

\begin{theorem}[Quillen]
\label{QuillenRegCrit}
Fix a noetherian ring $A$ and a maximal ideal $\mathfrak{m}$ with residue field $k = A/\mathfrak{m}$. If $L_{k/A}$ is concentrated in degree $-1$, then $A$ is regular at $\mathfrak{m}$. 
\end{theorem}
\begin{proof}
\cite[Corollary 10.5]{QuillenCCNotes} shows that $\mathfrak{m}$ is generated by a regular sequence if $H_2(L_{k/A}) = 0$. As the maximal ideal of a noetherian local ring is generated by a regular sequence exactly when the ring is regular, the claim follows. 
\end{proof}

The following lemma will be useful later as well.

\begin{lemma}
\label{noethlemma}
Say $A$ is a reduced Jacobson noetherian ring. Let $M \to N$ be a map of finitely generated $A$-modules such that $M/\mathfrak{m}M \to N/\mathfrak{m}N$ is injective for every maximal ideal $\mathfrak{m}$. If $M$ is a locally free $A$-module, then $M \to N$ is injective.
\end{lemma}
\begin{proof}
The question is local, so we can assume $M = A^{\oplus n}$ is a free module. Let $x = (x_1,...,x_n) \in M$ lie in the kernel of $M \to N$. The hypothesis implies that $x_i \in \mathfrak{m}$ for every maximal ideal $\mathfrak{m} \subset A$. But the intersection of all maximal ideals of $A$ is its nilradical (as $A$ is Jacobson) which is $0$ (as $A$ is reduced). So $x = 0$, as wanted.
\end{proof}

 \subsection{Regularity of Shilov fibers}
 
 In this subsection, we establish the key technical ingredient behind our generic smoothness result, using the analytic cotangent complex.
 
 \begin{notation}
 Fix a complete nonarchimedean field $C$ with valuation ring $\mathcal{O}_C \subset C$ and a pseudouniformizer $t \in \mathcal{O}_C$.\footnote{Contrary to modern notational conventions, we are not assuming $C$ is algebraically closed. We hope this causes no confusion.}
 \end{notation}

Our goal is the following.

\begin{theorem}
\label{keystep}
Let $T \to R$ be a map of tfp $\mathcal{O}_C$-algebras. Assume $\mathrm{Spa}(R[1/t])$ is smooth over $C$. Let $V$ denote the $t$-completed Zariski local ring of $T$ at a generic point $\eta \in \mathrm{Spec}(T/t) \subset \mathrm{Spec}(T)$, and set $R_V = R \widehat{\otimes}_T V$. Then $R_V[1/t]$ is regular.
\end{theorem}
\begin{proof}

Our strategy is to first simplify $T$, and then argue that $R_V[1/t]$ is regular using the cotangent complex.

First, observe that the statement of the theorem is Zariski local around $\eta \in \mathrm{Spf}(T)$, so we may shrink $\mathrm{Spf}(T)$ around $\eta$ to assume $\mathrm{Spec}(T/t)$ is irreducible with generic point $\eta$. If $T' = \mathcal{O}_C \langle x_1,...,x_n \rangle \to T$ is a finite injective map, then $\mathrm{Spec}(T/t) \to \mathrm{Spec}(T'/t)$ is finite and surjective. As both these schemes are irreducible, the image of the generic point $\eta \in \mathrm{Spec}(T/t)$ is the generic point $\eta' \in \mathrm{Spec}(T'/t)$, and $\eta$ is the unique preimage of $\eta'$. In particular, we must have $T \widehat{\otimes}_{T'} V' \simeq V$ by base change. But then $R \widehat{\otimes}_{T'} V' \simeq R_V$. Thus, we may replace $T$ with $T'$ to assume that $T$ is a standard Tate algebra over $\mathcal{O}_C$.  In particular, $T$ is formally smooth over $\mathcal{O}_C$. In this case, we have $T = T[1/t]^{\circ}$, so $V$ is actually a $t$-complete and $t$-torsionfree rank one valuation ring (either by explicit calculation, or as explained in Proposition~\ref{ShilovAffd}).

Next, we describe $L^{an}_{V/\mathcal{O}_C}$. By definition, the ring $V$ is a $t$-completed ind-Zariski localization of $T$. As the formation of the $t$-completed cotangent complex is compatible with $t$-completed localizations and $t$-completed filtered colimits, we learn that $L^{an}_{V/\mathcal{O}_C} \simeq L^{an}_{T/\mathcal{O}_C} \widehat{\otimes}_T^L V$. But $T/\mathcal{O}_C$ is formally smooth, so $L^{an}_{T/\mathcal{O}_C}$ is a finite projective $T$-module placed in degree $0$, whence  $L^{an}_{V/\mathcal{O}_C}$ is a finite free $V$-module placed in degree $0$. 

We now begin proving the theorem. Since $V$ is a $t$-complete and $t$-torsionfree rank $1$ valuation ring, the ring $K := V[1/t]$ is a nonarchimedean field extension of $C$ with $K^\circ  = V$. As $R_V$ is a tfp $V$-algebra, we can (and will) regard $R_V[1/t]$ as an affinoid $K$-algebra. To show regularity of $R_V[1/t]$, it is enough to show that the local rings of $R_V[1/t]$ at all closed points are regular: the regular locus in $R_V[1/t]$ is open (as affinoid $K$-algebras are noetherian and excellent) and the maximal ideals are dense (as affinoid $K$-algebras are Jacobson). A closed point is given by a quotient $R_V[1/t] \to E$ where $E/K$ is a finite extension; fix such a point. By Theorem~\ref{QuillenRegCrit}, it suffices to show that $L_{E/R_V[1/t]}$ has homology only in degrees $0,1$. (In fact, there is no $H_0$ as $R_V[1/t] \to E$ is surjective, but it will be convenient to formulate things this way.) Let $W_0 \subset E^\circ$ be the image of $R_V$ under this map. As $E^\circ$ is the integral closure of $V$ in $E$, the ring $W_0$ is a finitely presented finite flat $V$-algebra. Since the map $R_V \to W_0$ is a surjection of tfp $V$-algebras, we have $L_{W_0/R_V} \simeq L^{an}_{W_0/R_V}$ by Theorem~\ref{GRCC} (3), and hence $L_{E/R_V[1/t]} \simeq L^{an}_{W_0/R_V}[1/t]$ by inverting $t$, so it is enough to show that $L^{an}_{W_0/R_V}[1/t]$ has homology only in degrees $0,1$. Consider the transitivity triangle for $\mathcal{O}_C \to R_V \to W_0$:
\begin{equation}
\label{cc0}
 L^{an}_{R_V/\mathcal{O}_C} \widehat{\otimes}_{R_V}^L W_0 \to L^{an}_{W_0/\mathcal{O}_C} \to L^{an}_{W_0/R_V}.
 \end{equation}
As $R \to R_V$ is a $t$-completed Zariski localization, we have 
\[ L^{an}_{R/\mathcal{O}_C} {\otimes}_R^L R_V \simeq  L^{an}_{R/\mathcal{O}_C} \widehat{\otimes}_R^L R_V \simeq  L^{an}_{R_V/\mathcal{O}_C},\] 
where the first isomorphism is from the pseudocoherence of $L^{an}_{R/\mathcal{O}_C}$ coming from Theorem~\ref{GRCC} (1). The same reasoning also allows us to drop the completion on the leftmost term in \eqref{cc0}. Inverting $t$ in \eqref{cc0} then gives a triangle
\begin{equation}
\label{cc1}
 L^{an}_{R_V/\mathcal{O}_C}[1/t] {\otimes}_{R_V[1/t]}^L E \to L^{an}_{W_0/\mathcal{O}_C}[1/t] \to L^{an}_{W_0/R_V}[1/t].
 \end{equation}
Our task was to show that the term on the right has homology only in degrees $0,1$. The term on the left is a finite projective $E$-module in degree $0$: indeed, $L^{an}_{R/\mathcal{O}_C}[1/t]$ is a finite projective $R[1/t]$-module by Theorem~\ref{GRCC} (2) and the smoothness assumption on $\mathrm{Spa}(R[1/t])$, and we have $L^{an}_{R/\mathcal{O}_C} \otimes_{R[1/t]}^L R_V[1/t] \simeq L^{an}_{R_V/\mathcal{O}_C}[1/t]$ by the reasoning explained above. By the long exact sequence, it thus suffices to check that $L^{an}_{W_0/\mathcal{O}_C}$ has homology only in degrees $0,1$. For this, we consider transitivity triangle for $\mathcal{O}_C \to V \to W_0$:
\[ L^{an}_{V/\mathcal{O}_C} \widehat{\otimes}_V^L E \to L^{an}_{W_0/\mathcal{O}_C}[1/t] \to L^{an}_{W_0/V}[1/t].\]
Now $L^{an}_{V/\mathcal{O}_C}[1/t]$ is a finite free $K$-module as explained previously, so the term on the left is a finite free $E$-module. It remains to observe that $L^{an}_{W_0/V}[1/t]$ has homology in degrees $0,1$ by Corollary~\ref{NACC}.
\end{proof}

\begin{remark}
In Theorem~\ref{keystep}, one cannot strengthen the conclusion from regularity to smoothness. For example, the Frobenius map on the Tate algebra over any characteristic $p$ nonarchimedean field satisfies the hypothesis of Theorem~\ref{keystep}, but does not have a single smooth fibre. 
\end{remark}

\begin{remark}
The proof of Theorem~\ref{keystep} relies on the analytic cotangent complex. When $C$ is discretely valued, it is possible to prove Theorem~\ref{keystep} by avoiding the cotangent complex and using instead Popescu's desingularization theorem. Indeed, one first observes that $R$ is an excellent noetherian ring: by Elkik's theorem, we can write $R$ as the $t$-completion of a finite type $\mathcal{O}_C$-algebra, so $R$ is excellent since $\mathcal{O}_C$ is so. But then the maps $R \to R \otimes_T T_\eta \to R_V$ are regular: the first map is a localization, while the second one is the completion of an excellent ring. By Popescu's theorem, the map $R \to R_V$ is ind-smooth. It follows that $R_V[1/t]$ must be regular since $R[1/t]$ is so. In fact, this reasoning shows that any property of $R[1/t]$ that is local for the smooth topology passes to $R_V[1/t]$. We do not know how to prove the analogous statement in the general case. 
\end{remark}

\subsection{Generic flatness}
Fix a nonarchimedean base field $K$. Our goal in this section is the following theorem.

\begin{theorem}\label{genericflatness}Let $f:X \to Y$ be a quasicompact map of rigid spaces over $K$.  Assume that $Y$ is geometrically reduced. Let $\mathrm{Fl}_{X/Y} \subset Y$ be the maximal open subset such that $X\times_Y U \to U$ is flat. Then $\mathrm{Fl}_{X/Y}$ contains all weakly Shilov points of $Y$. In particular, $\mathrm{Fl}_{X/Y}$ is a dense open subset of $Y$.
\end{theorem}

Observe that even the non-emptiness of $\mathrm{Fl}_{X/Y}$ is not clear a priori. Moreover, if $K$ has characteristic $0$, then it suffices to assume that $Y$ is reduced: as in algebraic geometry, this implies geometric reducedness in characteristic $0$ (see \cite[Lemma 3.3.1]{ConradIrr}).

\begin{lemma} Let $A$ be a ring with a locally nilpotent ideal $J \subset A$, and let $f:R \to S$ be any map of finitely presented flat $A$-algebras. Then $f$ is flat if and only if $\overline{f}: R/J R \to S / J S$ is flat. 
\end{lemma}
\begin{proof}``Only if'' is clear. Conversely, suppose $\overline{f}$ is flat. If $J$ is nilpotent (e.g. if $A$ is Artinian), then flatness of $f$ follows from Proposition 0.8.3.7 in \cite{FK}. 

In the general case, write $A$ as a filtered colimit $A \simeq \colim_{i \in I} A_i$ where $A_i \subset A$ is a filtered system of $\mathbf{Z}$-algebras of finite type. Set $J_i = J \cap A_i$, so $J_i \subset A_i$ is a nilpotent ideal. By standard approximation arguments, there is an index $i_0 \in I$ such that $f$ is the base change along $A_{i_0} \to A$ of a map $f_{i_0} : R_{i_0} \to S_{i_0}$ of finitely presented $A_{i_0}$-algebras. For all $i \geq i_0$, write $f_i : R_i \to S_i$ for the evident base change of $f_{i_0}$. By two applications of \cite[Th\'eor\`eme 11.2.6]{EGAIV3}, $R_i$ and $S_i$ are flat $A_i$-algebras for all sufficiently large $i$. 

Next, note that $\overline{f}$ is the colimit of the diagrams $\overline{f_i}: R_i / J_i R_i \to S_i / J_i S_i$. Since $\overline{f}$ is flat by assumption, $\overline{f_i}$ is flat for all sufficiently large $i$ by another application of \cite[Th\'eor\`eme 11.2.6]{EGAIV3}. But $J_i$ is nilpotent, so flatness of $\overline{f_i}$ implies flatness of  $f_i$ by the special case of the lemma treated in the first paragraph of the proof. Therefore $f_i$ is flat for all sufficiently large $i$, so $f$ is flat.
\end{proof}

Recall that an adic ring $A$ admitting a finitely generated ideal of definition $I$ is called \emph{topologically universally (t.u.) rigid-noetherian} if the scheme $\Spec A \left \langle T_1,\dots,T_n \right \rangle \smallsetminus V(I A \left \langle T_1,\dots,T_n \right \rangle)$ is Noetherian for all $n\geq 0$. If $A$ is a tft $K$-algebra, then any ring of definition $A_0 \subset A$ is t.u. rigid-Noetherian.

\begin{proposition} Let $A$ be a t.u. rigid-Noetherian ring, and let $J\subset A$ be any open ideal consisting of topologically nilpotent elements. Let $f: R\to S$ be a morphism of flat and topologically finitely presented $A$-algebras. Then $f$ is flat if and only if $\overline{f}: R/JR \to S/JS$ is flat.
\end{proposition}
\begin{proof} ``Only if'' is clear. For the converse, choose a finitely generated ideal of definition $I \subset A$ contained in $J$; let $\overline{J}_n \subset A/I^n$ be the image of $J$, so $\overline{J}_n$ is locally nilpotent. Since $R/JR \to S/JR$ is flat, the previous lemma implies that $R/I^nR \to S/I^n S$ is flat for all $n\geq 1$. By Corollary 0.8.3.9 in \cite{FK}, we then deduce that $R \to S$ is flat. 
\end{proof}

\begin{proposition}\label{keyflatness} Let $K$ be a nonarchimedean field, and let $f: A \to B$ be a map of tft $K$-algebras with $A$ geometrically reduced. Then there is a (nonempty) rational subset $U \subset \Spa A$ containing the Shilov boundary such that $\Spa B \times_{\Spa A} U \to U$ is flat.
\end{proposition}
\begin{proof} 
By (the proof of) Theorem 1.3 in \cite{BLR}, we can find a finite \'etale Galois extension $K'/K$ such that the unit ball $A_{K'}^{\circ} \subset A_{K'} \overset{def}{=} A \otimes_K K'$ is topologically finitely presented over $K'^\circ$ and the special fiber $A_{K'}^{\circ} / K'^{\circ \circ} A_{K'}^{\circ}$ is (geometrically) reduced. Choose an open tfp $K'^\circ$-algebra $B_0 \subset B \otimes_K K'$ such that $(f \otimes_K K')(A_{K'} ^\circ) \subset B_0$. Let $k'$ be the residue field of $K'^\circ$. Then $A_{K'}^{\circ} \otimes_{K'^\circ} k' \to B_0 \otimes_{K'^\circ} k'$ is a map of finite-type $k'$-algebras with reduced source, so there exists a non-zero-divisor $f \in A_{K'}^{\circ} \otimes_{K' \circ} k'$ such that $(A_{K'}^{\circ} \otimes_{K'^\circ} k')[1/f] \to (B_0 \otimes_{K'^\circ} k')[1/f]$ is flat. Choose any lift $\tilde{f} \in A_{K'}^{\circ}$, and let $C$ be the $\pi$-adic completion of $A_{K'}^{\circ}[1/\tilde{f}]$; similarly, let $D$ be the $\pi$-adic completion of $B_0[1/\tilde{f}]$.  Applying the previous proposition with $A=K'^\circ$, $J = K'^{\circ \circ}$, $R=C$, and $S=D$, we deduce that the map $C \to D$ is flat. Then $\Spa C[1/\pi] \to \Spa A_{K'}$ is the inclusion of the Laurent domain $U(\frac{1}{\tilde{f}})$, and $\Spa B_{K'} \times_{\Spa A_{K'}} U(\frac{1}{\tilde{f}}) \cong \Spa D[1/\pi]$ by design, so $\Spa B_{K'} \times_{\Spa A_{K'}} U(\frac{1}{\tilde{f}}) \to U(\frac{1}{\tilde{f}})$ is flat. Moreover, $U(\frac{1}{\tilde{f}})$ contains all the Shilov points of $\Spa A_{K'}$ by construction.

It remains to undo the base change from $K$ to $K'$. For this, let $h \in A^\circ$ be the image of $\tilde{f}$ under the norm map $A_{K'} \to A$, and let $U=U(\frac{1}{h}) \subset \Spa A$ be the associated Laurent domain. We claim that $U$ satisfies the conclusions of the theorem. Indeed, writing $\pi: \Spa A_{K'} \to \Spa A$ for the evident finite \'etale map, it is clear from the definitions that $\pi^{-1}(U) = \cap_{g \in \mathrm{Gal}(K'/K)} U(\frac{1}{\tilde{f}})g$ as subsets of $\Spa A_{K'}$. Since each $g$-translate of $U(\frac{1}{\tilde{f}})$ still contains all Shilov points of $\Spa A_{K'}$, we see that $\pi^{-1}(U)$ contains all Shilov points of $\Spa A_{K'}$, and so $U$ contains all Shilov points of $\Spa A$. Moreover, the results in the previous paragraph show that $f_U: \Spa B \times_{\Spa A} U \to U$ becomes flat after base change along the surjective finite \'etale map $\pi^{-1}(U) \to U$, so $f_U$ is flat by \cite[Proposition 3.1.12]{W}. This concludes the proof.
\end{proof}

\begin{proof}[Proof of Theorem \ref{genericflatness}]
It suffices to check that $\mathrm{Fl}_{X/Y}$ contains an open neighborhood of every weakly Shilov point $y \in Y$. The formation of $\mathrm{Fl}_{X/Y}$ commutes with base change along open immersions $Y' \to Y$ in the evident sense, so we're reduced to showing that if $Y$ is affinoid, then $\mathrm{Fl}_{X/Y}$ contains an open neighborhood of every Shilov point of $Y$.  For this, cover $X$ by finitely many open affinoid subsets $X_i$. Then each $\mathrm{Fl}_{X_i/Y} \subset Y$ contains all Shilov points of $Y$ by Proposition \ref{keyflatness}. Since $\mathrm{Fl}_{X/Y} = \cap_i \mathrm{Fl}_{X_i/Y}$, we deduce that $\mathrm{Fl}_{X/Y}$ contains all Shilov points of $Y$.
\end{proof}

\subsection{Generic smoothness}
 
In this subsection, we combine Theorem \ref{keystep} and Theorem \ref{genericflatness} to prove our main generic smoothness result. We begin by translating Theorem \ref{keystep} into geometric language.

\begin{theorem}\label{keystepgeom} Fix a nonarchimedean field $K$, and let $f:X=\Spa B \to Y = \Spa A$ be a map of $K$-affinoid rigid spaces. Suppose that $X$ is smooth over $\Spa K$. Then for any Shilov point $y \in Y$, the adic fiber $X_y = \Spa(B\widehat{\otimes}_{A} K_y)$ is regular.  In particular, if $K$ has characteristic zero, then $X_y \to \Spa K_y$ is smooth.
\end{theorem}
\begin{proof}
The first part follows immediately from Theorem~\ref{keystep} thanks to the characterization of Shilov points in Proposition~\ref{ShilovAffd} once one observes that the local rings of $\mathrm{Spa}(R)$ for a regular tft $K$-algebra $R$ are regular. The last statement follows as a rigid analytic space over a characteristic zero nonarchimedean field is smooth exactly when all of its local rings are regular. 
\end{proof}

We also need the following result, relating smoothness and fibral smoothness.

\begin{lemma}\label{fibralsmoothnesswithflatness} Let $f:X \to Y$ be any map of rigid spaces, and let $x \in X$ be any point with image $y = f(x)$. Then the following are equivalent

\begin{enumerate}
\item $f$ is smooth at $x$.
\item $f$ is flat in a neighborhood of $x$ and $X_y = X \times_Y \Spa(K_y, K_y^+) \to \Spa(K_y, K_y^+)$ is smooth at $x$.
\end{enumerate}

\end{lemma}
\begin{proof} This follows from (the proof of) Lemma 2.9.2 in \cite{W}.
\end{proof}

\begin{theorem}\label{genericsmoothness} Let $f: X \to Y$ be any quasicompact map of rigid spaces in characteristic zero. Assume that $X$ is smooth and $Y$ is reduced. Then there is a dense open subset $U \subset Y$ such that $f^{-1}(U) \to U$ is smooth.
\end{theorem}

\begin{remark}Consideration of standard examples (for instance, quasi-elliptic fibrations in characteristics $2$ and $3$) shows that no result like this can hold in positive characteristic, not even with the weaker conclusion that $f^{-1}(U) \to U$ is smooth up to a universal homeomorphism.
\end{remark}

\begin{proof} Let $y \in Y$ be any weakly Shilov point. By Theorem \ref{genericflatness}, we can choose some open subset $U(y) \subset Y$ containing $y$ such that $X \times_Y U(y) \to U(y)$ is flat. Moreover, by Theorem \ref{keystepgeom}, the entire fiber $X_y$ is smooth over $y$. Applying Lemma \ref{fibralsmoothnesswithflatness} and the openness of the smooth locus in the source, we deduce that every point $x \in f^{-1}(y)$ admits a quasicompact open neighborhood $W_x$ in $X$ such that $W_x \to Y$ is smooth. Forming a suitable union of the $W_x$'s and using the quasicompacity of $f$, we deduce that there is a quasicompact open neighborhood $W \subset X $ of the fiber $f^{-1}(y)$ such that $W \to Y$ is smooth. Shrinking $U(y)$ further, we may assume by an easy quasicompacity argument that $f^{-1}(U(y)) \subset W$. In particular, $X \times_Y U(y) \to U(y)$ is smooth. 

Since weakly Shilov points are dense in $Y$, setting $U = \cup_{y\,\mathrm{weakly\,Shilov}} U(y)$ concludes the proof.
\end{proof}

In the proper case, we can do even better.

\begin{theorem}\label{genericsmoothnessproper} Let $f:X \to Y$ be a proper map of rigid spaces in characteristic zero, with $X$ smooth and $Y$ reduced. Then the maximal open subset $S_f \subset Y$ over which $f$ becomes smooth is a dense Zariski-open subset.
\end{theorem}
\begin{proof} 
Let $W \subset X$ be the Zariski-open subset where $X \to Y$ is smooth, and let $Z = X-W$ be the complement regarded as a rigid space with its induced reduced structure. The composite morphism $g:Z \to Y$ is proper, so by Kiehl's results $g_{\ast} \mathcal{O}_Z$ is a coherent sheaf on $Y$. Since $\mathrm{Supp}(g_{\ast} \mathcal{O}_Z)=f(Z)$, we deduce that $f(Z) \subset Y$ is Zariski-closed, so then $S_f=Y - f(Z)$ is Zariski-open, and density follows from the previous theorem.
\end{proof}

\begin{remark}[Deducing Theorem~\ref{genericsmoothnessproper} from \cite{Ducros}]
\label{rmk:DucrosProof}
Let us indicate how to prove the Berkovich variant of Theorem~\ref{genericsmoothnessproper}, using results from \cite{Ducros}.\footnote{We thank a referee for encouraging us to flesh out this deduction (and in fact for providing a complete argument).} Precisely, given a nonarchimedean base field $K$ of characteristic $0$, we claim that if $f:X \to Y$ is a proper map of $K$-analytic spaces in the sense of Berkovich with $X$ quasi-smooth and $Y$ reduced, then $f$ is smooth over a dense Zariski-open subset of $Y$. To see this, let $T \subset Y$ be image under $f$ of the non-smooth locus of $f$;  the latter  is Zariski-closed  by \cite[Theorem 10.7.2]{Ducros}, so the former is Zariski-closed by properness (as in the proof of Theorem~\ref{genericsmoothnessproper}). As $f$ is smooth over the Zariski-open subset $Y-T \subset T$, it suffices to show that $Y-T$ is dense in $Y$. In fact, we claim that $Y-T$ contains all the Abhyankar points $y \in Y$. By \cite[Theorem 10.3.7]{Ducros}, the map $f$ is flat at $y$, so by \cite[Theorem 5.3.4]{Ducros}, it suffices to note that $f^{-1}(y)$ is quasi-smooth (or equivalently regular, as $K_y$ has characteristic $0$) by \cite[Theorem 6.3.7]{Ducros}. 
\end{remark}

\section{The six functors for Zariski-constructible sheaves}
\label{sec:SixFun}

In this section, we use the geometric results of \S \ref{sec:gensmooth} to develop the six functor formalism for Zariski-constructible sheaves in rigid  analytic geometry over a characteristic $0$ field.

\subsection{Definition of Zariski-constructible sheaves}

In this subsection we briefly review the definition and basic properties of Zariski-constructible sheaves on rigid analytic spaces. Most of this material is taken from \cite{H}; the exception is Theorem~\ref{ZCLocal}.

\begin{definition}\label{zcdefinition} Let $X$ be a rigid analytic space over a nonarchimedean field $K$, and let $\Lambda$ be a finite commutative ring.

\begin{enumerate}
\item An \'etale sheaf $\mathscr{F} \in \mathrm{Sh}(X,\Lambda)$ is \emph{lisse} there exists an \'etale cover $\{U_i \to X\}$ such that $\mathscr{F}|_{U_i}$ is the constant sheaf associated to a finitely generated $\Lambda$-module. 

\item  A complex $A \in D(X,\Lambda)$ is {\em lisse} if the cohomology sheaves $\mathcal{H}^n(A)$ are lisse for all $n$. We write $D_{lis}(X,\Lambda) \subset D(X,\Lambda)$ for the full subcategory spanned by lisse complexes.

\item  An \'etale sheaf $\mathscr{F} \in \mathrm{Sh}(X,\Lambda)$ is \emph{Zariski-constructible} if $X$ admits a locally finite stratification $X = \coprod_{i \in } X_i$ into Zariski locally closed subsets $X_i$ such that $\mathscr{F}|X_i$ is a lisse sheaf of $\Lambda$-modules for all $i \in I$. We write $\mathrm{Sh}_{zc}(X,\Lambda)$ for the full subcategory of Zariski-constructible sheaves.

\item A complex $A \in D(X,\Lambda)$ is Zariski-constructible if the cohomology sheaves $\mathcal{H}^n(A)$ are Zariski-constructible for all $n$. We write $D_{zc}(X,\Lambda)$ for the full subcategory spanned by Zariski-constructible complexes.
\end{enumerate}
One has the bounded below variant $D^+_{zc}(X,\Lambda)$; similarly for $D^{-}$ and $D^b$. Finally, let $D^{(b)}_{zc}(X,\Lambda) \subset D_{zc}(X,\Lambda)$ denotes the full triangulated subcategory of complexes which are locally bounded; similarly for $D^{(-)}$ and $D^{(+)}$. The natural $\infty$-categorical refinements of $D(X,\Lambda), D^{(b)}_{zc}(X,\Lambda)$, etc. shall be denoted $\mathcal{D}(X,\Lambda), \mathcal{D}^{(b)}_{zc}(X,\Lambda)$, etc., as usual.
\end{definition}

\begin{warning} 
\label{ZCWarning}
Let us record some subtleties concerning this notion.
\begin{enumerate}

\item Given a Zariski-open immersion $j:U \to X$ and a Zariski-constructible sheaf $\mathscr{F}$ on $U$, the extension $j_! \mathscr{F}$ can fail to be Zariski-constructible $X$, unlike the situation in algebraic geometry (see next example). The main problem is that the operation of taking Zariski-closures in $X$ of Zariski-closed subsets of $U$ is poorly behaved in general (e.g., it does something non-trivial over $U$); this issue does not arise if $\mathscr{F}$ is itself locally constant. 

\begin{example}
Let $\mathscr{F}$ be the direct sum of skyscraper sheaves supported at an infinite discrete set of classical points in $(\mathbf{A}^1)^{an}$, and let $j:(\mathbf{A}^1)^{an} \to  (\mathbf{P}^1)^{an}$ be the standard open immersion. Then $j_! \mathscr{F}$ is not Zariski-constructible on $(\mathbf{P}^1)^{an}$: any Zariski-closed set of $(\mathbf{P}^1)^{an}$ must be either finite or all of $(\mathbf{P}^1)^{an}$ by rigid GAGA. 
\end{example}

This phenomenon should not be regarded as a pathology: similar examples occur in complex analytic geometry as well, and are a natural consequence of the non-quasi-compactness of affine space in any kind of analytic geometry.

\item Huber's book \cite{Hub96} defines a notion of ``constructible'' sheaves that is very well-behaved from a topos theoretic perspective. However, these sheaves are typically not Zariski-constructible; for instance, if $j$ is the qcqs open immersion defined by including a closed disc of radius (say) $1/2$ inside a closed disc of radius $1$, then $j_! \Lambda$ is constructible in Huber's sense but is not Zariski-constructible. In fact, the overlap between these two notions is exactly the lisse sheaves.
\end{enumerate}
\end{warning}

Next, we record some simple stability properties of this notion.

\begin{proposition}
\label{ZCStability}
\begin{enumerate}

\item $\mathrm{Sh}_{zc}(X,\Lambda)$ is a weak Serre subcategory of $\mathrm{Sh}(X,\Lambda)$, and $D_{zc}(X,\Lambda)$ is a thick triangulated subcategory of $D(X,\Lambda)$.

\item \emph{(Devissage)}  A sheaf $\mathscr{F} \in \Sh(X,\Lambda)$ is Zariski-constructible iff there is a dense Zariski-open subset $U \subset X$ such that $\mathscr{F}|U$ is lisse and $\mathscr{F}|(X\smallsetminus U)$ is Zariski-constructible.

\item Zariski-constructibility is stable under $f^{\ast}$ for $f$ any morphism of rigid spaces, and under $f_\ast$ for finite morphisms.

\item A sheaf $\mathscr{F} \in \Sh(X,\Lambda)$ is Zariski-constructible iff $\mathscr{F}|X_i$ is Zariski-constructible for all irreducible components $X_i \subset X$.

\end{enumerate}
\end{proposition}

\begin{proof} (1)-(3) are proved in \cite{H}. For (4), one direction is clear. For the other direction, let $f_0: \tilde{X} \to X$ be the normalization of $X$, and let $f_1: \tilde{X} \times_X \tilde{X} \to X$ be the evident map. The hypothesis guarantees that $f_0^\ast \mathscr{F}$ (and then also $f_1^{\ast} \mathscr{F}$) is Zariski-constructible, since $\tilde{X} = \coprod_i \tilde{X}_i$ is the disjoint union of normalizations of the irreducible components of $X$. Since  $f_0$ and $f_1$ are both finite, pushforward along these maps preserves Zariski-constructibility by (3). The exact sequence $ 0 \to \mathscr{F} \to f_{0 \ast} f_0^{\ast} \mathscr{F} \to f_{1 \ast} f_1^{\ast} \mathscr{F}$ now exhibits $\mathscr{F}$ as the kernel of a map between Zariski-constructible sheaves, so we conclude by (1).
\end{proof}

In view of Warning~\ref{ZCWarning} (1), the following result on the analytic (or even \'etale) locality of the notion of Zariski-constructibility is somewhat surprising:

\begin{theorem}
\label{ZCLocal}
Let $X$ be a rigid space over a characteristic zero nonarchimedean field $K$, equipped with an \'etale $\Lambda$-sheaf $\mathscr{F}$. If there exists an \'etale cover $\{U_i\}$ of $X$ such that $\mathscr{F}|_{U_i}$ is Zariski-constructible, then $\mathscr{F}$ is Zariski-constructible.

In particular, the assignment carrying a rigid space $X$ to the $\infty$-category $\mathcal{D}^{(b)}_{zc}(X,\Lambda)$ is a stack for the \'etale topology.
\end{theorem}

We shall prove this result in \S \ref{ss:AlgZC}.  Finally, the following result is often very useful.

\begin{proposition}\label{Dbzcgenerators}If $X$ is a quasicompact rigid space over a nonarchimedean field $K$ of characteristic zero, then $D^{b}_{zc}(X,\Lambda)$ is the thick triangulated subcategory of $D(X,\Lambda)$ generated by objects of the form $f_{\ast}M$ for $f:Y\to X$ a finite morphism and $M$ a constant constructible $\Lambda$-sheaf on $Y$.
\end{proposition}
\begin{proof}
By induction on $\dim X$ and devissage, it's enough to show that if $j:U \to X$ is a dense Zariski-open and $\mathscr{F}$ is lisse and killed by a prime $\ell$, then $j_!{\mathscr{F}}$ lies in the claimed subcategory. For this, choose (as in \cite[Tag 0A3R]{Stacks}) a finite \'etale cover $g:U'\to U$ of prime-to-$\ell$ degree such that $g^\ast \mathscr{F}$ is an iterated extension of copies of $\mathbf{F}_\ell$. Then $\mathscr{F}$ is a summand of $g_\ast g^\ast \mathscr{F}$, so $\mathscr{F}$ is a summand of an iterated extension of copies of $g_{\ast}\mathbf{F}_\ell$. Now extend $g$ to a finite cover $f:X' \to X$ as in \cite{H}, so $j_!\mathscr{F}$ is a summand of an iterated extension of copies of $\mathscr{G} = j_!g_{\ast}\mathbf{F}_\ell = f_{\ast} j'_!\mathbf{F}_\ell$, where $j':U' \to X'$ is the evident map. Letting $i:Z\to X'$ be the complement of $j'$, the exact sequence $0 \to f_{\ast} j'_!\mathbf{F}_\ell \to f_{\ast} \mathbf{F}_\ell \to (f \circ i)_\ast \mathbf{F}_\ell \to 0$ shows that $\mathscr{G}$ lies in the desired subcategory.
\end{proof}

\subsection{Zariski-constructible sheaves via algebraic geometry}
\label{ss:AlgZC}

In this subsection, we describe Zariski-constructible sheaves on affinoids purely in terms of algebraic geometry, and deduce that the property of being Zariski-constructible is \'etale local. 

Let $K$ be a characteristic zero nonarchimedean field. Recall from \cite{H} that for any affinoid $K$-algebra $A$ and any scheme $\mathcal{X}$ locally of finite type over $\Spec A$, there is a naturally associated rigid space $X= \mathcal{X}^{an}$ over $\Spa A$, and a natural  map $X_{\et} \to \mathcal{X}_{\et}$ of sites, inducing a $t$-exact pullback functor $\mu_X : D(\mathcal{X},\mathbf{Z}/n) \to D(X,\mathbf{Z}/n)$ carrying $D_c$ into $D_{zc}$. Here we change notation slightly from \cite{H}, and write $(-)^{an}$ interchangeably for $\mu_X^{\ast}(-)$.

\begin{proposition}[Algebraization of Zariski-constructible sheaves over affinoids]
\label{schemescomparison1}

Fix an affinoid $K$-algebra $A$, and write $\mathcal{S} = \Spec\,A$ and $S=\Spa A$.

\begin{enumerate}

\item If $f:\mathcal{X} \to \mathcal{Y}$ is any finite type map of finite type $\mathcal{S}$-schemes, then for any $\mathscr{F} \in D^b_c(\mathcal{X},\mathbf{Z}/n)$ the natural base change map $(Rf_\ast \mathscr{F})^{an} \to Rf^{an}_{\ast}\mathscr{F}^{an}$ is an isomorphism. In particular, $Rf^{an}_{\ast}\mathscr{F}^{an}$ lies in $D^b_{zc}(Y,\mathbf{Z}/n)$.

\item If $\mathcal{X}$ is any finite type $\mathcal{S}$-scheme, the functor $(-)^{an} : D^b_c(\mathcal{X},\mathbf{Z}/n) \to D^b_{zc}(X,\mathbf{Z}/n)$ is fully faithful. If $\mathcal{X}$ is proper over $\mathcal{S}$, it is an equivalence of categories.

\item If $\mathcal{X}$ is any finite type $\mathcal{S}$-scheme, the fully faithful functor $(-)^{an} : D^b_c(\mathcal{X},\mathbf{Z}/n) \to D^b_{zc}(X,\mathbf{Z}/n)$ from (2) identifies the full subcategory of lisse objects on both sides.
\end{enumerate}
\end{proposition}

Note that part (2) applies notably when $\mathcal{X}=\mathcal{S}$.

\begin{proof}
\begin{enumerate}
\item This is exactly \cite[Theorem 1.8]{H}.

\item The full faithfulness is a special case of \cite[Theorem 1.10.ii]{H}. Essential surjectivity in the case where $\mathcal{X} / \mathcal{S}$ is proper can be checked on hearts. We can also assume that $\mathcal{X}$ and $X$ are reduced. Arguing by induction on $\dim \mathcal{X}$ as in the proof of \cite[Theorem 1.7]{H}, one reduces to checking that any sheaf on $X$ of the form $j_! \mathscr{F}$ is in the essential image of $(-)^{an}$; here $j:U \subset X$ is the inclusion of any normal Zariski-open subset and $\mathscr{F}$ is any lisse sheaf.  To do this, first note that the complement $Z = X-U$ algebraizes to a closed subscheme $\mathcal{Z} \subset \mathcal{X}$ by relative rigid GAGA, so then $\mathcal{U} = \mathcal{X} - \mathcal{Z}$ is an algebraization of $U$. We're now reduced to proving that the analytification functor $\mathrm{F\acute{E}t}(\mathcal{U}) \to \mathrm{F\acute{E}t}(U)$ is an equivalence of categories. We explain the construction of an essential inverse. Suppose $V \to U$ is any finite \'etale map. By \cite[Theorem 1.6]{H}, this extends uniquely to a branched covering $V' \to X^n$, where $X^n$ is the normalization of $X$. By relative rigid GAGA again, this algebraizes to a branched covering $\mathcal{V}' \to \mathcal{X}^n$, and then $\mathcal{V}:= \mathcal{V}' \times_{\mathcal{X}^n} \mathcal{U} \to \mathcal{U}$ is the desired algebraization of $V$. 

\item Since we already have full faithfulness, it suffices to prove essential surjectivity on the hearts, i.e., we want to realize a lisse sheaf on $X$ as the analytification of a unique lisse sheaf on $\mathcal{X}$. By uniqueness and Zariski/analytic descent for lisse sheaves in algebraic/analytic geometry, we may assume that $\mathcal{X}$ is separated (or even affine). By finite descent for lisse sheaves in both algebraic and analytic geometry, we may also assume $\mathcal{X}$ is normal. In this case, we can realize $\mathcal{X}$ as an an open subscheme of a normal proper $\mathcal{S}$-scheme $\overline{\mathcal{X}}$; the argument used in the proof of (2) now yields the desired algebraization.
\end{enumerate}
\end{proof}

We shall use the above description to prove Theorem~\ref{ZCLocal} by a topological argument. To run this argument, we need a couple of lemmas in pure algebraic geometry on the existence and properties of the maximal open set where a constructible sheaf is lisse.

\begin{lemma}
\label{MaxLisseOpen}
Let $X$ be a scheme and let $\mathscr{F}$ be a constructible sheaf on $X$. Then there exists a maximal open subset $U_{\mathscr{F}} \subset X$ such that $\mathscr{F}|_{U_{\mathscr{F}}}$ is locally constant. Moreover, $U_{\mathscr{F}}$ is given by either of the following equivalent descriptions:
\begin{enumerate}
\item The set of all $x \in X$ such that $\mathscr{F}|_{X_x}$ is locally constant. (Here $X_x$ is the local scheme of $X$ at $x$.)
\item The set of all $x \in X$ admitting  an open neighborhood $x \in U \subset X$ with $\mathscr{F}|_U$ being locally constant.
\end{enumerate}
In particular, $U_{\mathscr{F}}$ contains all the generic points of $X$.
\end{lemma}
\begin{proof}
As local constancy is a local property, the collection of all opens $V \subset X$ such that $\mathscr{F}|_V$ is locally constant is stable under taking unions. Taking the union of all such opens then gives the maximal open $U_{\mathscr{F}}$ such that $\mathscr{F}|_{U_{\mathscr{F}}}$ is open. It is also clear from this description that $U_{\mathscr{F}}$ agrees with the set in (2). The set in (2) is trivially contained in the set in (1). Conversely, as the functor sending a scheme $Y$ to its category of locally constant sheaves (resp. constructible sheaves) is locally finitely presented (i.e., carries cofiltered limits of affine schemes to direct limits), the set in (1) is also contained in the set in (2), so the two sets coincide. The final statement is clear from the description of $U_{\mathscr{F}}$ given by the set in (1).
\end{proof}

\begin{lemma}
\label{MaxLisseOpenBC}
The formation of the open set $U_{\mathscr{F}} \subset X$ associated with a pair $(X,\mathscr{F})$ as in Lemma~\ref{MaxLisseOpen} is compatible with pullback along universally generalizing maps of schemes.
\end{lemma}
\begin{proof}
Let $f:Y \to X$ be a universally generalizing map of schemes. We have an obvious containment $f^{-1} (U_{\mathscr{F}}) \subset U_{f^* \mathscr{F}}$, and we must show it is an equality. Assume towards contradiction that there exists some $y \in U_{f^* \mathscr{F}} - f^{-1}(U_{\mathscr{F}})$. Thus, $(f^* \mathscr{F})|_{Y_y}$ is locally constant but $\mathscr{F}|_{X_{f(y)}}$ is not locally constant. There is then a specialization $x_1 \rightsquigarrow x_2$ of geometric points of $X_{f(y)}$ such that the corresponding cospecialization map $\mathscr{F}_{x_2} \to \mathscr{F}_{x_1}$ is not an isomorphism. Choose an absolutely integrally closed valuation ring $V$ and a map $\eta_X:\mathrm{Spec}(V) \to X_{f(y)}$ that witnesses the specialization $x_1 \rightsquigarrow x_2$, so $\eta_X^*(\mathscr{F}|_{X_{f(y)}})$ is not locally constant. Now that the map $Y_y \to X_{f(y)}$ is universally generalizing (it factors as $Y_y \to X_{f(y)} \times_X Y \to X_{f(y)}$, with both maps being universally generalizing) and surjective (all points of $X_{f(y)}$ specialize to $y$, so surjectivity follows from the universally generalizing property). By stability of universally generalizing surjective maps under base change, we can replace $V$ with an extension if necessary to lift $\eta_X$ to a map $\eta_Y:\mathrm{Spec}(V) \to Y_y$. But then we have $\eta_X^* (\mathscr{F}|_{X_{f(y)}}) = \eta_Y^* (f^* \mathscr{F}|_{Y_y})$; this is a contradiction as the left side is not locally constant by choice of $\eta_X$, while the right side is locally constant by choice of $y$.
\end{proof}

\begin{proof}[Proof of Theorem~\ref{ZCLocal}]
Let us first give the argument when $\{U_i\}$ is a cover of $X$ for the analytic topology. We proceed by induction on $\dim(X)$. The $\dim(X)=0$ case is clear: $X$ is a disjoint union of points in this case. In general, as Zariski-constructibility is stable under pullback, we may assume each $U_i = \mathrm{Spa}(A_i)$ is affinoid. Write $\mathcal{U}_i = \mathrm{Spec}(A_i)$ for the obvious algebraization of $U_i$.  The map $U_i \to \mathcal{U}_i$ identifies constructible $\mathbf{Z}/n$-sheaves on the target with Zariski-constructible $\mathbf{Z}/n$-sheaves on the source by Proposition~\ref{schemescomparison1}, so there is a unique constructible $\mathbf{Z}/n$-sheaf $\mathscr{F}_i$ over $\mathcal{U}_i$ descending $\mathscr{F}|_{U_i}$. Let $\mathcal{V}_i := \mathcal{U}_{i,\mathscr{F}_i} \subset \mathcal{U}_i$ be the maximal Zariski open over which $\mathcal{F}_i$ is locally constant as in Lemma~\ref{MaxLisseOpen}, let $V_i \subset U_i$ be its Zariski-open preimage, and let $Z_i \subset U_i$ be the Zariski-closed complement of $V_i$ (regarded as a reduced rigid space); note that $Z_i \subset V_i$ is nowhere dense as $\mathcal{V}_i \subset \mathcal{U}_i$ contains all the generic points. As the natural algebraizations of the maps given by rational localizations of affinoids are universally generalizing, Lemma~\ref{MaxLisseOpenBC} implies that for all $i,j$, the Zariski-open subsets $V_i \cap (U_i \cap U_j)$ and $V_j \cap (U_i \cap U_j)$ of $U_i \cap U_j$ agree, and consequently their complements also agree. By descent for coherent ideal sheaves applied to the ideal sheaves $\mathcal{I}_{Z_i} \subset \mathcal{O}_{U_i}$ of the $Z_i$'s, there is a unique Zariski-closed subset $Z \subset X$ such that $Z \cap U_i = Z_i$. As $Z_i$ is nowhere dense in $U_i$ for all $i$, we must have $\dim(Z) < \dim(X)$. Moreover, the sheaf $\mathscr{F}$ is lisse over $X-Z$ by construction. Induction on dimension shows that $\mathscr{F}|_Z$ is also Zariski-constructible, so we win.

To adapt this argument to the \'etale topology, note the proof above has two essential ingredients:
\begin{enumerate}
\item For each pair of indices $i,j$, if $V \subset U_i \cap U_j = U_i \times_X U_j$ is an open affinoid, then the natural algebraization of $V \to U_i$ (resp. $V \to U_j$) is a universally generalizing map of affine schemes.
\item Descent for coherent sheaves holds true with respect to the cover $\{U_i\}$.
\end{enumerate}
These properties are also true for \'etale covers of $X$, so the descent claim also holds true in the \'etale topology. Indeed, \'etale descent for coherent sheaves on rigid spaces is \cite[Corollary 3.2.3]{dJvdP}, while the first property reduces to the well-known fact that for any \'etale map of affinoid rigid spaces, the associated ring map is flat.  
\end{proof}

\subsection{Pushforward, $\otimes$, and $\inthom$}

In this subsection, we prove some of our main stability properties for Zariski-constructible sheaves. Until further notice, we fix a characteristic zero nonarchimedean base field $K$ of residue characteristic $p$ (with $p=0$ allowed).  The first main result in this section is the following theorem, which was conjectured by the second author \cite[Conjecture 1.14]{H}.

\begin{theorem}
[Proper direct images]
\label{properdirectimage}
Let $f:X\to Y$ be a proper map of rigid spaces over $K$. Then $Rf_*$ preserves $D^{(b)}_{zc}(-, \mathbf{Z}/n)$.
\end{theorem}

Note that $p|n$ is allowed here. We caution the reader that if $f:X \to Y$ is a proper map of rigid spaces with $Y$ irreducible, then in contrast with the situation for algebraic varieties, it is not always true that $X$ has finitely many irreducible components, or that the fibers of $f$ have bounded dimension.\footnote{ For a simple example where both conditions fail, let $Y=(\Spec K[T])^{an}$ be the rigid affine line, set $X_n = (\mathbf{P}^n)^{an}$ and let $f_n: X_n \to Y$ be the map which factors over the inclusion of the closed point $T=p^{-n}$.  Then $f  = \coprod f_n : \coprod X_n \to Y$ is proper.} Thus, it is necessary to use $D^{(b)}$ and not $D^b$ in the above formulation, and similarly in many other places in this section.

\begin{proof}
By Theorem~\ref{ZCLocal}, we may assume $Y=\mathrm{Spa}(A)$ for an affinoid $K$-algebra $A$. In particular, $X$ and $Y$ are both qcqs.  We may clearly also assume that $Y$ is reduced. Our task is to show that $Rf_*$ preserves $D^b_{zc}(-,\mathbf{Z}/n)$ in this situation. As $X$ is quasi-compact, we may apply Proposition~\ref{Dbzcgenerators}, so it suffices to show that $Rf_* \mathbf{Z}/n \in D^b_{zc}$. In fact, as the fibres of $f$ have bounded dimension (e.g., by checking on formal models), we know by cohomological dimension estimates and proper base change that $Rf_*$ has finite cohomological dimension, so it is enough to show that $Rf_* \mathbf{Z}/n \in D_{zc}$, i.e., that each $R^i f_*\mathbf{Z}/n$ is Zariski-constructible on $Y$. By proper base change and induction on $\dim(Y)$, it suffices to find a dense open $U \subset Y$ such that $(R^if_* \mathbf{Z}/n)|_U$ is locally constant. As $K$ has characteristic $0$, Temkin's \cite[Theorem 1.1.13 (i)]{Temkin} gives a proper hypercover $\epsilon:X^\bullet \to X$ with each $X^i$ being $K$-smooth. Cohomological descent enables us to compute $Rf_* \mathbf{Z}/n$ as the totalization of $Rg_* \mathbf{Z}/n$, where $g  = f \circ \epsilon$. As $\mathcal{H}^i$ of the totalization of a cosimplicial object $K(\bullet):\Delta \to \mathcal{D}^{\geq 0}$ only depends on truncated cosimplicial object $K(\bullet)|_{\Delta_{\leq i+1}}$, we can replace $X^\bullet$ with the finite diagram $X^{\leq i+1}$ and then by each $X^j$ to assume that $X$ is smooth. As $Y$ is reduced, our generic smoothness result (Theorem~\ref{genericsmoothnessproper}) yields a Zariski-dense Zariski-open $U \subset X$ such that $f:X \to Y$ is smooth over $U$. It then suffices to show that $R^i g_* \mathbf{Z}/n$ is locally constant when $g$ is both proper and smooth. To check this, we may assume $n=\ell$ is a prime. Now if $\ell \neq p$, the claim reduces to \cite[Corollary 6.2.3]{Hub96}, while the claim for $\ell = p$ reduces to \cite[Theorem 10.5.1]{SWBerkeley}.
\end{proof}

As a consequence of Theorem \ref{properdirectimage}, we also get some additional stability results. First, locally constant sheaves are carried to Zariski-constructible complexes via pushforward along a fairly general class of maps.

\begin{corollary}[Direct images of lisse complexes]
\label{openpushforward} 
Let $f:X \to Y$ be a Zariski-compactifiable map of rigid spaces. Then $Rf_!$ and $Rf_*$ carry $D^{(b)}_{lis}(-,\mathbf{Z}/n)$ into $D^{(b)}_{zc}(-,\mathbf{Z}/n)$.
\end{corollary}

As explained in Warning~\ref{ZCWarning} (1), the functor $Rf_*$ does not preserve $D^{(b)}_{zc}$ in general, even for Zariski-open immersions. Thus, the above seems to be the best general statement one can expect.

\begin{proof}
The statement is local on the target by Theorem~\ref{ZCLocal}, so we may assume $Y=\mathrm{Spa}(A)$ is affinoid. By assumption, we can factor $f$ as $X \xrightarrow{j} \overline{X} \xrightarrow{g} Y$ with $j$ being a Zariski-open immersion and $g$ being proper. Using Theorem~\ref{properdirectimage}, we can reduce to the case $f=j$ is a Zariski-open immersion. The claim for $Rf_!$ is clear from the definition of Zariski-constructible sheaves (using $\{X,\overline{X}-X\}$ as the stratification on $\overline{X}$ witnessing Zariski-constructibility), so it remains to check the assertion for $Rf_*$. As $Y$ is affinoid, the Zariski-open immersion $f:X \hookrightarrow Y$ is the algebraization of a unique open immersion $g:\mathcal{X} \to \mathcal{Y} = \mathrm{Spec}(A)$. Moreover, any object of $D^b_{lis}(X,\mathbf{Z}/n)$ is the analytification of a unique object in $D^b_{lis}(\mathcal{X},\mathbf{Z}/n)$ by Proposition~\ref{schemescomparison1} (3). Using the compatibility of pushforwards with analytification from Proposition~\ref{schemescomparison1} (1), the claim follows from Gabber's constructibility theorem in \cite[Expose XIII, Theorem 1.1.1]{ILO14}.
\end{proof}

Secondly, $!$-pullback along finite morphisms preserves Zariski-constructibility.

\begin{corollary}[Finite $!$-pullback]
\label{finiteshriekpullback}
Let $f:X \to Y$ be a finite morphism of rigid spaces over $K$. Then the right adjoint $Rf^!:D^+(Y,\mathbf{Z}/n) \to D^+(X,\mathbf{Z}/n)$ to $f_! = f_*$ constructed in \cite[\S 7.1]{Hub96} preserves $D^{(b)}_{zc}(-)$.
\end{corollary}
\begin{proof}
By Theorem~\ref{ZCLocal}, the assertion is \'etale-local on $Y$, so we may assume both $X$ and $Y$ are affinoid, corresponding to a finite map $A \to B$ of affinoid $K$-algebras. Fix $F \in D^b_{zc}(Y,\mathbf{Z}/n)$.  Let $S_F \subset Y$ be the smallest Zariski-closed subset of $Y$ containing the support of $F$. We shall prove the claim by induction on $d_F = \dim(S_F)$. 

If $d_F = 0$, then $F$ is supported at finitely many points. As the claim is \'etale local on $Y$, we may then assume that $F$ is a finite direct sum of sheaves of the form $k_* \mathbf{Z}/n$, where $k:W \to Y$ is the inclusion of a Zariski-closed point. Now $f^!$ and $k_*$ commute: the corresponding statement for left adjoints is the proper base change theorem for $f$. We are thus reduced to checking the statement when $Y$ (and thus $X$) are $0$-dimensional. In this case, up to universal homeomorphisms, the map $f$ is finite \'etale, so $Rf^! = f^*$, so the claim is clear.

Now assume $d_F > 0$. We may then choose a Zariski open subset $j:U \hookrightarrow Y$ such that $U \cap S_F$ is dense is $S_F$, the restriction $L := F|_{U \cap S_F}$ is lisse, and $f$ is finite \'etale (up to universal homeomophisms) over $U$: one can find such an open $U$ as the algebraization of an open $\mathcal{U} \subset \mathrm{Spec}(A)$ satisfying the analogous properties for the map $\mathrm{Spec}(B) \to \mathrm{Spec}(A)$ and the algebraization $\mathcal{F}$ of $F$ (in the sense of Proposition~\ref{schemescomparison1}). Let $i:Z \hookrightarrow Y$ be the closed complement, so we have the standard exact triangle
\[ i_* Ri^! F \to F \to Rj_* (F|_U).\]
Now $Rj_*(F|_U) \simeq k_* Rj'_* L$, where $j':U \cap S_F \to S_F$ and $k:S_F \to X$ are the natural maps (and thus Zariski open and Zariski closed immersions respectively). By Corollary~\ref{openpushforward}, the third term in the triangle above is then Zariski-constructible. The remaining term $i_* Ri^! F$ in the triangle is then also Zariski-constructible, so $Ri^! F$ is itself Zariski-constructible. Applying $Rf^!$ to the above triangle gives a triangle
\[ Rf^!i_* i^! F \to f^! F \to Rf^! k_* Rj'_* L.\]
By proper base change for $f$ as in the previous paragraph, the last term identifies with $k_{X,*} Rj'_{X,*} (f|_{U \cap S_F})^! L$, where $k_X$ and $j'_X$ are the base changes of $k$ and $j'$ along $f$. As $f$ is finite \'etale up to universal homeomorphisms over $U$, we have $(f|_{U \cap S_F})^! L \simeq (f|_{U \cap S_F})^* L$, so this object is lisse on $f^{-1}(U \cap S_F)$. Corollary~\ref{openpushforward} then implies that the third term in the triangle above is lisse. For the first term, using proper base change again lets us write it as $i_{X,*} f_Z^! i^! F$, where $i_X$ and $f_Z$ are the base changes of $i$ and $f$ against $f$ and $i$. As $i^! F$ is known to be Zariski constructible, the induction hypothesis then shows that the first term is also Zariski constructible.
\end{proof}

\begin{remark}
Using results from \cite[\S 7]{Hub96}, in the special case $(p,n)=1$, we can extend Corollary~\ref{finiteshriekpullback} to much larger generality. Indeed, if $f:Y \to X$ is any separated taut morphism of rigid spaces over $K$, then $Rf^!$ sends $D^{(b)}_{zc}(X,\mathbf{Z}/n)$ into $D^{(b)}_{zc}(Y,\mathbf{Z}/n)$. To see this, by Theorem~\ref{ZCLocal}, this assertion can be checked locally on $X$ and $Y$, so we can assume they are affinoid. The map $f$ can be then be factored as the composition of a Zariski-closed immersion followed by a smooth map of pure dimension $d$. The claim for Zariski-closed immersions follows from Corollary~\ref{finiteshriekpullback}, while that for smooth morphisms follows from Huber's \cite[Theorem 7.5.3]{Hub96}, which identifies $Rf^!$ with $f^*(d)[2d]$. 
\end{remark}

We deduce the existence of  $\otimes$ and $\inthom$.

\begin{corollary}[$\otimes$ and $\inthom$]
\label{inthomexists}
Let $X$ be a rigid space. For any $\mathscr{F}, \mathscr{G} \in D^{(b)}_{zc}(X,\mathbf{Z}/n)$ with $\mathscr{F}$ having finite Tor dimension, both $\mathscr{F} \otimes \mathscr{G}$ and $\inthom(\mathscr{F},\mathscr{G})$ lie in $D^{(b)}_{zc}(X,\mathbf{Z}/n)$. 
\end{corollary}

\begin{proof}
We may work locally on $X$ on Theorem~\ref{ZCLocal}, so assume $X$ is affinoid. The claim about tensor products is clear (e.g., by Proposition~\ref{schemescomparison1} and the corresponding statement in algebraic geometry). For $\inthom$, we proceed by induction on $d=\dim(X)$, the case $d=0$ being trivial. Choose a dense Zariski-open $j:U \subset X$ such that both $\mathscr{F}|_U$ and $\mathscr{G}|_U$ are lisse.  Applying $\inthom(-,\mathscr{G})$ to the triangle $j_! j^\ast \mathscr{F} \to \mathscr{F} \to i_{\ast} i^\ast \mathscr{F} \to$, we get a triangle 
\[ i_{\ast} \inthom(i^{\ast} \mathscr{F},i^! \mathscr{G}) \to \inthom(\mathscr{F},\mathscr{G}) \to Rj_{\ast} (\mathscr{F}|_U^\vee \otimes \mathscr{G}|_U) \to,\]
where we simplified the first term using the adjunction defining $i^!$, and the last term by using $\inthom(A,B) = A^\vee \otimes B$ for $A,B \in D(\mathbf{Z}/n)$ with $A \in D_{perf}(\mathbf{Z}/n)$. Now induction on dimension and Corolary~\ref{finiteshriekpullback} ensure that the first term lies in $D^b_{zc}$. The last term lies in $D^b_{zc}$ by Corollary~\ref{openpushforward}, so we win.
\end{proof}

We also deduce the proper base change theorem.

\begin{theorem}\label{properbc}Let 
\[
\xymatrix{X'\ar[d]^{g'}\ar[r]^{f'} & Y'\ar[d]^{g}\\
X\ar[r]^{f} & Y
}
\]
be a Cartesian diagram of rigid spaces over $K$, with $f$ proper. Then for any $\mathscr{F} \in D^{(b)}_{zc}(X,\mathbf{Z}/n)$, the natural base change map $g^\ast Rf_\ast \mathscr{F} \to Rf'_\ast g'^\ast \mathscr{F}$ is an isomorphism.
\end{theorem}
\begin{proof} This easily splits into the two disjoint cases where $p \nmid n$ or $n=p^a$. When $p \nmid n$, the result follows from Huber's much more general base change results \cite[Theorem 4.1.1.(c)]{Hub96}. We may thus assume that $n=p^a$. By Theorem \ref{properdirectimage}, we know that $g^\ast Rf_\ast \mathscr{F}$ and $Rf'_\ast g'^\ast \mathscr{F}$ are Zariski-constructible, hence overconvergent, so it suffices to show that the base change map induces an isomorphism on stalks at all rank one geometric points $\overline{y} \to Y'$. By two applications of \cite[Example 2.6.2]{Hub96}, we compute that $(g^\ast Rf_\ast \mathscr{F})_{\overline{y}} \simeq R\Gamma(X \times_{Y} \overline{g(y)},\mathscr{F})$ and $(Rf'_\ast g'^\ast \mathscr{F})_{\overline{y}} \simeq R\Gamma(X \times_{Y} \overline{y},\mathscr{F})$. We now conclude by Lemma \ref{changeofbasefieldp}.
\end{proof}

We end this section by recording that the equivalence in Proposition~\ref{schemescomparison1} is compatible with all the operations we have seen so far. 

\begin{proposition}\label{schemescomparison2}

Fix an affinoid $K$-algebra $A$, and write $\mathcal{S} = \Spec\,A$ and $S=\Spa A$.  The functor $(-)^{an} : D^b_c(\mathcal{X},\mathbf{Z}/n) \to D^b_{zc}(X,\mathbf{Z}/n)$ for finite type $\mathcal{S}$-schemes $\mathcal{X}$ with $X = \mathcal{X}^{an}$ is compatible with $ \otimes $, $\inthom$ when the first argument has finite Tor dimension, $f^{\ast}$, $Rf_\ast$, $Rf_!$ for compactifiable $f$, and $Ri^!$ for any finite morphism $i$. If $p \nmid n$, it is compatible with $Rf^!$ for compactifiable $f$.
\end{proposition}

\begin{proof}
The compatibilities for $f^\ast$, $\otimes$, and $j_!$ for open immersions $j$ are easy and left to the reader. The compatibility for $Rf_\ast$ is a special case of Proposition~\ref{schemescomparison1} (1). For proper $f$, the claim for $Rf_{\ast}=Rf_!$ is \cite[Theorem 3.7.2]{Hub96} This implies the claim for $Rf_!$ for compactifiable $f$. The result for $Ri^!$ in the case of a closed immersion follows from the triangle $Ri^! \to i^{\ast} \to i^{\ast} Rj_\ast j^\ast \to $ and its analytic counterpart, where $j$ is the complementary open immersion, using the known compatibilities for $i^{\ast}$, $j^{\ast}$, and $Rj_{\ast}$.  Given these compatibilities, the claim for $\inthom$ now follows by induction on the dimension, imitating the devissage carried out in Corollary~\ref{inthomexists}. The result for $Ri^!$ for general finite maps $i: \mathcal{X} \to \mathcal{Y}$ can be proven by following the argument in Corollary~\ref{finiteshriekpullback}. 

Finally, the claim for $Rf^!$ in the case $(n,p)=1$ is local on the source, so we may factor $f$ as $g\circ i$ where $g$ is smooth of some pure relative dimension $d$ and $i$ is a closed immersion. Then $Rf^! = Ri^! \circ Rg^! =Ri^! \circ g^{\ast}[2d](d)$ by Poincar\'e duality for schemes, and $Rf^{an !} = Ri^{an !} \circ g^{an \ast}[2d](d)$ by Poincar\'e duality for rigid spaces as in Huber's book. The result now follows from the known compatibilities for $g^{\ast}$ and $Ri^!$. 
\end{proof}

\subsection{Verdier duality}

It remains to discuss Verdier duality. Recall that if $p$ is invertible in the coefficient ring $\mathbf{Z}/n$, then \cite[\S 7]{Hub96} shows that any separated taut morphism $f: X \to Y$ of rigid spaces over $K$ induces a well-behaved functor $Rf_!: D(X,\mathbf{Z}/n) \to D(Y,\mathbf{Z}/n)$ with a well-behaved right adjoint $Rf^!$. In particular, for any separated taut rigid space $X$ and any $n$ prime to $p$, we can define the dualizing complex $\omega_X=R\pi_X^{!}(\mathbf{Z}/n)$, where $\pi_X : X \to \Spa\,K$ is the structure map. We shall construct dualizing complexes in a different way using results from \cite{ILO14}; our construction works without restriction on $p$, and is equivalent to Huber's if $(p,n)=1$. We will need the following form of unbounded BBDG gluing \cite{BBDG} for complexes. 

\begin{lemma}\label{gluing}Let $(\mathcal{C},\mathcal{O})$ be a ringed site with a final object $X$ and fiber products, and with enough points. Let $\mathcal{B}$ be a collection of open subobjects $U \subset X$ such that $X=\cup_{U \in \mathcal{B}}U$ and for all $U,V\in \mathcal{B}$ we have $U\cap V =\cup_{W \in \mathcal{B}, W \subset U \cap V} W$. 

Suppose we are given objects $K_U \in D(U,\mathcal{O})$ for all $U\in \mathcal{B}$ together with isomorphisms $\rho_{V}^{U}: K_U|_V \to K_V$ for all $V \subset U$ with $V,U \in \mathcal{B}$ which are compatible with composition. Finally, suppose that $\mathrm{Ext}^{i}(K_U,K_U)=0$ for all $U \in \mathcal{B}$ and all $i<0$.

Then there exists a pair $(K,\{ \rho_U\}_{U \in \mathcal{B}} )$ consisting of an object $K \in D(X,\mathcal{O})$ and isomorphisms $\rho_U:K|_U \to K_U$ such that $\rho_{V}^{U} \circ \rho_U = \rho_V$ for all $V \subset U$ with $V,U \in \mathcal{B}$. The pair $(K,\{ \rho_U\} )$ is unique up to unique isomorphism.
\end{lemma}

Recall that a subobject $U \subset X$ is open if the associated map $\mathrm{Sh}(\mathcal{C}_U) \to \mathrm{Sh}(\mathcal{C})$ is an open immersion of topoi, cf. \cite[Tag 08M0]{Stacks} for the latter notion. Note that unlike the unbounded gluing results proved  in \cite{LO} or \cite[Tag 0DCC]{Stacks}, we do not assume here that the underlying site locally has finite cohomological dimension. The price to pay is that we only prove gluing for open covers.
\begin{proof}This follows from the proof of \cite[Tag 0D6C]{Stacks}.
\end{proof}

Next, we recall some results from \cite{ILO14} on the existence and uniqueness of \'etale dualizing complexes for a fairly general class of noetherian schemes. For this, we need good dimension functions.

\begin{proposition}\label{dimensionfunction}
Let $\mathcal{X}$ be a Noetherian universally catenary scheme whose irreducible components are {\em equicodimensional} in the sense of EGA, i.e. such that $\dim(\mathcal{O}_{\mathcal{Y},y}) =\dim \mathcal{Y}$ for every irreducible component $\mathcal{Y} \subset \mathcal{X}$ and every closed point $y \in \mathcal{Y}$.  Then the function $\delta(x)=\dim \overline{ \{ x \} }$ is a dimension function on $\mathcal{X}$, which we call the \emph{canonical} dimension function.
\end{proposition}

The hypotheses here are satisfied for any scheme $\mathcal{X}$ of finite type over $\mathbf{Z}$, over a field, or over $\Spec A$ for some affinoid $K$-algebra $A$. We will only need the latter case.

\begin{proof} By \cite[Corollaire XIV.2.4.4]{ILO14}, the map $x\mapsto \dim \overline{ \{ x \} }$ defines a dimension function on each irreducible component of $\mathcal{X}$. Since these functions agree on overlaps of irreducible components, this implies the claim.
\end{proof}

In the next theorem, we shall use the language introduced in \cite[\S XVII]{ILO14}. In particular, we shall use the notion of a {\em potential dualizing complex} from \cite[\S XVII.2]{ILO14}. Recall that such a complex on a scheme $\mathcal{X}$ equipped with a dimension function $\delta$ is an object $\omega_{\mathcal{X}} \in D^+(\mathcal{X},\mathbf{Z}/n)$ equipped with some additional data: one has specified isomorphisms $R\Gamma_{\overline{x}} (\omega_{\mathcal{X}}) \cong \mathbf{Z}/n[2\delta(x)](\delta(x))$, called {\em pinnings}, for each geometric point $\overline{x} \to \mathcal{X}$ lying over a point $x \in \mathcal{X}$, and these isomorphisms are required to be compatible with immediate specializations in the appropriate sense. The following theorem asserts that such complexes exist in large generality, and have good properties.

\begin{theorem}[Existence of dualizing complexes in algebraic geometry]
\label{DCExistScheme}
Let $\mathcal{X}$ be an excellent Noetherian $\mathbf{Z}[1/n]$-scheme satisfying the hypotheses of Proposition \ref{dimensionfunction}, and set $\Lambda=\mathbf{Z}/n$. Then $\mathcal{X}$ admits a potential dualizing complex $\omega_{\mathcal{X}} \in D^b_{ctf}(\mathcal{X},\Lambda)$ relative to the canonical dimension function, which is unique up to unique isomorphism. The functor $\mathbf{D}_{\mathcal{X}}(-)=\inthom(-,\omega_{\mathcal{X}})$ preserves $D^b_c(\mathcal{X},\Lambda)$ and the biduality map $\mathscr{F} \to \mathbf{D}_{\mathcal{X}} \mathbf{D}_{\mathcal{X}} \mathscr{F}$ is an isomorphism for all $\mathscr{F} \in D^b_c(\mathcal{X},\Lambda)$.
\end{theorem}
\begin{proof}
The existence and uniqueness is \cite[Theorem XVII.5.1.1]{ILO14}, while the rest follows from \cite[Theoreme XVII.6.1.1]{ILO14}.
\end{proof}

\begin{remark}
\label{EtaleDualizingRegular} 
Choose $\mathcal{X}$ as in Theorem~\ref{DCExistScheme}. Assume additionally than $\mathcal{X}$ is regular of (locally constant) dimension $d$. Then the twisted constant sheaf $\mathbf{Z}/n[2d](d)$ comes equipped with the required pinning data thanks to absolute cohomological purity \cite[Theorem XVI.3.1.1]{ILO14}, so it follows that there is a unique isomorphism $\mathbf{Z}/n[2d](d)\cong \omega_{\mathcal{X}}$ compatible with the pinnings.
\end{remark}

Using the preceding results in algebraic geometry and the algebraization results in Proposition~\ref{schemescomparison1}, we construct dualizing complexes in rigid geometry using Lemma~\ref{gluing} on gluing.

\begin{theorem}[Existence of dualizing complexes in rigid geometry]
\label{dualizingcomplex}
Let $X$ be a rigid space over $K$. 
\begin{enumerate}
\item Existence: There exists a natural dualizing complex $\omega_X \in D^{(b)}_{zc}(X,\mathbf{Z}/n)$, characterized up to unique isomorphism by the requirement that its formation commutes with passage to open subsets and is given by the algebraization (in the sense of Proposition~\ref{schemescomparison1}) of the potential dualizing complexes from Theorem~\ref{DCExistScheme} when $X$ is affinoid. Moreover, one has $\omega_X \cong (\mathbf{Z}/n)[2d](d)$ for $X$ smooth of pure dimension $d$, and canonical isomorphisms $\omega_Z \simeq Ri^! \omega_X$ for any finite morphism $i:Z \to X$.

\item $!$-compatibility for $(n,p)=1$: If $X$ is separated and taut and $(n,p)=1$, then $\omega_X \cong R\pi_X^!(\mathbf{Z}/n)$ where $\pi_X: X \to \Spa K$ is the structure map. 

\item Biduality: The dualizing functor $\mathbf{D}_X(-) = \inthom (-,\omega_X)$ induces a contravariant self-equivalence of $D^{(b)}_{zc}(X,\mathbf{Z}/n)$ satisfying biduality $\mathrm{id} \cong \mathbf{D}_X \circ \mathbf{D}_X$ via the natural map.

\item Duality and finite morphisms: For any finite morphism $i:Z \to X$, there are natural identifications of functors $Ri^! \mathbf{D}_X \cong \mathbf{D}_Z i^\ast$ and $i_{\ast} \mathbf{D}_Z \cong \mathbf{D}_X i_\ast$.

\item Duality and open immersions: For a Zariski-open immersion $j:U \to X$, there are natural isomorphisms $j^* \mathbf{D}_X \simeq \mathbf{D}_U j^*$ and $Rj_* \mathbf{D}_U \simeq \mathbf{D}_X j_!$.

\item Base change: If $L/K$ is an extension of nonarchimedean fields, with $a^\ast : D_{zc}(X,\mathbf{Z}/n) \to D_{zc}(X_L,\mathbf{Z}/n)$ the natural pullback map, then $a^\ast \omega_X \cong \omega_{X_L}$.

\item Compatibility with algebraic geometry: If $X=\mathcal{X}^{an}$ for a finite type $K$-scheme $\mathcal{X}$, then $\omega_X = (\omega_{\mathcal{X}})^{an}$. 
\end{enumerate}
\end{theorem}

\begin{proof}
Let us first construct $\omega_X$ by glueing together the analytifications of the dualizing complexes coming from Theorem~\ref{DCExistScheme}.  Let $U=\Spa A$ be any affinoid rigid space, and set $\mathcal{U}=\Spec A$. Then $\mathcal{U}$ satisfies the hypotheses of Theorem~\ref{DCExistScheme}. Let $\omega_{\mathcal{U}} \in D^b_c(\mathcal{U},\mathbf{Z}/n)$ be the dualizing complex provided by that theorem.

By Propositions \ref{schemescomparison1} and \ref{schemescomparison2}, we have an equivalence $(-)^{an}: D^b_c(\mathcal{U},\mathbf{Z}/n) \to D^b_{zc}(U,\mathbf{Z}/n)$ compatible with $\inthom$s. We now define $\omega_U := (\omega_{\mathcal{U}})^{an}$, so $(\mathbf{D}_{\mathcal{U}}\mathscr{F})^{an} \cong \mathbf{D}_U(\mathscr{F}^{an})$ for every $\mathscr{F} \in D^b_c(\mathcal{U},\mathbf{Z}/n)$. Since the dualizing functor $\mathbf{D}_{\mathcal{U}}$ induces a contravariant self-equivalence on $D^b_c(\mathcal{U},\mathbf{Z}/n)$ which satisfies biduality, and $(-)^{an}$ is an equivalence of categories, we now conclude the analogous results for $\mathbf{D}_U$.

To construct $\omega_X$ over an arbitrary rigid space $X$, we apply Lemma \ref{gluing}, taking $\mathcal{B}$ to be the collection of affinoid opens $U \subset X$, and letting $K_U = \omega_U$ as constructed in the previous paragraph. Note that the full faithfulness in Proposition~\ref{schemescomparison1} and Theorem~\ref{DCExistScheme} show that $\mathbf{Z}/n \cong \inthom(\omega_U,\omega_U)$ for all $U \in \mathcal{B}$, so all negative self-exts vanish. Moreover, for any inclusion of open affinoid subsets $g: V \subset U$, there is a natural isomorphism $g^{\ast} \omega_U \cong \omega_V$ compatible with compositions by Lemma \ref{dualizingcomplexweirdpullback} below (applied with $d=0$). The gluing lemma now applies, and gives a unique $\omega_{X} \in D(X,\mathbf{Z}/n)$ equipped with a transitive system of isomorphisms $\omega_X|_U \simeq \omega_U$ for all open affinoids $U\subset X$. 

We now prove this construction has all the required properties.

\begin{enumerate}
\item By construction, the complex $\omega_X$ is characterized by the properties demanded in the first sentence of (1); these also characterize $\omega_X$ uniquely (up to unique isomorphism) by the uniqueness assertions in Theorem~\ref{DCExistScheme} and Lemma~\ref{dualizingcomplexweirdpullback} through the equivalence in Proposition~\ref{schemescomparison1}. It is also clear from the construction and Remark~\ref{EtaleDualizingRegular} that $\omega_X|_{X^{sm}} \cong \mathbf{Z}/n[2d](d)$, where $d$ is the dimension. For the compatibility with $Ri^!$ for finite morphisms, we reduce to the affinoid case following our construction, and use compatibility of the analytification functor with $Ri^!$ (Proposition~\ref{schemescomparison2}) to reduce to the corresponding result in algebraic geometry (\cite[Proposition XVII.4.1.2]{ILO14}). The compatibility with restriction to open subsets is clear from our construction. As these properties guarantee uniqueness of $\omega_X$ up to isomorphism, we are done with (1).

\item The identification with $R\pi_X^!(\mathbf{Z}/n)$ can be checked locally, where it follows by factoring $\pi_X$ as a closed immersion followed by a smooth map and using the results on Poincare duality proved in Huber's book.

\item This follows from the corresponding assertion in the affinoid case (which was explained above whilst constructing $\omega_X$) as the property of being Zariski-constructible is local in the analytic topology (Theorem~\ref{ZCLocal}).

\item This is a formal argument given duality and known adjunctions. For compatibility with $i_*$, fix $G \in D^{(b)}_{zc}(Z,\mathbf{Z}/n)$. Then we have isomorphisms
\begin{align*}
 i_* \mathbf{D}_Z(G) =& i_* \inthom(\mathbf{Z}/n,\mathbf{D}_Z(G)) \\
 =& i_* \inthom(G, \omega_Z) \\
 =& i_* \inthom(G, i^! \omega_X) \\
 =&  \inthom(i_* G, \omega_X) \\
 =&  \mathbf{D}_X(i_* G).
 \end{align*}
 where the second isomorphism is by biduality, the third by $\omega_Z = i^! \omega_X$, and the fourth by the defining property of $i^!$. Comparing the first and last term gives $i_* \mathbf{D}_Z = \mathbf{D}_X i_*$.
 
For compatibility with $i^!$, fix additionally $F \in D^{(b)}_{zc}(X,\mathbf{Z}/n)$.  Then we have bifunctorial isomorphisms
\begin{align*} 
\mathrm{RHom}(i^* \mathbf{D}_X(F), G) =& \mathrm{RHom}(\mathbf{D}_X(F), i_* G)\\
 =& \mathrm{RHom}(\mathbf{D}_X(i_* G), F)  \\
 =& \mathrm{RHom}(i_* \mathbf{D}_Z(G),F)  \\
 =& \mathrm{RHom}(\mathbf{D}_Z(G), Ri^! F) \\
 =& \mathrm{RHom}(\mathbf{D}_Z(Ri^! F),G),
 \end{align*}
 where first equality is by adjunction for $(i^*,i_*)$, the second and last by duality, the third by the equality $i_* \mathbf{D}_Z = \mathbf{D}_X i_*$ we just showed, and the fourth by adjunction for $(i_*,i^!)$.

\item The compatibility with $j^*$ is built into the construction. The rest is again a formal argument using duality and known adjunctions. For $F \in D^{(b)}_{zc}(U,\mathbf{Z}/n)$ and $G \in D^{(b)}_{zc}(X,\mathbf{Z}/n)$, we have
\begin{align*}
\mathrm{RHom}(F, Rj_* \mathbf{D}_U(G)) =& \mathrm{RHom}(j^* F, \mathbf{D}_U(G)) \\
=& \mathrm{RHom}(G, \mathbf{D}_U (j^* F)) \\
=& \mathrm{RHom}(G, j^* \mathbf{D}_X(F)) \\
=& \mathrm{RHom}(j_! G, \mathbf{D}_X(F)) \\
=& \mathrm{RHom}(F, \mathbf{D}_X(j_! G)),
\end{align*}
with first equality by the adjunction $(j^*,Rj_*)$, the second and last by duality, third by $j^* \mathbf{D}_X = \mathbf{D}_U j^*$, and the fourth by the adjunction $(j_!, j^*)$.

\item By our construction of dualizing complexes, it suffices to construct a natural system of such isomorphisms over affinoids. Thus, assume $X=\mathrm{Spa}(A)$ is affinoid with base change $X_L = \mathrm{Spa}(A_L)$. Write $\mathcal{X}$ and $\mathcal{X}_L$ for the natural algebraization, and let $f:\mathcal{X}_L \to \mathcal{X}$ be the natural map. It is enough to construct a natural isomorphism $f^* \omega_{\mathcal{X}}  \simeq \omega_{\mathcal{X}_L}$. We claim this follows from \cite[Proposition 4.1.1]{ILO14} (and uniqueness of potential dualizing complexes). To apply this lemma, we need to know that $f$ is a regular map, and that the dimension function $y \mapsto \delta'(y) := \delta(f(y)) - \mathrm{codim}_{f^{-1}(f(y))}(y)$ on $\mathcal{X}_L$ agrees with the standard dimension function $y \mapsto \dim(\overline{\{y\}})$. The regularity of $f$ is discussed in the last paragraph of the proof of Proposition~\ref{OperationsBC} below, and essentially comes from \cite{A}. To obtain agreement of the dimension functions, it suffices to know that $\delta'(y) = 0$ for a closed point $y$ (as any point on $\mathcal{X}_L$ is linked to a closed point by a finite chain of immediate specializations). Thus, we must check that $\dim(\overline{\{f(y)\}}) =  \mathrm{codim}_{f^{-1}(f(y))}(y)$. Since $y$ is a closed point, this follows from \cite[Lemma 2.1.5]{ConradIrr} applied to the affinoid algebra of functions on integral scheme $\overline{\{y\}} \subset \mathcal{X}$. 

\item To deduce this from our construction, it suffices to show the following: if $\mathcal{X} = \mathrm{Spec}(A)$ is an affine finite-type $K$-scheme and $U = \mathrm{Spa}(B) \subset X = \mathcal{X}^{an}$ is an open affinoid with natural map $\nu:U \to \mathcal{X}$, then there is a natural isomorphism $\nu^* \omega_{\mathcal{X}} \simeq \omega_U$. This follows by Lemma~\ref{dualizingcomplexweirdpullback} below applied to the map $\mathrm{Spec}(B) \to \mathrm{Spec}(A)$ with $d=0$.
\end{enumerate}
\end{proof}

In the course of the previous proof, we used the following lemma.

\begin{lemma}\label{dualizingcomplexweirdpullback} Let $g:V = \Spec B \to U = \Spec A$ be a regular morphism between excellent affine $\mathbf{Q}$-schemes of finite Krull dimension. Suppose moreover that $U$ and $V$ have equicodimensional irreducible components, that $g$ maps closed points to closed points, and that $V \times_{U} \Spec \kappa(g(x))$ is equidimensional of dimension $d$ for all closed points $x\in V$ for some (constant) $d\geq 0$. Then there is a unique isomorphism $g^{\ast} \omega_U[2d](d) \cong \omega_V$ compatible with the pinnings.
\end{lemma}

This lemma applies if $g$ arises from an \'etale map of affinoid rigid spaces $\Spa B \to \Spa A$, with $d=0$. This is the only case we will need.

\begin{proof}
This is a variant of the argument used in Theorem~\ref{dualizingcomplex} (6). By \cite[Exp. XVII, Prop. 4.1.1]{ILO14}, $g^{\ast} \omega_{U}$ is a potential dualizing complex for the dimension function $\tilde{\delta}:|V|\to \mathbf{Z}$ defined by $\tilde{\delta}(x)=\dim \overline{ \{ g(x) \} } - \mathrm{codim}_{g^{-1}(g(x))}(x)$. Our assumptions guarantee that $\tilde{\delta}(x) +d = \dim \overline{ \{ x \} } = 0$ for all closed points $x$. Since the difference of any two dimension functions is locally constant, this implies that $\tilde{\delta}(x) +d = \dim \overline{ \{ x \} } $ for all $x \in V$, and therefore that $g^{\ast} \omega_U[2d](d)$ is a potential dualizing complex for the canonical dimension function on $V$. We now conclude by the uniqueness of potential dualizing complexes.
\end{proof}

\begin{remark}[Duality and proper maps] \label{dualityproper}
In Theorem~\ref{dualizingcomplex}, we have only discussed the functor $Rf^!$ when $p$ is invertible in the coefficient ring, or when $f$ is finite. However, we expect that for any proper map $f:X \to Y$, the functor $Rf_* : D_{zc}^{(b)}(X,\mathbf{Z}/p^n) \to D_{zc}^{(b)}(Y,\mathbf{Z}/p^n)$ admits a right adjoint $Rf^!$ naturally isomorphic to $\mathbf{D}_X f^\ast \mathbf{D}_Y$. One can check that this expectation holds iff there is a natural isomorphism $\mathbf{D}_Y Rf_\ast \cong Rf_\ast \mathbf{D}_X$. For $f$ proper and smooth, ongoing work of Zavyalov confirms these expectations.
\end{remark}

\subsection{Miscellany}

We collect some auxiliary results.

First, we note that the entire formalism is compatible with changing the nonarchimedean base field. More precisely, let $K \to L$ be an extension of characteristic zero nonarchimedean fields. For any rigid space $X/ \Spa K$, there is a natural map of \'etale sites $X_{L,\mathrm{\acute{e}t}} \to X_{\mathrm{\acute{e}t}}$ which induces a pullback functor $D(X) \to D(X_L)$ sending $D_{zc}$ into $D_{zc}$.

\begin{proposition}
\label{OperationsBC}
 Notation as above, the change of base field functors $D_{zc}(X) \to D_{zc}(X_L)$ are compatible (under the appropriate boundedness conditions) with the operations $f^{\ast}$, $\otimes$, $\inthom$, Verdier duality, $Rf_{\ast}$ for proper $f$, $Rf_!$ and $Rf_{\ast}$ on lisse complexes for Zariski-compactifiable morphisms $f$, and $Rf^{!}$ if either $f$ is a finite morphism or $p$ is invertible in the coefficient ring. \end{proposition}

\begin{proof} The compatibilities for $f^{\ast}$, $\otimes$ and $j_!$ are trivial. The compatibilities for $Rf_{\ast}$ in the proper case and $Rj_{\ast}$ in the Zariski-open lisse case are the hardest; the rest follow from these. 

For $Rf_{\ast}$ in the case of a proper map $f:Y \to X$, one easily reduces to the two disjoint cases $p \nmid n$ and $n=p^a$. The first case follows from Huber's general base change theorem \cite[Theorem 4.1.1.b]{Hub96}. The second case can be reduced, via \cite[Theorem 4.1.1.b']{Hub96}, to the situation where $K$ and $L$ are algebraically closed, and then by Theorem \ref{properdirectimage} it can be checked on stalks at classical points of $X_L$ which map to classical points of $X$. At these points it reduces to Lemma \ref{changeofbasefieldp} below.

For $Rj_{\ast}$ on lisse sheaves in the case of a Zariski-open immersion $j:U \to X$, we can assume that $X = \Spa A$ is affinoid. Write $j_L : U_L \to X_L=\Spa A\widehat{\otimes}_K L$ for the base change, and let $j^{alg}:\mathcal{U} \to \mathcal{X} = \Spec A$ and $j^{alg}_L: \mathcal{U}_L \to \mathcal{X}_L = \Spec A_L$ be the evident algebraizations. By (multiple applications of) Proposition \ref{schemescomparison1}, we are reduced to proving that $g^{\ast} Rj^{alg}_{\ast} \mathcal{F} \cong Rj_{L,\ast}^{alg} g'^{\ast} \mathcal{F}$ for any lisse sheaf $\mathcal{F}$ on $\mathcal{U}$, where $g: \mathcal{X}_L \to \mathcal{X}$ is the natural map and $g': \mathcal{U}_L \to \mathcal{U}$ is its base change to $\mathcal{U}$. Since $g$ is regular (see next paragraph), this follows from regular base change \cite[Exp. XVII, Prop. 4.2.1]{ILO14}.  (The results for $\inthom$ and $Ri^!$ can be handled in an entirely analogous way, with the endgame supplied by \cite[Exp. XVII, Prop. 4.2.2 and Cor. 4.2.3]{ILO14}.)

The claimed regularity of $A \to A\widehat{\otimes}_K L$ can be reduced by Noether normalization to the regularity of the map $K\left \langle x_1, \dots, x_n \right \rangle \to L \left \langle x_1,\dots, x_n \right \rangle$. By standard excellence properties of affinoid rings, regularity of this map can be verified after $\mathfrak{m}$-adic completion at all maximal ideals $\mathfrak{m}$ in the source.  This finally reduces us to the fact that for any separable extension of nonarchimedean fields $L/K$ the ring map $K[[x_1,\dots,x_n]] \to L[[x_1,\dots,x_n]]$ is regular, which follows from the formal smoothness of this map, cf. \cite{A}.
\end{proof}

In the previous proof, we used the following lemma.

\begin{lemma}\label{changeofbasefieldp} Let $C' /C / \mathbf{Q}_p$ be an extension of algebraically closed nonarchimedean fields. Then for any proper rigid space $X/C$ and any $\mathscr{F} \in D^{b}_{zc}(X,\mathbf{Z}/p^a)$, the natural map $R\Gamma(X,\mathscr{F}) \to R\Gamma(X_{C'}, \mathscr{F}_{C'})$ is an isomorphism.
\end{lemma}

\begin{proof} By an easy induction on $a$ and Proposition \ref{Dbzcgenerators}, we can assume that $\mathscr{F}=f_{\ast}\mathbf{F}_p$ for some finite map $f:X'\to X$. Replacing $X$ by $X'$, we can assume further that $\mathscr{F} = \mathbf{F}_p$ is constant. By two applications of the primitive comparison theorem, it's enough to check that the natural map \[ R\Gamma(X,\mathcal{O}^+_{X}/p) \otimes_{\mathcal{O}_C / p} \mathcal{O}_{C'} / p \to R\Gamma(X_{C'}, \mathcal{O}^+_{X_{C'}}/p)\] is an almost isomorphism. This can be deduced from a purely local statement: if $U=\Spa A / C$ is any affinoid, then the natural map \[R\Gamma(U,\mathcal{O}^+_{U}/p) \otimes_{\mathcal{O}_C / p} \mathcal{O}_{C'} / p \to R\Gamma(U_{C'}, \mathcal{O}^+_{U_{C'}}/p)\] is an almost isomorphism. 

To prove the local statement, choose a perfectoid $\mathbf{Z}_{p}^{d}$-torsor $A \to A_\infty$. Then \[ R\Gamma(U,\mathcal{O}^+_{U}/p) \cong^{a} R\Gamma_{cts}(\mathbf{Z}_{p}^{d},A_{\infty}^{\circ} / p),\] and similarly for $U_{C'}$.  The result now follows by writing $R\Gamma_{cts}(\mathbf{Z}_{p}^{d},-)$ as the usual $d+1$-term Koszul complex and observing that the natural map $A_{\infty}^{\circ} \widehat{\otimes}_{\mathcal{O}_C} \mathcal{O}_{C'} \to A_{C',\infty}^{\circ}$ is an almost isomorphism (e.g., because both sides provide perfectoid rings of definition for the Tate ring $A_\infty \widehat{\otimes}_C C'$).
\end{proof}

Secondly, with biduality in hand, we can somewhat extend our results on pushfoward.

\begin{proposition}\label{morepushforwards} Let $f:X \to Y$ be a Zariski-compactifiable morphism, and let $\mathscr{F} \in D_{zc}^{(b)}(X,\mathbf{Z}/n)$ be an object such that one of the following holds true:

\begin{enumerate}
\item $\mathscr{F}$ is lisse or $\mathbf{D}_X \mathscr{F}$ is lisse, or

\item  $\mathscr{F} = j^{\ast}\mathscr{F}'$ for some compactification $X\overset{j}{\to} X' \overset{\overline{f}}{\to} Y$ and some $\mathscr{F}' \in D_{zc}^{(b)}(X',\mathbf{Z}/n)$. 
\end{enumerate}

Then $Rf_\ast \mathscr{F}$ and $Rf_!\mathscr{F}$ lie in $D^{(b)}_{zc}(Y,\mathbf{Z}/n)$.
\end{proposition}

It follows from the above proposition that if $f$ is a Zariski locally closed immersion, any $\mathscr{F}$ satisfying (1) automatically satisfies (2): we may simply take $\mathscr{F}' = Rf_* \mathscr{F}$ by the proposition.

\begin{proof}For the first case, the result is already proved for $\mathscr{F}$ lisse. Suppose now that $\mathbf{D}_X \mathscr{F}$ is lisse, and choose a compactification $X\overset{j}{\to} X' \overset{\overline{f}}{\to} Y$. We first show that $Rj_\ast \mathscr{F}$ is Zariski-constructible. To see this, use Theorem~\ref{dualizingcomplex} to write
\[Rj_\ast \mathscr{F} \cong Rj_\ast \mathbf{D}_X \mathbf{D}_X \mathscr{F} \cong \mathbf{D}_{X'} j_! \mathbf{D}_X \mathscr{F}.\] 
Then $\mathbf{D}_X \mathscr{F}$ is lisse by assumption, so $j_!\mathbf{D}_X \mathscr{F}$ is Zariski-constructible by Corollary~\ref{openpushforward}, and then duality preserves Zariski-constructibility. Now the triangle $j_!\mathscr{F} \to Rj_\ast \mathscr{F} \to i_{\ast} i^\ast Rj_\ast \mathscr{F} \to $ shows that $j_!\mathscr{F}$ is also Zariski-constructible. Applying $R\overline{f}_\ast$ now gives the claim.

For the second case, let $i:Z \to X'$ be the complementary closed immersion. Then we get a triangle $Rf_!\mathscr{F} \to R\overline{f}_{\ast}\mathscr{F}' \to R(\overline{f} \circ i)_{\ast} i^{\ast} \mathscr{F}' \to$, and the second and third terms are Zariski-constructible by Theorem \ref{properdirectimage}. Likewise, we get a triangle $R(\overline{f} \circ i)_{\ast} Ri^{!} \mathscr{F}' \to R\overline{f}_{\ast}\mathscr{F}' \to Rf_{\ast}\mathscr{F} \to$, and the first two terms are Zariski-constructible by Theorem \ref{properdirectimage} and Proposition \ref{finiteshriekpullback}.
\end{proof}

\begin{remark}[Unbounded variants]
We briefly discuss without proof how the results discussed above extend to the unbounded derived category. We shall need the following notion:

\begin{definition}
Given a non-archimedean base field $K$ and the coefficient ring $\mathbf{Z}/n$, we say $(\dagger)$ holds if $\mathrm{Gal}(\overline{K}/K)$ has finite $\ell$-cohomological dimension for all primes $\ell | n$.
\end{definition}
This condition is very mild, and holds for example if $K$ is separably closed, or if $K$ is a local field, or if $K$ has separably closed residue field and $(n,p)=1$. One can also check that $(\dagger)$ is stable under replacing $K$ by $K_x$, where $K_x$ is the residue field of any rigid space $X/K$ at any (adic) point $x \in X$. The main reason for introducing this condition is the following:

\begin{proposition}
\label{daggergood}
Fix $K$ and $\Lambda=\mathbf{Z}/n$ such that  $(\dagger)$ holds. Then for any rigid space $X/K$, the derived category $D(X_{\mathrm{\acute{e}t}},\Lambda)$ is left-complete and compactly generated, and the functor $R \lim :  D(X_{\mathrm{\acute{e}t}}^{\mathbf{N}},\Lambda) \to  D(X_{\mathrm{\acute{e}t}},\Lambda)$ has bounded cohomological amplitude on any finite-dimensional open subspace of $X$.
\end{proposition}
\begin{proof} This is well-known; cf. \cite{Roos} for the final statement. 
\end{proof}

Let us now formulate the promised unbounded variants of the results discussed in this paper. Let $f:X \to Y$ be a map of rigid spaces over $K$, and let $\mathscr{F} \in D_{zc}(X,\mathbf{Z}/n)$. 
\begin{enumerate}
\item Pullback: The pullback $f^*$ takes $D_{zc}$ into $D_{zc}$.

\item Proper pushforward: If $\mathscr{F}$ is bounded below or $(\dagger)$ holds, then $Rf_* \mathscr{F} \in D_{zc}$.

\item General pushforward: Say $\mathscr{F}$ is lisse and $f$ is Zariski-compactifiable. If $\mathscr{F}$ is bounded below or $(\dagger)$ holds, then $Rf_! \mathscr{F}, Rf_* \mathscr{F} \in D_{zc}$.

\item Duality: The functor $\mathbf{D}_X(-) = \inthom(-,\omega_X)$ carries $D^{(-)}_{zc}$ into $D^{(+)}_{zc}$. Furthermore, if $(\dagger)$ holds, then $\mathbf{D}_X(-)$ also carries $D^{(+)}_{zc}$ into $D^{(-)}_{zc}$, and gives an autoequivalence of $D_{zc}(X,\mathbf{Z}/n)$ satisfying biduality.

\item $!$-pullback: If $f$ is a finite map, then $Rf^!$ takes $D_{zc}$ to $D_{zc}$. 

\item $!$-pullback for good coefficients: If $(p,n) = 1$, $f$ is any taut separated map and $(\dagger)$ holds true, then $Rf^!$ preserves $D_{zc}$.

\item Tensor product: $D^{(-)}_{zc}(X,\mathbf{Z}/n)$ is stable under $\otimes$ inside $D(X,\mathbf{Z}/n)$. 

\item Internal Hom: The bifunctor $\inthom(-,-)$ carries $D^{(-)}_{zc} \times D^{(+)}_{zc}$ into $D^{(+)}_{zc}$.

\end{enumerate}

These assertions are all proven by using homological arguments to reduce to the locally bounded case (using Proposition~\ref{daggergood} as needed). We omit the proofs.  The most notably difficult case is (4), where the reduction to the bounded situation is non-formal, and requires the following lemma.

\begin{lemma}\label{affinoidduallemma}Let $X$ be a $d$-dimensional rigid space. Then for any $\mathscr{F} \in \mathrm{Sh}_{zc}(X,\mathbf{Z}/n)$, the complex $\inthom(\mathscr{F},\omega_X)$ is concentrated in degrees $[-2d,0]$.
\end{lemma}

In particular, for any finite-dimensional $X$, the functor $\mathbf{D}_X(-)$ preserves $D^b$ without assuming that $X$ is quasi-compact.

\begin{proof}
One could deduce this from the last assertion in \cite[Theorem XVII.5.1.1]{ILO14} using \cite[\S 3]{GabberNotest} as well as our algebraization Proposition~\ref{schemescomparison1}. For the convenience of the reader, we give a direct proof, essentially mimicing the argument in Corollary~\ref{inthomexists} when $\mathscr{G} = \omega_X$, while controlling the amplitude of the terms showing up. 

Without loss of generality, we may assume $X$ is affinoid and reduced. We proceed induction on $d$, the $d=0$ case being trivial. Let $j:U \to X$ be a smooth Zariski-open subset of pure dimension $d$ with complement of dimension $<d$ such that $\mathscr{F}|U$ is locally constant. Let $i:Z=\Spa B \to X$ be the inclusion of the closed complement of $U$. Applying $\inthom(-,\omega_X)$ to the triangle $j_! j^\ast \mathscr{F} \to \mathscr{F} \to i_{\ast} i^\ast \mathscr{F} \to$, we get a triangle 
\[ i_{\ast} \inthom(i^{\ast} \mathscr{F},\omega_Z) \to \inthom(\mathscr{F},\omega_X) \to (Rj_{\ast} j^\ast \mathscr{F}^\vee)[2d](d) \to,\]
where we used the isomorphism $\omega_U \simeq \mathbf{Z}/n[2d](d)$ (coming from the smoothness of $U$) to simplify the last term, and Theorem~\ref{dualizingcomplex} (4) for the first term. By induction, the first term here is concentrated in degrees $[-2d+2,0]$. It thus suffices to show that $Rj_{\ast} j^\ast \mathscr{F}^\vee$ is concentrated in degrees $[0,2d-1]$. Proposition~\ref{schemescomparison1} identifies $Rj_{\ast} j^\ast \mathscr{F}^\vee$ as the analytification of $Rj^{alg}_{\ast} \mathcal{G}$, where $j^{alg}: \mathcal{U}=\Spec A - \Spec B \to \mathcal{X}=\Spec A$ is the natural algebraization of $j$ and $\mathcal{G}$ is some locally constant sheaf on $\mathcal{U}$. By \cite[Exp. XVIII-A, Theorem 1.1]{ILO14}, $Rj^{alg}_{\ast} \mathcal{G}$ is concentrated in degrees $[0,2d-1]$, so the result follows.
\end{proof}

\end{remark}

\subsection{Adic coefficients}
\label{ss:adic}
We fix a characteristic zero nonarchimedean base field $K$ of residue characteristic $p > 0$. Fix any prime $\ell$. In this section we explain a variant of the theory of Zariski-constructible sheaves with $\mathbf{Z}_\ell$-coefficients, using the formalism\footnote{The paper \cite{Sch} assumes that a prime number is topologically nilpotent in the base field $K$. This is the reason we assume that the residue characteristic $p$ of $K$ is $> 0$. Since we only use the relatively formal aspects of \cite{Sch}, we expect that this assumption can be removed once a theory of adic coefficients over nonarchimedean base fields of residue characteristic $0$ has been developed.} from \cite{Sch}.

For a rigid space $X/K$, let $X_v$ denote the $v$-site of $X$ from \cite{Sch}. Let ${\mathbf{Z}_\ell} = \lim_n \underline{\mathbf{Z}/\ell^n}$ be the displayed inverse limit of constant sheaves, regarded as a sheaf of (abstract) rings. Any perfect complex $M \in D(\mathbf{Z}_\ell)$ yields a ``constant'' $\ell$-complete sheaf $\underline{M} := \lim_n \underline{M/\ell^n}\in D(X_v, \mathbf{Z}_\ell)$. Our goal is to build a theory of Zariski-constructible $\mathbf{Z}_\ell$-complexes on a rigid space $X$ where the locally constant objects are twisted forms of $\underline{M}$ for $M \in D_{perf}(\mathbf{Z}_\ell)$. We first recall the basic notion of ``\'etale $\mathbf{Z}_\ell$-sheaves'' that is introduced in \cite{Sch} for the purposes of defining operations.

\begin{construction}
Fix $n \geq 1$. For any strictly totally disconnected perfectoid space $Y$, the usual derived category $D(Y_{\et},\mathbf{Z}/\ell^n)$ is left-complete and identifies with a full subcategory $D(Y_v,\mathbf{Z}/\ell^n)$ via pullback along $Y_v \to Y_{\et}$. Moreover, containment in this subcategory can be checked $v$-locally (see \cite[Proposition 14.10, Proposition 14.11 (iii), Theorem 14.12 (ii)]{Sch}). 

For any $v$-stack $X$, let $D_{\et}(X,\mathbf{Z}/\ell^n) \subset D(X_v, \mathbf{Z}/\ell^n)$ be the full subcategory spanned by objects whose pullback along any map $Y \to X$ with $Y$ a strictly totally disconnected perfectoid space lies in the subcategory $D(Y_{\et}, \mathbf{Z}/\ell^n) \subset D(Y_v,\mathbf{Z}/\ell^n)$; this condition can be checked after pullback along a $v$-cover (\cite[Remark 14.14]{Sch}. Imposing this condition for a complex $K$ is equivalent to imposing it for each $v$-cohomology sheaf $\mathcal{H}^i(K)$ (regarded as a complex) (\cite[Proposition 14.16]{Sch}). Moreover, if $X$ is a locally spatial diamond (e.g., one attached to a rigid space), then $D_{\et}(X,\mathbf{Z}/\ell^n)$ admits a classical description: it agrees with the left-completion $\widehat{D}(X_{et},\mathbf{Z}/\ell^n)$ of the usual derived category $D(X_{\et},\mathbf{Z}/\ell^n)$ (\cite[Proposition 14.15]{Sch}).  

Finally, for any $v$-stack $X$, write $D_{\et}(X,\mathbf{Z}_\ell) \subset D_{\ell-\text{comp}}(X_v,\mathbf{Z}_\ell)$ as the full subcategory derived $\ell$-complete objects in $D(X_v,\mathbf{Z}_\ell)$ whose mod $\ell$-reduction lies in the subcategory $D_{\et}(X,\mathbf{Z}/\ell)$ mentioned above (\cite[Definition 26]{Sch}). If $X$ is a locally spatial diamond, then derived $\ell$-completeness as well as the previous remark on left-completeness give an equivalence 
\begin{equation}
\label{DetLimit}
\mathcal{D}_{\et}(X,\mathbf{Z}_\ell) = \lim_n \mathcal{D}_{\et}(X,\mathbf{Z}/\ell^n) \simeq \lim_n \widehat{\mathcal{D}}(X_{\et},\mathbf{Z}/\ell^n)
\end{equation}
at the level of the corresponding $\infty$-categories, thus giving a classical description of the left side in the case of rigid spaces, see \cite[Proposition 26.2]{Sch}. In fact, we could have defined $D_{\et}(X,\mathbf{Z}_\ell)$ as the homotopy category of the right side above, and thus avoided ever mentioning the ambient category $D_{\ell-\text{comp}}(X_v,\mathbf{Z}_\ell)$; one reason we introduce the latter is that it carries an obvious $t$-structure, which we shall use in our proofs.

As in \cite[\S 26]{Sch}, all operations between the categories $D_{\et}(-,\mathbf{Z}_\ell)$ of $\ell$-adic complexes introduced above are always interpreted in the $\ell$-completed sense, i.e., the functor in question takes values in $\ell$-complete complexes by fiat, and agrees after reduction mod $\ell$ with the corresponding functor for finite coefficients. For instance, if $j:U \to X$ is a Zariski-open immersion and $M \in D^b_{\et}(U,\mathbf{Z}_\ell) \subset D^b_{\ell-\text{comp}}(U_v,\mathbf{Z}_\ell)$, then $j_! M \in D^b_{\ell-\text{comp}}(X_v,\mathbf{Z}_\ell)$ is defined to be the derived $\ell$-completion of $j^{top}_! M \in D^b(X_v,\mathbf{Z}_\ell)$ where $j^{top}_!$ denotes the topos theoretic $!$-extension (without any completions); then one can see that $j_! M$ lies in $D^b_{\et}(X,\mathbf{Z}_\ell)$ and $j_! M \otimes_{\mathbf{Z}_\ell} \mathbf{F}_\ell$ agrees with $j_!(M \otimes_{\mathbf{Z}_\ell} \mathbf{F}_\ell)$, where the latter is defined in the classical way.
\end{construction}

In the above setting, we can introduce Zariski-constructible complexes:

\begin{definition}
Let $X$ be a rigid space. We define full subcategories
\[ D^{(b)}_{lis}(X,\mathbf{Z}_\ell) \subset D^{(b)}_{zc}(X,\mathbf{Z}_\ell) \subset D^{(b)}_{\et}(X,\mathbf{Z}_\ell) \subset D^{(b)}_{\ell-\text{comp}}(X_v,\mathbf{Z}_\ell)\]
as follows: 
\begin{itemize}
\item An object $K \in D^{(b)}_{\et}(X,\mathbf{Z}_\ell)$ lies in $D^{(b)}_{lis}(X,\mathbf{Z}_\ell)$ (and is called {\em lisse}) if $K/\ell \in D^{(b)}(X_{\et},\mathbf{Z}/\ell)$ is lisse in our previous sense (Definition~\ref{zcdefinition}).
\item An object $K \in D^{(b)}_{\et}(X,\mathbf{Z}_\ell)$ lies in $D^{(b)}_{zc}(X,\mathbf{Z}_\ell)$ (and is called {\em Zariski-constructible}) if $K/\ell \in D^{(b)}(X_{\et},\mathbf{Z}/\ell)$ has Zariski-constructible cohomology sheaves.
\end{itemize}
As before, we write $\mathcal{D}^{(b)}_{lis}(X,\mathbf{Z}_\ell)$ and $\mathcal{D}^{(b)}_{zc}(X,\mathbf{Z}_\ell)$ for the corresponding full $\infty$-categories inside $\mathcal{D}_{\et}(X,\mathbf{Z}_\ell)$. 
\end{definition}

\begin{remark}
\label{ZCZellLimit}
Let us explain an inverse limit description of $\mathcal{D}^{(b)}_{zc}(X,\mathbf{Z}_\ell)$, similarly to \eqref{DetLimit}. For each $n \geq 1$, let $\mathcal{D}^{(b)}_{zc,\ell-\text{ftd}}(X,\mathbf{Z}/\ell^n) \subset \mathcal{D}^{(b)}_{zc}(X,\mathbf{Z}/\ell^n)$ be the full subcategory spanned by objects $M$ such that $M \otimes_{\mathbf{Z}/\ell^n}^L \mathbf{Z}/\ell$ is locally bounded. These are compatible under the base change functors changing $n$. We claim that the equivalence in \eqref{DetLimit} restricts to an equivalence
\[ \mathcal{D}^{(b)}_{zc}(X,\mathbf{Z}_\ell) \simeq \lim_n \mathcal{D}^{(b)}_{zc,\ell-\text{ftd}}(X,\mathbf{Z}/\ell^n).\]
Indeed, since $\mathbf{Z}_\ell$ has global dimension $1$, it is clear that the equivalence in \eqref{DetLimit} gives a fully faithful functor from the left to the right. The essential surjectivity follows by observing that, under the equivalence in \eqref{DetLimit}, the condition that $M \in \mathcal{D}_{\et}(X,\mathbf{Z}_\ell)$ lies inside $\mathcal{D}^{(b)}_{zc}(X,\mathbf{Z}_\ell)$ can be checked after reduction mod $\ell$. 
\end{remark}

As both locally constant sheaves and Zariski-constructible sheaves with $\mathbf{Z}/\ell$-coefficients form weak Serre subcategories of the category of $\mathbf{Z}/\ell$-sheaves on $X_{\et}$, both categories introduced above form triangulated subcategories of $D^{(b)}_{\et}(X,\mathbf{Z}_\ell)$. These categories admit an algebraic description on affinoids:

\begin{lemma}
\label{ladicschemecomp}
Let $X=\mathrm{Spa}(A)$ be an affinoid rigid space with the natural algebraization $\mathcal{X} = \mathrm{Spec}(A)$. The pullback along $X \to \mathcal{X}$ induces equivalences
\[ D^b_{lis}(\mathcal{X},\mathbf{Z}_\ell) \simeq D^{(b)}_{lis}(X,\mathbf{Z}_\ell) \quad \text{and} \quad  D^b_{zc}(\mathcal{X},\mathbf{Z}_\ell) \simeq D^{(b)}_{zc}(X,\mathbf{Z}_\ell)\]
of triangulated categories.
\end{lemma}
\begin{proof}
This follows from the description in Remark~\ref{ZCZellLimit} together with Proposition~\ref{schemescomparison1} that implies the corresponding statements with $\mathbf{Z}/\ell^n$-coefficients by passing to the full subcategory of objects with finite Tor dimension.
\end{proof}

Using the aforementioned algebraic description, we can show that local constancy mod $\ell$ implies local constancy, justifying our definition of lisse complexes.

\begin{lemma}
\label{LisseLC}
Let $X$ be a rigid space and let $M \in D^{(b)}_{lis}(X,\mathbf{Z}_\ell)$. Then $M$ is locally constant. More precisely, for any cover $\{U_i\}$ of $X$ by connected affinoids, there exist perfect complexes $N_i \in D_{perf}(\mathbf{Z}_\ell)$ such that $M|_{U_i}$ is locally isomorphic (for the $v$- or in fact even the pro-(finite \'etale) topology of $U_i$) to $\underline{N_i}$. In particular, each $v$-cohomology sheaf $\mathcal{H}^i(M)$ is locally constant as well. 
\end{lemma}

This lemma is analogous to \cite[Remark 6.6.13]{BSProetale}, with Achinger's theorem replacing Artin's theorem.

\begin{proof}
We may assume $X = \mathrm{Spa}(A)$ is a connected affinoid. Proposition~\ref{ladicschemecomp} then implies that $M$ is uniquely pulled back from some $M' \in D^b_{lis}(\mathrm{Spec}(A),\mathbf{Z}_\ell)$. Let $B := \colim_i B_i$ be a universal cover of $A$, i.e., this is filtered colimit of connected finite \'etale covers $A \to B_i$ with $B$ itself being simply connected. Thus, $\mathrm{Spec}(B)$ admits no non-trivial locally constant sheaves of finitely generated $\mathbf{Z}_\ell$-modules. Moreover, Achinger has shown \cite[\S 1.5]{AchingerKpi1} that each $\mathrm{Spec}(B_i)$ is a $K(\pi,1)$, which implies that $R\Gamma(\mathrm{Spec}(B),\mathbf{Z}_\ell) = \mathbf{Z}_\ell$. The combination of these two properties of $\mathrm{Spec}(B)$ implies that taking the ``constant'' sheaf gives an equivalence $D_{perf}(\mathbf{Z}_\ell) \simeq D^b_{lis}(\mathrm{Spec}(B),\mathbf{Z}_\ell)$, so $M'|_{\mathrm{Spec}(B)} \in D_{lis}^b(\mathrm{Spec}(B),\mathbf{Z}_\ell)$ is the ``constant'' $\mathbf{Z}_\ell$-complex attached to a perfect $\mathbf{Z}_\ell$-complex $N$. Analytifying this cover then solves the problem, i.e., taking $Y := \lim_i \mathrm{Spa}(B_i)$ where each $B_i$ is given the natural topology and the inverse limit is computed in $v$-sheaves, we obtain a pro-(finite \'etale) cover $Y \to X$ such that $M|_Y \simeq \underline{N}$ for some $N \in D_{perf}(\mathbf{Z}_\ell)$, as wanted.
\end{proof}

Next, we observe that all operations defined before extend to $\mathbf{Z}_\ell$-sheaves.

\begin{theorem}\label{operationsZell}
On the category of rigid spaces over $K$, the following operations (defined in \cite[\S 26]{Sch}) restrict to operations on $D^{(b)}_{zc}(-,\mathbf{Z}_\ell)$ and are compatible with reduction modulo $\ell^n$.
\begin{enumerate}
\item $f^*$, $\otimes$, and $\inthom$.
\item Verdier duality.
\item $Rf_*$ for $f$ proper.
\item $Rf_!$ and $Rf_*$ on lisse complexes for Zarisk-compactifiable morphisms $f$.
\item $Rf^!$ if either $f$ is a finite morphism or $p \neq \ell$.
\end{enumerate}
Moreover, proper base change holds, and all of these operations are compatible with extensions of the nonarchimedean base field.
\end{theorem}
\begin{proof} 
Let us first define the dualizing complex $\omega_X \in \mathcal{D}^{(b)}_{zc}(X,\mathbf{Z}_\ell)$, thereby defining the operation that is supposed to give Verdier duality. Given a rigid space $X$ and an integer $n \geq 1$, we have constructed in Theorem~\ref{dualizingcomplex} a dualizing complex $\omega_n \in \mathcal{D}^{(b)}_{zc}(X,\mathbf{Z}/\ell^n)$. Given two integers $n \geq m$, we claim that there is a transitive system of isomorphisms $a_{nm}:\omega_{n} \otimes^L_{\mathbf{Z}/\ell^{n}} \mathbf{Z}/\ell^m \simeq \omega_m$ in $D^b_{zc}(X, \mathbf{Z}/\ell^m)$: for $X=\mathrm{Spa}(A)$ being affinoid, this follows by a similar isomorphism for potential dualizing complexes on $\mathrm{Spec}(A)$ (see discussion on potential dualizing complexes following Theorem~\ref{DCExistScheme}, and use the pinning data there to see transitivity), and the general case follows by BBDG glueing (as in the proof of Theorem~\ref{dualizingcomplex}). By canonicity as well as the fact that $\mathrm{Ext}^{< 0}_{\mathbf{Z}/\ell^n}(\omega_n,\omega_n) = 0$ for all $n \geq 1$, the system $\{\omega_n\}$ lifts naturally to an object of the $\infty$-category $\lim_n \mathcal{D}^{(b)}_{zc,\ell-\text{ftd}}(X,\mathbf{Z}/\ell^n)$ from Remark~\ref{ZCZellLimit}. Using the equivalence there, the inverse limit $\omega_X := \lim_n \omega_n \in \mathcal{D}(X_v,\mathbf{Z}_\ell)$ then lies in $\mathcal{D}^{(b)}_{zc}(X,\mathbf{Z}_\ell)$; this  object comes equipped with a transitive system of isomorphisms $\omega_X \otimes_{\mathbf{Z}_\ell}^L \mathbf{Z}/\ell^n \simeq \omega_n$, thus providing our candidate dualizing complex $\omega_X$.

All the operations are now defined on the larger category $\mathcal{D}_{\et}(X,\mathbf{Z}_\ell)$, and are compatible with reduction mod $\ell$; the claims in the proposition now follow from the analogous statements mod $\ell$.
\end{proof}

\begin{remark}[Relating Verdier duality with finite and $\mathbf{Z}_\ell$-coefficients]
\label{VDZellZmodell}
For any $n \geq 1$, the reduction modulo $\ell^n$-functor $D^{(b)}_{zc}(X,\mathbf{Z}_\ell) \to D^{(b)}_{zc}(X,\mathbf{Z}_\ell/\ell^n)$ carries the dualizing complex $\omega_{X,\mathbf{Z}_\ell} := \omega_X \in D^{(b)}_{zc}(X,\mathbf{Z}_\ell)$ constructed in Proposition~\ref{operationsZell} to the dualizing complex $\omega_{X,\mathbf{Z}/\ell^n} := \omega_X \in D^{(b)}_{zc}(X,\mathbf{Z}_\ell)$ from Theorem~\ref{dualizingcomplex} (1), which gives the formula $\mathbf{D}_{X,\mathbf{Z}_\ell}(-)/\ell^n \simeq \mathbf{D}_{X,\mathbf{Z}/\ell^n}(-/\ell^n)$ relating the Verdier duality operations under reduction modulo $\ell^n$. Moreove, using the formula $\mathrm{RHom}_{\mathbf{Z}_\ell}(\mathbf{Z}/\ell^n,\mathbf{Z}_\ell) = \mathbf{Z}/\ell^n[-1]$, it follows that the restriction of scalars functor $\mathrm{Res}:D^{(b)}_{zc}(X,\mathbf{Z}_\ell/\ell^n) \to  D^{(b)}_{zc}(X,\mathbf{Z}_\ell)$ satisfies  $\mathbf{D}_{X,\mathbf{Z}_\ell} \circ \mathrm{Res} = \mathrm{Res} \circ \mathbf{D}_{X,\mathbf{Z}/\ell^n}[-1]$.
\end{remark}

\begin{remark}
\label{Proetalevsv}
In the entire discussion in this section, we could have used the pro-\'etale topology from \cite{Sch} instead of the $v$-topology without any modification: this follows from the full faithfulness results in \cite[\S 14]{Sch} for the ``change of topology map'' and the observation that any $\mathscr{F} \in \mathrm{Sh}_{lis}(X, \mathbf{Z}_\ell)$ as defined above is in fact locally constant in the pro-\'etale topology by Lemma~\ref{LisseLC}. Nevertheless, we have preferred to formulate things using the $v$-topology since the operations defined in \cite[\S 26]{Sch} are defined using the $v$-topology. 
\end{remark}

As our final goal in this section, we define the ``standard'' or ``constructible'' $t$-structure on $D^{(b)}_{zc}(X,\mathbf{Z}_\ell)$. We first explain how to do the analogous construction in algebraic geometry; we use the pro-\'etale approach from \cite{BSProetale}, but a closely related result can be found in \cite[Theorem 3.6 (v)]{EkedahlAdic}.

\begin{proposition}[The constructible $t$-structure for $\mathbf{Z}_\ell$-sheaves on a noetherian scheme]
\label{Conststrscheme}
Let $Y$ be a noetherian scheme. Then the standard $t$-structure on $D(Y_{proet},\mathbf{Z}_\ell)$ restricts to one on $D^b_{cons}(Y_{proet},\mathbf{Z}_\ell)$. 
\end{proposition}

In the statement above and the proof below, we use the notions from \cite[\S 5, 6]{BSProetale}. In particular, we refer to an object of $D^b(Y_{proet})$ as {\em classical}  if it is in the essential image of the (fully faithful) pullback along $\nu:Y_{proet} \to Y_{et}$ (see \cite[\S 5.1]{BSProetale}). Classical  abelian sheaves on $Y_{proet}$ are thus equivalent to abelian sheaves on $Y_{et}$ and form an abelian Serre subcategory of all abelian sheaves on $Y_{proet}$.

\begin{proof}
Given $M \in D^b_{cons}(Y_{proet},\mathbf{Z}_\ell)$, we must show that each $\mathcal{H}^i(M)$ lies in $D^b_{cons}(Y_{proet},\mathbf{Z}_\ell)$. Using the definition of constructibility \cite[\S 5]{BSProetale}, we must show that the abelian pro-\'etale sheaves $\mathcal{H}^i(M)/\ell$ and $\mathcal{H}^i(M)[\ell]$ are constructible $\mathbf{F}_\ell$-sheaves for all $i$. Note that these sheaves can be regarded as subobjects (resp. quotient objects) of some $\mathcal{H}^i(M/\ell)$ via the Bockstein sequence for $\ell$. As \'etale subquotients of \'etale constructible constructible sheaves on a noetherian scheme are constructible \cite[Tag 09BH]{Stacks}, it suffices to show that the pro-\'etale sheaves $\mathcal{H}^i(M)/\ell$ and $\mathcal{H}^i(M)[\ell]$ are classical. In fact, by the Bockstein sequence and stability of classical sheaves under cokernels in all pro-\'etale sheaves, it suffices to prove that each $\mathcal{H}^i(M)/\ell$ is \'etale. By definition of constructibility, we know that  $\mathcal{H}^i(M/\ell^n)$ is classical for all $i$ and $n$. As classical sheaves are stable under images, it is then enough to show that $\mathcal{H}^i(M)/\ell \subset \mathcal{H}^i(M/\ell)$ is exactly the image of $\mathcal{H}^i(M/\ell^n) \to \mathcal{H}^i(M/\ell)$ for $n \gg 0$. By the Bockstein sequences for $\ell^n$, this would follow if we knew that the projective system $\{\mathcal{H}^i(M)[\ell^n]\}_{n \geq 1}$ are Mittag-Leffler for each $i$, i.e., if each $\mathcal{H}^i(M)$ had bounded $\ell$-power torsion. If $M$ is lisse, this is clear. In general, recall the following fact from \cite[\S 6.2]{BSProetale}: if $k:Z \to Y$ is a (necessarily constructible, as $Y$ is noetherian) locally closed immersion, then the functors $k^*$ and $k_!$ on the derived category of all pro-\'etale sheaves preserve limits and colimits and commute with $\mathcal{H}^i(-)$. By \cite[Proposition 6.6.11]{BSProetale}, we know that $Y$ admits a finite stratification $\{k_j:Y_j \to Y\}$ such that each $N_j := k_j^* M$ is lisse. The aforementioned properties of $k_{j,!}$ and $k_j^*$ then show that $\mathcal{H}^i(M)[\ell^n]$ admits a finite filtration whose graded pieces have the form $k_{j,!} k_j^* (\mathcal{H}^i(N_j)[\ell^n])$ for lisse complexes $N_i$. But then each $\mathcal{H}^i(N_j)$ is also lisse and hence has bounded $\ell$-power torsion, so the corresponding claim for $\mathcal{H}^i(M)$ follows by devissage.
\end{proof}

\begin{theorem}[The constructible $t$-structure for $\mathbf{Z}_\ell$-sheaves on a rigid space]
Let $X/K$ be a rigid space. Then there exists a natural ``constructible'' $t$-structure $({}^c D^{\leq 0}_{zc}(X,\mathbf{Z}_\ell), {}^c D^{\geq 0}_{zc}(X,\mathbf{Z}_\ell))$ on $D^{(b)}_{zc}(X,\mathbf{Z}_\ell)$ with the following properties:
\begin{enumerate}
\item An object $K$ lies in ${}^c D^{\leq 0}_{zc}(X,\mathbf{Z}_\ell)$ if and only if $K/\ell \in D^{\leq 0}(X,\mathbf{Z}/\ell)$.
\item An object $K$ lies in the heart if and only if there exists a locally finite stratification $X = \{X_i\}$ by Zariski locally closed subsets such that $K|_{X_i}$ is locally constant and concentrated in degree $0$ in the obvious sense (i.e., isomorphic locally on $X_{i,v}$ to an object of the form $\underline{N}$ with $N$ a finitely generated $\mathbf{Z}_\ell$-module).
\end{enumerate}

The restrictions appearing in part (2) above are in the sense of the operations in Proposition~\ref{operationsZell} (see also Remark~\ref{Proetalevsv}).

\end{theorem}
\begin{proof}
First assume $X=\mathrm{Spa}(A)$ is affinoid. In this case, to obtain a $t$-structure by the description in (1), we may use Lemma~\ref{ladicschemecomp} to translate to a similar question on $Y=\mathrm{Spec}(A)$. Thus, it suffices to show that the $t$-structure on $D^b_{cons}(Y_{proet},\mathbf{Z}_\ell)$ constructed in Proposition~\ref{Conststrscheme} is characterized by the property that $K \in D^{\leq 0}(Y_{proet})$ exactly when $K/\ell \in D^{\leq 0}(Y_{et})$; this follows from repleteness of $Y_{proet}$, exactness and full faithfulness of pullback along $Y_{proet} \to Y_{et}$, and standard facts on derived completions. Moreover, part (2) also follows from the reasoning at the end of the proof of Proposition~\ref{Conststrscheme} as well as the fact that Zariski-closed subsets of $\mathrm{Spa}(A)$ are the same as closed subsets of $\mathrm{Spec}(A)$. 

For future reference, still in the affinoid case, we remark that once we know (2) is satisfied for some stratification, there is in fact {\em canonical} stratification where (2) is satisfied. Indeed, if we take the open stratum $X_0 \subset X$ to be the maximal Zariski dense open provided by Proposition~\ref{MaxLisseOpen} for $\mathcal{H}^*(K/\ell)$ and continue inductively, we obtain a stratification $\{X_i\}_{i \geq 0}$ of $X$ by Zariski locally closed subsets such that $K|_{X_i}$ is lisse by Lemma~\ref{LisseLC}. To check that $K|_{X_i}$ is concentrated in degree $0$ in the sense of (2), we may refine the canonical stratification to ensure it is finer than a given stratification witnessing the property in (2), take stalks,  and then deduce the result for the canonical stratification itself. We observe also that this canonical stratification has the feature that it is compatible with restricting to smaller affinoids by Lemma~\ref{MaxLisseOpenBC}.

We now deduce the general case by glueing. Indeed, first observe that the pullback along maps of affinoids is $t$-exact with respect to the $t$-structure we constructed in the first paragraph: right $t$-exactness is clear from the description in (1), while left $t$-exactness follows from the description of the heart in (2) and the boundedness of the $t$-structure on affinoids. As the condition appearing in part (1) is of a local nature, it follows that for any rigid space $X$, we can glue the $t$-structures defined above on the affinoid opens of $X$ to produce a $t$-structure on $D^{(b)}_{zc}(X,\mathbf{Z}_\ell)$ satisfying part (1). For part (2), thanks to the last sentence of the previous paragraph, we may simply glue together the canonical stratification on affinoids constructed in the previous paragraph to obtain the desired stratification.
\end{proof}

\section{Perverse sheaves}

In this section, we use the results of \S \ref{sec:SixFun} to define a notion of perverse sheaves on rigid spaces over nonarchimedean base field $K$ of  characteristic $0$. Our results with finite coefficents work for any rigid space, with $\mathbf{Z}_\ell$ coefficients when $K$ has residue characteristic $p$ (due to the corresponding requirement in \S \ref{ss:adic}), and $\mathbf{Q}_\ell$-coefficients when one further restricts to the qcqs case.

\subsection{Finite coefficients}
\label{ss:PervFin}
Let $K$ be a nonarchimedean field of characteristic $0$ and let $X/K$ be a rigid space. In this subsection, we use $\mathbf{Z}/n$-coefficients for some $n \geq 1$. In this section, we develop a theory of perverse sheaves on $X$ that enjoys the same pleasant formal properties as its counterpart in algebraic geometry \cite{BBDG}.

\begin{definition} \label{perverse} Let $X/K$ be a rigid space. 
\begin{enumerate}
\item Define $\pervleqzero(X) \subset D_{zc}^{(b)}(X)$ as the full subcategory of complexes $\mathscr{F}$ such that $ \dim \mathrm{supp} \mathcal{H}^j(\mathscr{F}) \leq -j$ for all $j \in \mathbf{Z}$. 

\item Define $\pervgeqzero(X) \subset D_{zc}^{(b)}(X)$ as the full subcategory of complexes $\mathscr{F}$ such that $\mathbf{D}_X(\mathscr{F}) \in \pervleqzero(X)$.
\end{enumerate}
\end{definition}

The main results about this definition are summarized as follows.

\begin{theorem}[Properties of the perverse $t$-structure]
\label{perverseproperties} 
In the setup above, we have the following.
\begin{enumerate}
\item The pair $(\pervleqzero(X),\pervgeqzero(X))$ define a t-structure on $D_{zc}^{(b)}(X)$. Write $\Perv(X)=\Perv(X,\mathbf{Z}/n)$ for the heart of this $t$-structure, and write $\pervcoh^n : D^{(b)}_{zc}(X) \to \Perv(X)$ for the associated cohomology functors.

\item For a Zariski-open immersion $j$ (resp. Zariski-closed immersion $i$), we have the following exactness properties with respect to the perverse $t$-structure:
\begin{enumerate}
\item $j^*$ and $i_*$ are $t$-exact.
\item $j_!$ is right $t$-exact in the context of Proposition~\ref{morepushforwards} (2), i.e., if $\mathscr{F} \in \pervleqzero(U,\mathbf{Z}/n)$ arises as the pullback of some object from $D^{(b)}_{zc}(X,\mathbf{Z}/n)$, then $j_! \mathscr{F} \in \pervleqzero(X,\mathbf{Z}/n)$.
\item $i^*$ is right $t$-exact and $Ri^!$ is left $t$-exact.
\item $Rj_*$ is left $t$-exact in the context of Proposition~\ref{morepushforwards} (2).
\end{enumerate}

\item $\Perv(X)$ is stable under Verdier duality. 

\item If $X = \mathcal{X}^{an}$ for a finite type $K$-scheme $\mathcal{X}$, the functor $D^b_c(\mathcal{X}) \to D^b_{zc}(X)$ induces a fully faithful functor $\Perv(\mathcal{X}) \to \Perv(X)$. If $\mathcal{X}$ is proper over $\Spec K$, this functor is an equivalence of categories.

\item Say $j:U \subset X$ is the inclusion of any Zariski locally closed subset and $\mathscr{L}$ is a perverse sheaf on $U$ that admits an extension to $D^{(b)}_{zc}(X,\mathbf{Z}/n)$ under $j^*$ (e.g., if one of $\mathscr{L}$ or $\mathbf{D}_{U}(\mathscr{L})$ is lisse, see Proposition~\ref{morepushforwards}). Then there is a naturally associated intermediate extension $j_{! \ast}\mathscr{L} \in \Perv(X)$ such that $j^{\ast}j_{! \ast} \mathscr{L}\cong \mathscr{L}$. Moreover, $\mathbf{D}_X(j_{! \ast}\mathscr{L})\cong j_{! \ast}\mathbf{D}_{U}(\mathscr{L})$. 

\item If $X$ is quasicompact, $\Perv(X)$ is Noetherian and Artinian. The simple objects have the form $j_{!*} (\mathscr{L}[d])$, where $j:U \to X$ is a Zariski-locally closed immersion with $U$ smooth of dimension $d$ and $\mathscr{L}$ is a simple locally constant sheaf on $U$. 

\item Perversity is stable under pushforward along finite morphisms. 

\item Assume $p$ is invertible on $\Lambda$. If $K$ is algebraically closed and $\mathfrak{X}$ is a formal model of $X$ with special fiber $\mathfrak{X}_s$, the nearby cycles functor $R\lambda_{\mathfrak{X} \ast}: D^{b}_{zc}(X) \to D^{b}_{c}(\mathfrak{X}_s)$ is t-exact for the perverse t-structures.

\end{enumerate}

\end{theorem}

We expect that the $t$-exactness in (8) holds true without the assumption on $p$ (using the perverse $t$-structure on the target constructed in \cite{GabberNotest}). The right $t$-exactness ought to follow from the relevant affinoid vanishing theorem, generalizing \cite{BM, H}, that has been announced by Gabber.

\begin{proof}
\begin{enumerate}
\item We give two proofs: one via localizing to \cite{GabberNotest}, and one via a direct argument.

{\em Proof via \cite{GabberNotest}:} We have seen before that $X \mapsto \mathcal{D}^{(b)}_{zc}(X)$ is a stack for the analytic topology on $X$. Moreover, pullback along open inclusions $U \subset X$ of rigid spaces preserves $\pervleqzero(-)$ by definition, and $\pervgeqzero(-)$ as Verdier duality localizes. Consequently, these pullbacks are perverse $t$-exact once we know the perverse $t$-structure exists. Given a diagram of stable $\infty$-categories equipped with $t$-structures and $t$-exact transition maps, the inverse limit carries a unique $t$-structure compatible with those of the terms. Using the stackyness of $\mathcal{D}^{(b)}_{zc}(-)$, we thus conclude that it suffices to prove (1) when $X=\mathrm{Spa}(A)$ is affinoid. In this case, using Proposition~\ref{schemescomparison1} as well as the compatibility of the notion of dimension and duality with analytification, it is enough to prove the corresponding statements for $D^b_{cons}(\mathcal{X},\mathbf{Z}/n)$ where $\mathcal{X} = \mathrm{Spec}(A)$; we do this next via \cite{GabberNotest}.

Consider the strong perversity function $p:\mathcal{X} \to \mathbf{Z}$ given by $p(x) = -\dim(\overline{\{x\}})$. The results of \cite[\S 2 \& 6]{GabberNotest} show that there is a natural perverse $t$-structure on $D^b_c(\mathcal{X},\mathbf{Z}/n)$ attached to the function $p(-)$. It is clear from the definition in \cite[\S 2]{GabberNotest} as well as the compatibility of the notion of dimension with analytification that the connective part ${}^p D^{\leq 0}_c(\mathcal{X},\mathbf{Z}/n) \subset D^b_c(X,\mathbf{Z}/n)$ of this $t$-structures agrees with $\pervleqzero(X) \subset D_{zc}^b(X)$ under the equivalence $(-)^{an}:D^b_c(\mathcal{X},\mathbf{Z}/n) \simeq D^b_{zc}(X)$ from Proposition~\ref{schemescomparison1}. It remains to identify $\pervgeqzero(X) \subset D_{zc}^b(X)$ as defined above (via stalks of the dual) with ${}^p D^{\geq 0}_c(\mathcal{X},\mathbf{Z}/n) \subset D^{\geq 0}_c(\mathcal{X},\mathbf{Z}/n)$ as defined in \cite[\S 2]{GabberNotest} (via costalks). For this, it suffices to show the following pair of assertions: 
\begin{itemize}
\item[$(\ast)$] For any $\mathscr{G} \in D^b(\mathcal{X},\mathbf{Z}/n)$ and any geometric point $\overline{x} \to x \in \mathcal{X}$, the costalk $i_{\overline{x}}^! \mathscr{G}$ of $\mathscr{G}$ identifies with the $\mathbf{Z}/n$-linear dual of the stalk $i_{\overline{x}}^* \mathbf{D}_{\mathcal{X}}(\mathscr{G})$ of the Verdier dual of $\mathscr{G}$.
\end{itemize}
Indeed, assume $(\ast)$. Fix some $\mathscr{F} \in D^b_{zc}(X,\mathbf{Z}/n)$ arising as the analytification of $\mathscr{G} \in D^b_c(\mathcal{X},\mathbf{Z}/n)$. Assume first that $\mathscr{F} \in \pervgeqzero(X,\mathbf{Z}/n)$. Then for any irreducible Zariski-closed subset $Z \subset X$ of dimension $i$, we know by assumption $\mathcal{H}^{-j}(\mathbf{D}_X(F))$ vanishes after restriction to a Zariski open subset of $Z$ for all $j < i$. This implies a similar constraint on $\mathscr{G}$ by $t$-exactness of analytification and its compatibility with duality and the notion of dimension. Using $(\ast)$ and passing to the limit then shows that $\mathscr{G} \in {}^p D^{\geq 0}(\mathcal{X},\mathbf{Z}/n)$. Conversely, if $\mathscr{G} \in {}^p D^{\geq 0}(\mathcal{X},\mathbf{Z}/n)$, then $(\ast)$ and the compatibility of analyitification with duality and the notion of dimension shows that $\mathscr{F} \in \pervgeqzero(X,\mathbf{Z}/n)$. 

It remains to prove $(\ast)$. This follows by passage to the limit from (the algebraic version of) Theorem~\ref{dualizingcomplex} (4) applied to quasi-finite maps of the form $\mathcal{U} \hookrightarrow \mathcal{Z} \hookrightarrow \mathcal{X}$, with the first map being a dense open immersion, and the second map being the closed immersion of an irreducible closed subset.

{\em Direct proof.} We now explain a direct proof of the existence of the perverse $t$-structure on $D^b_{zc}(X)$ when $X$ is finite dimensional by induction on $\dim X$. The result is trivial when $\dim X=0$. For the moment, fix a smooth dense Zariski-open subset $j:U \to X$, with closed complement $i: Z \to X$. It is trivial from the definition that $i^{\ast} : D_{zc}^b(X) \to D_{zc}^b(Z)$ carries $\pervleqzero$ into $\pervleqzero$, and then (using biduality) that $Ri^!$ carries $\pervgeqzero$ into $\pervgeqzero$. By induction, we can assume that (1) is true for $Z$.

Write $D_{zc.U-lis}^b(X) \subset D_{zc}^b(X)$ for the full subcategory spanned by complexes whose cohomology sheaves are lisse after restriction to $U$. One trivially checks that $\pervleqzero(U) \cap D_{lis}^b(U)$ and $\pervgeqzero(U) \cap D_{lis}^b(U)$ define a $t$-structure on $D_{lis}^b(U)$, which locally on connected components is the obvious shift by $\dim U$ of the standard $t$-structure. Moreover, a complex $\mathscr{F} \in D^b_{zc.U-lis}(X)$ lies in $\pervleqzero(X)$ iff $j^{\ast}\mathscr{F} \in D^{ \leq -\dim X}_{lis}(U)$ and $i^{\ast}\mathscr{F} \in \pervleqzero(Z)$. By duality, this implies that $\mathscr{F} \in D^b_{zc.U-lis}(X)$ lies in $\pervgeqzero(X)$ iff $j^{\ast}\mathscr{F} \in D^{ \geq -\dim X}_{lis}(U)$ and $i^{!}\mathscr{F} \in \pervgeqzero(Z)$. 

On the other hand, by \cite[Theoreme 1.4.10]{BBDG} we can glue the perverse $t$-structure on $D^b_{lis}(U)$ and the perverse $t$-structure on $D^b_{zc}(Z)$ to get an actual $t$-structure on $D_{zc.U-lis}^{b}(X)$. The key technical ingredient here is Theorem \ref{openpushforward}, which guarantees that $Rj_{\ast}$ carries $D^b_{lis}(U)$ into $D^b_{zc.U-lis}(X)$. This together with the induction hypothesis implies that the truncation functors $\phantom{}^{\mathfrak{p}}\tau^{\leq i}$ preserve $D^b_{zc.U-lis}(X)$.

It is clear that this glued t-structure agrees with the restriction of the putative perverse $t$-structure from Definition \ref{perverse} to the full subcategory $D_{zc.U-lis}^{b}(X) \subset D_{zc}^{b}(X)$. Since $D_{zc}^b(X)$ is the filtered colimit of $D_{zc.U-lis}^{b}(X)$ over (the opposite category of) all $U \subset X$ as above, we deduce that $\pervleqzero$ and $\pervgeqzero$ define an honest $t$-structure on $D^{b}_{zc}(X)$.

\item The right $t$-exactness in part (a) is clear, while the left $t$-exactness follows as both functors commute with Verdier duality. 

Part (b) is clear.

The claim for $i^*$ in part (c) is clear and that for $Ri^!$ then follows by duality.

For part (d), it suffices to identify $\mathbf{D}_X Rj_* \mathscr{F}$ with $j_! \mathbf{D}_U(\mathscr{F})$ (whenever $\mathscr{F}$ satisfies the hypothesis in the proposition). These sheaves are isomorphic over $U$ as duality is local, so it is enough to show that $i^* \mathbf{D}_X Rj_* \mathscr{F} = 0$ for $i:Z \to X$ being the complementary closed immersion. But this follows as $i^* \mathbf{D}_X = \mathbf{D}_Z Ri^!$ on $D^b_{zc}(X)$ and $Ri^! Rj_* = 0$ on all of $D(U)$.

\item Clear from the definitions.

\item By Proposition \ref{schemescomparison1}, $(-)^{an}: D^b_c(\mathcal{X}) \to D^b_{zc}(X)$ is fully faithful, and is an equivalence in the proper case. It remains to show that $(-)^{an}$ is perverse $t$-exact.  Right $t$-exactness is clear, while left $t$-exactness follows as $(-)^{an}$ is compatible with duality (e.g., via  Lemma~\ref{dualizingcomplexweirdpullback}).

\item As usual, we define $j_{!*} \mathscr{L}$ to be the image of the map ${}^{\mathfrak{p}} \mathcal{H}^0(j_! \mathscr{L}) \to {}^{\mathfrak{p}} \mathcal{H}^0(Rj_* \mathscr{L})$ of perverse sheaves, noting that this makes sense by part (2) and Proposition~\ref{morepushforwards}. The remaining claims are immediate, using the formula $\mathbf{D}_X Rj_* \mathscr{L} = j_! \mathbf{D}_U(\mathscr{L})$ from (2) for the last part.

\item It suffices to prove every perverse sheaf has finite length.  We prove the claim by induction on dimension $d=\dim(X)$. Clearly we can assume $X$ is reduced.

When $d=0$, the space $X$ identifies with $\bigsqcup_{i=1}^n \mathrm{Spa}(K_i)$ with $K_i/K$ a finite extension. For such spaces, the claim is clear after translating from \'etale sheaves to Galois representations, ultimately because finite $\mathbf{Z}/n$-modules have finite length in the category of all $\mathbf{Z}/n$-modules.

Next, we show that for any Zariski locally closed immersion $j:U \to X$ and any lisse sheaf $\mathscr{L}$ on $U$, the intermediate extension $j_{!\ast}\mathscr{L}[\dim U]$ is a perverse sheaf of finite length. As pushforward along closed immersions is exact and fully faithful with essential image closed under passage to subquotients, we may assume $j$ is a dense Zariski-open immersion. Using induction on dimension as well as the fact that $j_{!\ast}$ is exact up to perverse sheaves supported on the Zariski-closed space $i:Z := X-U \hookrightarrow X$ which has dimension $<d$, it is enough to prove that $j_{!\ast} \mathscr{L}$ is simple if $\mathscr{L}$ is so. As $j^*$ is perverse $t$-exact, it suffices to show that $j_{!\ast} \mathscr{L}$ admits no non-trivial subobjects or quotients supported on $X-U$. The statement for quotients follows from the surjection ${}^{\mathfrak{p}} \mathcal{H}^0(j_! \mathscr{L}) \to j_{!*} \mathscr{L}$, the right perverse $t$-exactness of $j_!$, and the fact that $\inthom(j_!(-),i_*(-)) = 0$; the statement for subobjects follows by duality.

We now handle the general case. Given a perverse sheaf $\mathscr{F}$ on $X$, let $U \subset X$ be a dense Zariski-open subset such that $\mathscr{F}|_U[-\dim(U)]$ is lisse. Then we have a correspondence 
\[ j_{!*} (\mathscr{F}|_U) \gets {}^{\mathfrak{p}} \mathcal{H}^0(j_! (\mathscr{F}|_U)) \to \mathscr{F}\]
of perverse sheaves with both maps having cones have perverse cohomology sheaves supported on $Z=X-U$. As $\dim(Z) < \dim(X)$, induction on dimension and the previous paragraph show that $\mathscr{F}$ has finite length. 

The claimed description of simple objects also follows from the proof above (and is similar to the algebraic case). Indeed, say $\mathscr{F}$ is simple and supported on some Zariski-closed subset $Z \subset X$. Replacing $X$ with $Z$, we can assume $\mathscr{F}$ is supported everywhere. Let $j:U \subset X$ be a Zariski-dense Zariski-open subset of (locally constant) dimension $d$ such that $\mathscr{F}|_U = \mathscr{L}[d]$ for a lisse sheaf $\mathscr{L}$ on $U$. As $j_{!*}$ preserves injections and has a left-inverse, the simplicity of $\mathscr{F}$ implies that $\mathscr{L}$ must be simple.  Both maps in the correspondence  $j_{!*} (\mathscr{F}|_U) \gets {}^{\mathfrak{p}} \mathcal{H}^0(j_! (\mathscr{F}|_U)) \to \mathscr{F}$ used in the previous paragraph must then be surjective by simplicity of the targets. The kernels of both maps are supported on $X-U$ while the simple targets are supported on all of $X$. It follows that kernels of both maps identify with the maximal perverse subsheaf of $ {}^{\mathfrak{p}} \mathcal{H}^0(j_! (\mathscr{F}|_U)) $ supported on $X-U$. In particular, both maps are isomorphic, so $\mathscr{F} = j_{!*}(\mathscr{F}|_U)$, as wanted.

\item Right t-exactness is clear, and commutation of finite pushforward with Verdier duality (Theorem~\ref{dualizingcomplex} (4)) gives left t-exactness.

\item  By the commutation of nearby cycles with Verdier duality \cite{H2}, it's enough to show that $R\lambda_{\mathfrak{X}\ast}$ is right $t$-exact for the perverse t-structures.  Fix some $\mathscr{F} \in \pervleqzero(X)$. By \cite[R\'eciproque 4.1.6]{BBDG}, to check that $R\lambda_{\mathfrak{X}\ast} \mathscr{F} \in \phantom{}^{\mathfrak{p}}D^{\leq 0}(\mathfrak{X}_s)$ it suffices to show that for any \'etale map $\mathfrak{j}:\mathfrak{Y}_s \to \mathfrak{X}_s$ from an affine scheme $\mathfrak{Y}_s$, the complex $R\Gamma(\mathfrak{Y}_s, \mathfrak{j}^{\ast}R\lambda_{\mathfrak{X}\ast}\mathscr{F})$ is concentrated in non-positive degrees. 

Let $j:Y \to X$ be the \'etale map obtained by deforming $\mathfrak{j}$ to a map of formal schemes and then passing to the rigid generic fiber. Note that $Y$ is affinoid. Then $R\Gamma(\mathfrak{Y}_s, \mathfrak{j}^{\ast}R\lambda_{\mathfrak{X}\ast}\mathscr{F}) \cong R\Gamma(Y,j^{\ast}\mathscr{F})$ by basic properties of the nearby cycles functor in this setting, and $j^\ast \mathscr{F} \in \pervleqzero(Y)$. But $R\Gamma(Y,\mathscr{G})$ is concentrated in degrees $\leq 0$ for any $\mathscr{G} \in \pervleqzero(Y)$ by rigid analytic Artin-Grothendieck vanishing \cite{BM, H}.
\end{enumerate}
\end{proof}

\subsection{$\mathbf{Z}_\ell$-coefficients}
\label{ss:ZellCoeff}

Let $K$ be a nonarchimedean field of characteristic $0$ and residue characteristic $p > 0$, let $\ell$ be a prime number (including possibly $\ell=p$), and let $X/K$ be a rigid space. Our goal is to define a perverse $t$-structure on the category $D^{(b)}_{zc}(X,\mathbf{Z}_\ell)$ (introduced in \S \ref{ss:adic}) that agrees on $\ell$-torsion objects with our previous construction. The definition of the connective part is the same, but that of the coconnective part needs to modified to account for the fact that the standard $t$-structure on $D_{perf}(\mathbf{Z}_\ell)$ is not quite self-dual. A similar issue occurs in algebraic geometry (see \cite[\S 3.3]{BBDG}), and our fix is also similar: there are {\em two} perverse $t$-structures with $\mathbf{Z}_\ell$-coefficients that are exchanged by Verdier duality and which differ from each other by torsion (Proposition~\ref{IntPerverseProp}).

\begin{construction}[The $\mathfrak{p}$- and $\mathfrak{p}^+$-perverse $t$-structures]
\label{cons:IntPerv}
Consider the following full subcategories of $D^{(b)}_{zc}(X,\mathbf{Z}_\ell)$:
\begin{itemize}
\item $\pervleqzero(X,\mathbf{Z}_\ell)$ is the collection of all $K$'s with $K/\ell \in \pervleqzero(X,\mathbf{F}_\ell)$.
\item $\pervgeqzero(X,\mathbf{Z}_\ell)$ is the collection of all those $K$'s with $\mathbf{D}_X(K) \in {}^{\mathfrak{p}} D^{\leq 1}_{zc}(X,\mathbf{Z}_\ell)$ and such that, locally on $X$, there exists some $c$ with $\ell^c \cdot {}^{\mathfrak{p}} \mathcal{H}^1(\mathbf{D}_X(K)/\ell^n) = 0$ for all $n$.
\end{itemize}
We refer to the pair $(\pervleqzero(X,\mathbf{Z}_\ell), \pervgeqzero(X,\mathbf{Z}_\ell))$ as the {\em $\mathfrak{p}$-perverse $t$-structure} on $D^{(b)}_{zc}(X,\mathbf{Z}_\ell)$; it will be shown to be a $t$-structure later (Proposition~\ref{IntPerverseProp}).

Write $({}^{\mathfrak{p}^+} D^{\leq 0}(X,\mathbf{Z}_\ell), {}^{\mathfrak{p}^+} D^{\geq 0}(X,\mathbf{Z}_\ell))$ for the dual of the pair $(\pervleqzero(X,\mathbf{Z}_\ell), \pervgeqzero(X,\mathbf{Z}_\ell))$, i.e., 
\[ {}^{\mathfrak{p}^+} D^{\leq 0}_{zc}(X,\mathbf{Z}_\ell) = \mathbf{D}_X \pervgeqzero(X,\mathbf{Z}_\ell) \quad \text{and} \quad  {}^{\mathfrak{p}^+} D^{\geq 0}_{zc}(X,\mathbf{Z}_\ell) =  \mathbf{D}_X \pervleqzero(X,\mathbf{Z}_\ell).\]
We refer to the pair $({}^{\mathfrak{p}^+} D^{\leq 0}_{zc}(X,\mathbf{Z}_\ell), {}^{\mathfrak{p}^+} D^{\geq 0}_{zc}(X,\mathbf{Z}_\ell))$ as the {\em $\mathfrak{p}^+$-perverse $t$-structure} on  $D^{(b)}_{zc}(X,\mathbf{Z}_\ell)$.
\end{construction}

\begin{example}[The case of a point]
\label{PerverseLisseIntegral}
Assume $X=\mathrm{Spa}(K)$ is a geometric point, so $K$ is algebraically closed. In this case, we may identify $D^{(b)}_{zc}(X,\mathbf{Z}_\ell) = D_{perf}(\mathbf{Z}_\ell)$. Under this equivalence, the $\mathfrak{p}$-perverse $t$-structure on $D_{perf}(\mathbf{Z}_\ell)$ identifies with the standard $t$-structure (and is thus a $t$-structure). Indeed, the identification of the connective part is clear. For the coconnective part,  we must show that $M \in D_{perf}(\mathbf{Z}_\ell)$ lies in $D^{\geq 0}$ exactly when $M^\vee := \mathrm{RHom}(M,\mathbf{Z}_\ell) \in D^{\leq 1}$ with $\mathrm{Ext}^1(M,\mathbf{Z}_\ell)$ being torsion. This follows easily by using biduality $M = \mathrm{RHom}(M^\vee,\mathbf{Z}_\ell)$ as well as the fact that $\mathrm{Hom}(N,\mathbf{Z}_\ell) = 0$ if $N$ is torsion.

More generally, a similar argument shows the following: for a smooth rigid space $X/K$ of dimension $d$, intersecting the $\mathfrak{p}$-perverse $t$-structure with $D^b_{lis}(X,\mathbf{Z}_\ell)$ gives (homological) $d$-fold shift of the standard $t$-structure on $D^b_{lis}(X,\mathbf{Z}_\ell)$.
\end{example}

To compare the two $t$-structures in Construction~\ref{cons:IntPerv}, we shall need the following notion.

\begin{definition}
We say that an object $K \in D^{(b)}_{zc}(X,\mathbf{Z}_\ell)$ is {\em locally bounded torsion} if, locally on $X$, there exists some $c$ with $\ell^c \cdot \mathcal{H}^*(K)= 0$. 
\end{definition}

The main result of this subsection is the following analog of some remarks in \cite[\S 3.3]{BBDG}:

\begin{proposition}[Properties of $\mathbf{Z}_\ell$-perverse sheaves]
\label{IntPerverseProp}
\begin{enumerate}
\item The $\mathfrak{p}$-perverse $t$-structure is indeed a $t$-structure on $D^{(b)}_{zc}(X,\mathbf{Z}_\ell)$. Consequently, the same holds for $\mathfrak{p}^+$-perverse $t$-structure.

\item For any $n \geq 1$, the reduction modulo $\ell^n$-functor $D^{(b)}_{zc}(X,\mathbf{Z}_\ell) \to D^{(b)}_{zc}(X,\mathbf{Z}_\ell/\ell^n)$ is right $t$-exact with respect to $\mathfrak{p}$-perverse $t$-structure on the source and the perverse $t$-structure on the target. 

\item For any $n \geq 1$, the restriction of scalars functor $\mathrm{Res}:D^{(b)}_{zc}(X,\mathbf{Z}_\ell/\ell^n) \to  D^{(b)}_{zc}(X,\mathbf{Z}_\ell)$ is $t$-exact with respect to the same pair of $t$-structures as in (2).

\item We have $\pervleqzero(X,\mathbf{Z}_\ell) \subset  {}^{\mathfrak{p}^+} D^{\leq 0}_{zc}(X,\mathbf{Z}_\ell) \subset {}^{\mathfrak{p}} D^{\leq 1}_{zc}(X,\mathbf{Z}_\ell)$.

\item Given $K \in D^{(b)}_{zc}(X,\mathbf{Z}_\ell)$, we have $K \in {}^{\mathfrak{p}^+} D^{\leq 0}_{zc}(X,\mathbf{Z}_\ell)$ if and only if $K \in {}^{\mathfrak{p}} D^{\leq 1}_{zc}(X,\mathbf{Z}_\ell)$ with ${}^{\mathfrak{p}} \mathcal{H}^1(K)$ being locally bounded torsion. (Note that ${}^{\mathfrak{p}} \mathcal{H}^1(-)$ makes sense by part (1).)
\end{enumerate}
\end{proposition}
\begin{proof}
\begin{enumerate}
\item  All assertions are local, so we may assume $X=\mathrm{Spa}(A)$ is affinoid. We proceed by induction on $\dim(X)$. If $\dim(X) = 0$, then we can reduce to the case where $X$ is a point. In this case, the claim follows by Example~\ref{PerverseLisseIntegral}. In general, we translate the theorem to a similar question about $D^b_{cons}(\mathrm{Spec}(A),\mathbf{Z}_\ell)$ with evident definitions, and proceed by imitating the glueing method of \cite[\S 1.4]{BBDG}. Fix a smooth dense Zariski-open $j:U \subset \mathcal{X}$ of dimension $d$ with complementary closed $i:Z \subset \mathcal{X}$. Consider the full subcategory $D_{U-lis} \subset D^b_{cons}(\mathrm{Spec}(A),\mathbf{Z}_\ell)$ spanned by complexes $K$ which are lisse over $U$. Then $D_{U-lis}$ admits a semi-orthogonal decomposition into $D_{lis}^b(U,\mathbf{Z}_\ell)$ as well as $D^b_{cons}(Z,\mathbf{Z}_\ell)$ as in \cite[\S 1.4.3]{BBDG}. Moreover, for $K \in D_{U-lis}$, one checks that $K \in {}^p D^{\leq 0}_{cons}(X,\mathbf{Z}_\ell)$ (resp. $K \in {}^p D^{\geq 0}_{cons}(X,\mathbf{Z}_\ell)$) exactly when its $*$-pullbacks (resp. $!$-pullbacks) to $U$ and $Z$ lie in ${}^p D^{\leq 0}_{lis}(U,\mathbf{Z}_\ell)$ and ${}^p D^{\leq 0}_{cons}(Z,\mathbf{Z}_\ell)$ (resp. ${}^p D^{\geq 0}_{lis}(U,\mathbf{Z}_\ell)$ and ${}^p D^{\geq 0}_{cons}(Z,\mathbf{Z}_\ell)$): this is clear for ${}^p D^{\leq 0}$ over both $U$ and $Z$ as well as for ${}^p D^{\geq 0}$ over $U$, and follows for ${}^p D^{\leq 0}$ over $Z$ by the formula $i^* \mathbf{D}_{\mathcal{X}} = \mathbf{D}_Z Ri^!$. One can then use \cite[Theorem 1.4.10]{BBDG} to glue the $\mathfrak{p}$-perverse $t$-structures on $D_{lis}^b(U,\mathbf{Z}_\ell)$ as well as $D^b_{cons}(Z,\mathbf{Z}_\ell)$ (which are $t$-structures by Example~\ref{PerverseLisseIntegral} and induction respectively) to conclude that intersecting the  $\mathfrak{p}$-perverse $t$-structure with $D_{U-lis}$ gives a $t$-structure on $D_{U-lis}$. Taking the colimit over all such $U$'s then proves (1).

\item Clear from the definition.

\item The right $t$-exactness is again clear from the definition. The left $t$-exactness follows by unwinding definitions from Remark~\ref{VDZellZmodell}.

\item Both containments are immediate from biduality.

\item Fix some $K \in D^{(b)}_{zc}(X,\mathbf{Z}_\ell)$. Both directions can be checked locally on $X$, so we may assume $X$ is qcqs and thus $K$ is bounded. 

We first prove the ``only if'' direction, so assume that $K \in  {}^{\mathfrak{p}^+} D^{\leq 0}(X,\mathbf{Z}_\ell)$.  Unwinding definitions and using biduality, this means that $K \in {}^{\mathfrak{p}} D^{\leq 1}_{zc}(X,\mathbf{Z}_\ell)$ and that there exists some $c \geq 1$ such that $\ell^c \cdot {}^{\mathfrak{p}} \mathcal{H}^1(K/\ell^n) = 0$ for all $n$. We shall prove that $\ell^c \cdot  {}^{\mathfrak{p}} \mathcal{H}^1(K) = 0$. Since $K \in {}^{\mathfrak{p}} D^{\leq 1}_{zc}(X,\mathbf{Z}_\ell)$ and  $\pervleqzero(X,\mathbf{Z}_\ell) \subset  {}^{\mathfrak{p}^+} D^{\leq 0}_{zc}(X,\mathbf{Z}_\ell)$, we are allowed to replace $K$ with ${}^{\mathfrak{p}} \mathcal{H}^1(K)[-1]$, so we may assume that $K$ is concentrated in cohomological degree $1$ with respect to the $\mathfrak{p}$-perverse $t$-structure. Moreover, by a variant of the argument used to prove (1), one checks that there exists a constant $c'$ such that $K$ (or any perverse $\mathbf{Z}_\ell$-sheaf) has $\ell^\infty$-torsion bounded by $\ell^{c'}$, i.e., that the perverse $\mathbf{Z}_\ell$-sheaves $\ker(\ell^n:K \to K)$ are killed by $\ell^{c'}$ for all $n$. Our hypothesis on $K$ then shows that the complex $K/\ell^n$ is killed by $\ell^{2 \max(c,c')}$ for all $n$. But this implies $K$ must be killed by $\ell^{2\max(c,c')}$ by generalities on derived $\ell$-complete sheaves in the replete topos\footnote{This follows from (the replete topos variant of) the following statement (which appears in \cite{BhattLuriepadicRH} and whose proof we leave as an exercise here): If $M$ is any derived $\ell$-complete abelian group such that there exists some $c \geq 0$ with $\ell^c \cdot (M/\ell^nM) = 0$ for all $n \geq c+1$, then $\ell^c M = 0$.} of all $v$-sheaves on $X$, so we are done.

For the ``if'' direction, assume that $K \in {}^{\mathfrak{p}} D^{\leq 1}_{zc}(X,\mathbf{Z}_\ell)$ and the object ${}^{\mathfrak{p}} \mathcal{H}^1(K)$ is killed by $\ell^c$ for some $c$. As reduction modulo powers of $\ell$ is right $t$-exact for the perverse $t$-structure, it is then trivially true that $\ell^c \cdot {}^{\mathfrak{p}} \mathcal{H}^1(K/\ell^n) = 0$ for all $n \geq 1$. It is then immediate from the definitions that $K \in {}^{\mathfrak{p}^+} D^{\leq 0}(X,\mathbf{Z}_\ell)$.

\end{enumerate}
\end{proof}

\begin{remark}
\label{IntPervLocalTors}
Proposition~\ref{IntPerverseProp} (5) cannot be strengthened to the assertion that ${}^{\mathfrak{p}} \mathcal{H}^1(K)$ is bounded torsion globally on $X$ for $K \in  {}^{\mathfrak{p}^+} D^{\leq 0}(X,\mathbf{Z}_\ell)$. Indeed, given a countable discrete subset $S := \{x_1,x_2,x_3,...\} \subset X := (\mathbf{A}^1)^{an}$ of classical points, one may take $K = \bigoplus i_{x_n,*} \mathbf{Z}/\ell^n[-1]$ to obtain a counterexample (where $i_{x_n}:\mathrm{Spa}(k(x_n)) \to X$ is the inclusion of the point at $x_n$). 
\end{remark}

\subsection{$\mathbf{Q}_\ell$-coefficients}
\label{ss:Qell}
We continue with notation from \S \ref{ss:ZellCoeff}. There are some subtleties with passing from $\mathbf{Z}_\ell$ to $\mathbf{Q}_\ell$-coefficents for rigid spaces that are not qcqs.\footnote{In fact, similar issues arise in algebraic geometry but are typically not as consequential as non-qcqs schemes are much rarer than non-compact rigid spaces. For instance, the affine line over $K$ is qcqs when regarded as a scheme simply because it is a noetherian scheme, but its analytification is not a qcqs rigid space.} Thus, in this subsection, we assume $X$ is qcqs (e.g., $X$ could be affinoid or proper over $K$), so $D^{(b)}_{zc}(X,\mathbf{Z}_\ell) = D^b_{zc}(X,\mathbf{Z}_\ell)$. In this setting, we shall prove in Theorem~\ref{PervQell} that the basic theory of perverse sheaves with $\mathbf{Q}_\ell$-coefficients behaves as well as can be expected.

Our constructions will take place in the following category:

\begin{definition}[$\mathbf{Q}_\ell$-constructible sheaves]
\label{def:QellZC}
Set $D^b_{zc}(X,\mathbf{Q}_\ell) := D^b_{zc}(X,\mathbf{Z}_\ell) \otimes_{\mathbf{Z}_\ell} \mathbf{Q}_\ell$ (i.e., objects remain the same and endomorphisms are tensored with $\mathbf{Q}_\ell$).  
\end{definition}

\begin{remark}[$\mathbf{Q}_\ell$-sheaves as Verdier quotient]
\label{QellVerdQuot}
The category $D^b_{zc}(X,\mathbf{Q}_\ell)$ can also be described as the Verdier quotient of $D^b_{zc}(X,\mathbf{Z}_\ell)$ by its full subcategory of objects annihilated by a power of $\ell$. In fact, the analogous statement holds true with $D^b_{zc}(X,\mathbf{Z}_\ell)$ replaced by any $\mathbf{Z}_\ell$-linear triangulated category $\mathcal{C}$. To see this, let $\mathcal{C}_{tors} \subset \mathcal{C}$ be the full subcategory of objects annihilated by a power of $\ell$. As the multiplication by $\ell$ map on any object of $\mathcal{C}$ has cone in $\mathcal{C}_{tors}$, it follows that $\mathcal{C}/\mathcal{C}_{tors}$ is naturally $\mathbf{Q}_\ell$-linear, so there is a natural map $\mathcal{C} \otimes_{\mathbf{Z}_\ell} \mathbf{Q}_\ell \to \mathcal{C}/\mathcal{C}_{tors}$. Conversely, as the Verdier quotient $\mathcal{C}/\mathcal{C}_{tors}$ can be regarded as the localization $S^{-1} \mathcal{C}$, where $S$ is the collection of maps in $\mathcal{C}$ whose cone lies in $\mathcal{C}_{tors}$, one also immediately constructs a natural map $\mathcal{C}/\mathcal{C}_{tors} \to \mathcal{C} \otimes_{\mathbf{Z}_\ell} \mathbf{Q}_\ell$. We leave it to the reader to check that these constructions give mutually inverse equivalences of categories.
\end{remark}

\begin{remark}[Problems in the non-qcqs case]
\label{rmk:Qellnonqcqs}
While Definition~\ref{def:QellZC} makes sense for any rigid space $X$, it is the ``wrong'' definition to use when $X$ is not qcqs. For example, the object $K$ described in Remark~\ref{IntPervLocalTors} is nonzero in $D^b_{zc}(X,\mathbf{Q}_\ell)$ yet vanishes after restriction to any quasi-compact open in $X$. While there are several candidate replacements (e.g., based on Remark~\ref{QellVerdQuot}, one might work with the quotient of $D^{(b)}_{zc}(X,\mathbf{Z}_\ell)$ by the full subcategory of locally bounded torsion objects; alternately, one might attempt to work with Zariski constructible $\mathbf{Q}_\ell$-complexes defined using the pro\'etale site), we were unable to develop enough machinery to construct a reasonable intersection cohomology theory (e.g., a self-dual theory with a GAGA theorem) using any of these approaches, so we restrict to the qcqs case in our discussion.
\end{remark}

Using our results on integral coefficients, we obtain a well-behaved perverse $t$-structure with $\mathbf{Q}_\ell$-coefficients:

\begin{theorem}[Properties of $\mathbf{Q}_\ell$-perverse sheaves]
\label{PervQell}
The $\mathfrak{p}$- and $\mathfrak{p}^+$- perverse $t$-structures from \S \ref{ss:ZellCoeff} induce a $t$-structure on $D^b_{zc}(X,\mathbf{Q}_\ell)$, and they are the same $t$-structure; we call this the perverse $t$-structure on $D^b_{zc}(X,\mathbf{Q}_\ell)$. This $t$-structure satisfies all the properties in Theorem~\ref{perverseproperties} with the following changes: one works with proper $\mathcal{X}$ in (4),  only finite maps of qcqs spaces in (7), and replaces the assumption $p \nmid \#\Lambda$ with $p \neq \ell$ in (8).
\end{theorem}

\begin{proof}
The first part is immediate from Proposition~\ref{IntPerverseProp} (using part (5) there and the fact $X$ is qcqs to get the equality of the two $t$-structures).  It remains to verify the properties in Theorem~\ref{perverseproperties} (2) - (7).

Property (3): this is immediate from the fact that the $\mathfrak{p}$- and $\mathfrak{p}^+$-perverse $t$-structures on $D^{(b)}_{zc}(X,\mathbf{Z}_\ell)$ are exchanged by Verdier duality.

Property (2): parts (a), (b) and the right $t$-exactness in (c) is clear from the definition, while the left $t$-exactness in (c) was implicitly asserted in the proof of Proposition~\ref{IntPerverseProp} (1). For part (d), we use the stronger property $\mathbf{D}_X Rj_* \mathscr{F} =j_! \mathbf{D}_U(\mathscr{F})$ proven in Theorem~\ref{perverseproperties} (2) (d) and invert $\ell$. 

Property (4): the equivalence is clear. For perverse $t$-exactness, one simply notes that entire discussion in this section also holds true in the algebro-geometric context (and in fact was borrowed from there, see \cite[\S 3.3]{BBDG}), and that analytification is compatible with duality and passing to perverse cohomology sheaves with finite coefficients.

Properties (5)-(8): these follow by the same proof as in Theorem~\ref{perverseproperties}.
\end{proof}

\subsection{Intersection cohomology}

In this section, fix a rigid space $X/K$, a prime $\ell$, and  a coefficient ring $\Lambda \in \{\mathbf{Z}/\ell^n,\mathbf{Q}_\ell\}$. If $\Lambda \in \{\mathbf{Z}_\ell, \mathbf{Q}_\ell\}$, we assume $K$ has positive residue characteristic $p > 0$. If $\Lambda=\mathbf{Q}_\ell$, then we also assume that $X$ is qcqs.  Note that we have a reasonable (e.g., self-dual) theory of perverse $\Lambda$-sheaves in this context by \S \ref{ss:PervFin} and \S \ref{ss:Qell} respectively.

\begin{construction}[Intersection cohomology of rigid spaces]
  Let $j:U \subset X$ a Zariski-dense Zariski-open subset such that $U_{red}$ is smooth. Write $\mathrm{IC}_{X,\Lambda} := j_{!*} \Lambda[\dim(X)] \in \Perv(X,\Lambda)$; one can show that this is independent of the choice of $U$. We call $\mathrm{IC}_{X,\Lambda}$ the {\em intersection cohomology complex on $X$} and write $IH^*(X,\Lambda) := H^*(X,\mathrm{IC}_{X,\Lambda})$ for its cohomology, called the {\em intersection homology of $X$}.
\end{construction}

We then have the following result on these objects:

\begin{theorem}[Basic properties of intersection cohomology]
\label{ICProperties}
Write $C/K$ for a completed algebraic closure.
\begin{enumerate}
\item $IH^*(X_C,\Lambda)$ are finitely generated $\Lambda$-modules if either $X$ is proper or if $X$ is qcqs and $p \neq \ell$. 
\item If $X = \mathcal{X}^{an}$ for a proper $K$-scheme $X$, then $IH^*(X_C,\mathbf{Q}_\ell) \simeq IH^*(\mathcal{X}_C,\mathbf{Q}_\ell)$.
\item If $X$ is proper and equidimensional of dimension $d$ and $p \neq \ell$, there is a natural Poincar\'e duality isomorphism 
\[ IH^i(X_C,\Lambda)^\vee \simeq IH^{-i}(X_C,\Lambda)(d)\]
for all $i$. 
\end{enumerate}
\end{theorem}
Ongoing work by Zavyalov suggests that the third part should hold true without the assumption on $\ell$.

\begin{proof}
\begin{enumerate}
\item For $X$ proper, we obtain the result from Theorem~\ref{properdirectimage}.  For $X$ qcqs and $p \neq \ell$, we can then choose a formal model of $X$ and deduce the claim from the constructibility of nearby cycles \cite{Hub98}.

\item  It is enough to prove that $\mathrm{IC}_{X,\Lambda} = (\mathrm{IC}_{\mathcal{X},\Lambda})^{an}$. This follows from the definition of either side as an appropriate image, and the compatibility of $(-)^{an}$ with all the constituent operations (namely, $j_!$, $j_*$, perverse truncations, and images).

\item As $X$ is proper, Huber's results show that $R\Gamma(X_C,-)$ is compatible with duality \cite[Ch. 7]{Hub96}, so it is enough to show that $\mathbf{D}_X(IC_{X,\Lambda}) \simeq IC_{X,\Lambda}(d)$. Using Theorem~\ref{perverseproperties} (4), this amounts to checking that $\mathbf{D}_U(\Lambda[d]) \simeq \Lambda[d](d) $ for a smooth rigid space $U$ of dimension $d$, which follows immediately as $\omega_U = \Lambda[2d](d)$ (Remark~\ref{EtaleDualizingRegular}).

\end{enumerate}
\end{proof}

\begin{remark}[Intersection cohomology for Zariski-compactifiable spaces]
\label{conj:ICQell}
We expect that there is a well-behaved notion of intersection cohomology with $\mathbf{Q}_\ell$-coefficients on any rigid space, not merely the qcqs ones. Since we did not construct a good category of perverse $\mathbf{Q}_\ell$-sheaves (see Remark~\ref{rmk:Qellnonqcqs}), let us formulate a precise conjecture. Say $X$ is a rigid space equipped with a Zariski open immersion $j:X \hookrightarrow \overline{X}$ with $\overline{X}$ proper over $K$. One can then define a candidate intersection cohomology complex $\mathrm{IC}_{X,\mathbf{Q}_\ell} := j^* \mathrm{IC}_{\overline{X},\mathbf{Q}_\ell} \in D^b_{zc}(X,\mathbf{Z}_\ell) \otimes_{\mathbf{Z}_\ell} \mathbf{Q}_\ell$ as well as the resulting intersection homology groups $IH^*(X,\mathbf{Q}_\ell) := \mathrm{Ext}^*(\mathbf{Q}_\ell, \mathrm{IC}_{X,\mathbf{Q}_\ell})$ (where the Exts are computed in  $D^b_{zc}(X,\mathbf{Z}_\ell) \otimes_{\mathbf{Z}_\ell} \mathbf{Q}_\ell$). We conjecture that these objects are independent of the compactification. Note that $IH^*(X,\mathbf{Q}_\ell)$ will be finite dimensional $\mathbf{Q}_\ell$-vector space using Theorem~\ref{morepushforwards}.
\end{remark}

\subsection{Some conjectures}
\label{ss:conj}

Given the results in this paper, it is natural to expect that most of the important foundational theorems on perverse sheaves in complex geometry or arithmetic algebraic geometry admit analogs for Zariski-constructible sheaves in $p$-adic analytic geometry. In this section, we formulate some conjectures along these lines.

Let $K/\mathbf{Q}_p$ be a finite extension, with residue field $k$ of cardinality $q$; let $C/K$ be a completed algebraic closure. Let $X$ be a rigid space over $K$. Let us begin by describing a conjecture on $\ell$-adic intersection cohomology; we believe this conjecture is accessible in the algebraic case thanks to de Jong's alterations theorem \cite{dJAlt}.

\begin{conjecture}[$\ell$-adic intersection cohomology]
\label{conj:ICell}
Assume $X$ is qcqs. Fix a prime $\ell \neq p$. 
\begin{enumerate}

\item {\em Nearby cycles:} For any formal model $\mathfrak{X}/\mathcal{O}_K$ and any $\ell \neq p$, the nearby cycle sheaf $\mathscr{F}=R\lambda_{\mathfrak{X}\ast}(\mathrm{IC}_{X,\mathbf{Q}_{\ell}})$ is a mixed $\ell$-adic perverse sheaf on the geometric special fiber $\mathfrak{X}_{\overline{s}}$. Moreover, if $X$ is equidimensional of dimension $d$, then $\mathrm{IC}_{\mathfrak{X}_s,\mathbf{Q}_{\ell}}$ occurs as a summand of the $d$th graded piece of the weight filtration of $\mathscr{F}$. 

\item {\em Weights:} For any prime $\ell \neq p$ and any $g \in W_K$ projecting to a nonnegative power of geometric Frobenius, the eigenvalues of $g$ acting on $IH^\ast(X_C,\mathbf{Q}_{\ell})$ are $q$-Weil numbers of weight $\geq 0$ .
\end{enumerate}
\end{conjecture}

Next, we formulate a conjecture on the $p$-adic Hodge theoretic properties of $p$-adic intersection cohomology. The first part of this conjecture can be proven in the algebraic case using the decomposition theorem.

\begin{conjecture}[$p$-adic intersection cohomology]
Say $X$ is proper over $K$. 
\begin{enumerate}
\item Each $IH^i(X_C,\mathbf{Q}_p)$ is a de Rham $G_K$-representation. (Moreover, assuming the conjecture in Remark~\ref{conj:ICQell}, this should be  true for any Zariski compactifiable rigid space.)

\item If $\mathbf{L}$ is a de Rham $\mathbf{Z}_p$-local system on a smooth Zariski-open subset $j:U \to X$, then $H^\ast(X_{\overline{K}},\mathrm{IC}(\mathbf{L}[\dim(X)]))$ is de Rham.
\end{enumerate}
\end{conjecture}

Finally, we discuss the rigid analog of the BBDG decomposition theorem. As in complex geometry, the Hopf surface construction $X = (\mathbf{A}^{2,\mathrm{an}}-\{0\})/q^{\mathbf{Z}}$ (with $q \in K$ with $0 < |q| < |1|$) gives a proper smooth genus $1$ fibration $f:X \to (\mathbf{A}^{2,\text{an}}-\{0\})/\mathbf{G}_m^{\text{an}} \simeq \mathbf{P}^{1,\text{an}}$ over $K$ such that $Rf_* \mathbf{Q}_\ell$ is not formal (i.e., is not isomorphic to a direct sum of its shifted cohomology sheaves). It is thus unreasonable to expect the decomposition theorem to hold true for arbitrary proper maps between rigid spaces. Nevertheless, by analogy with the complex geometric story in \cite{SaitoDecompKahler}, the following  appears plausible:

\begin{conjecture}[Decomposition theorem]
Let $f:X \to Y$ be a projective map of rigid spaces over $K$ with $Y$ qcqs. Then $Rf_* \mathrm{IC}_{X,\mathbf{Q}_\ell}$ is a direct sum of shifts of perverse sheaves of the form $j_{!*} \mathscr{L}$, where $j:U \to Y$ is a Zariski-locally closed immersion and $\mathscr{L}$ is a $\mathbf{Q}_\ell$-local system on $U$. 
\end{conjecture}

Finally, we also expect that Zariski-constructible sheaves on smooth rigid spaces are holonomic, in analogy with \cite{KS,BeilinsonSS}, but we do not formulate a precise statement here.

\bibliographystyle{alpha}
\bibliography{constructibility-rigid-bib}

\begin{thebibliography}{dJvdP96}

\bibitem[Ach17]{AchingerKpi1}
Piotr Achinger.
\newblock Wild ramification and {$K(\pi, 1)$} spaces.
\newblock {\em Invent. Math.}, 210(2):453--499, 2017.

\bibitem[And74]{A}
Michel Andr\'{e}.
\newblock Localisation de la lissit\'{e} formelle.
\newblock {\em Manuscripta Math.}, 13:297--307, 1974.

\bibitem[Bar76]{Bar76}
Wolfgang Bartenwerfer.
\newblock Der erste {R}iemannsche {H}ebbarkeitssatz im nichtarchimedischen
  {F}all.
\newblock {\em J. Reine Angew. Math.}, 286(287):144--163, 1976.

\bibitem[BBD82]{BBDG}
A.~A. Be\u{\i}linson, J.~Bernstein, and P.~Deligne.
\newblock Faisceaux pervers.
\newblock In {\em Analysis and topology on singular spaces, {I} ({L}uminy,
  1981)}, volume 100 of {\em Ast\'{e}risque}, pages 5--171. Soc. Math. France,
  Paris, 1982.

\bibitem[Bei16]{BeilinsonSS}
A.~Beilinson.
\newblock Constructible sheaves are holonomic.
\newblock {\em Selecta Math. (N.S.)}, 22(4):1797--1819, 2016.

\bibitem[Ber90]{BerSpectral}
Vladimir~G. Berkovich.
\newblock {\em Spectral theory and analytic geometry over non-{A}rchimedean
  fields}, volume~33 of {\em Mathematical Surveys and Monographs}.
\newblock American Mathematical Society, Providence, RI, 1990.

\bibitem[Ber93]{BerEt}
Vladimir~G. Berkovich.
\newblock \'{E}tale cohomology for non-{A}rchimedean analytic spaces.
\newblock {\em Inst. Hautes \'{E}tudes Sci. Publ. Math.}, (78):5--161 (1994),
  1993.

\bibitem[BGR84]{BGR}
S.~Bosch, U.~G\"{u}ntzer, and R.~Remmert.
\newblock {\em Non-{A}rchimedean analysis}, volume 261 of {\em Grundlehren der
  Mathematischen Wissenschaften [Fundamental Principles of Mathematical
  Sciences]}.
\newblock Springer-Verlag, Berlin, 1984.
\newblock A systematic approach to rigid analytic geometry.

\bibitem[BL20]{BhattLuriepadicRH}
Bhargav Bhatt and Jacob Lurie.
\newblock A $p$-adic {R}iemann-{H}ilbert functor {I}.
\newblock 2020.
\newblock preprint.

\bibitem[BLR95]{BLR}
Siegfried Bosch, Werner L\"{u}tkebohmert, and Michel Raynaud.
\newblock Formal and rigid geometry. {IV}. {T}he reduced fibre theorem.
\newblock {\em Invent. Math.}, 119(2):361--398, 1995.

\bibitem[BM20]{BM}
Bhargav Bhatt and Akhil Mathew.
\newblock The arc-topology.
\newblock 2020.
\newblock to appear in Duke Math. J.

\bibitem[Bos14]{Bosch}
Siegfried Bosch.
\newblock {\em Lectures on formal and rigid geometry}, volume 2105 of {\em
  Lecture Notes in Mathematics}.
\newblock Springer, Cham, 2014.

\bibitem[BS15]{BSProetale}
Bhargav Bhatt and Peter Scholze.
\newblock The pro-\'{e}tale topology for schemes.
\newblock {\em Ast\'{e}risque}, (369):99--201, 2015.

\bibitem[Con99]{ConradIrr}
Brian Conrad.
\newblock Irreducible components of rigid spaces.
\newblock {\em Ann. Inst. Fourier (Grenoble)}, 49(2):473--541, 1999.

\bibitem[dJ96]{dJAlt}
A.~J. de~Jong.
\newblock Smoothness, semi-stability and alterations.
\newblock {\em Inst. Hautes \'{E}tudes Sci. Publ. Math.}, (83):51--93, 1996.

\bibitem[dJvdP96]{dJvdP}
Johan de~Jong and Marius van~der Put.
\newblock \'{E}tale cohomology of rigid analytic spaces.
\newblock {\em Doc. Math.}, 1:No. 01, 1--56, 1996.

\bibitem[Duc18]{Ducros}
Antoine Ducros.
\newblock Families of {B}erkovich spaces.
\newblock {\em Ast\'{e}risque}, (400):vii+262, 2018.

\bibitem[Eke90]{EkedahlAdic}
Torsten Ekedahl.
\newblock On the adic formalism.
\newblock In {\em The {G}rothendieck {F}estschrift, {V}ol. {II}}, volume~87 of
  {\em Progr. Math.}, pages 197--218. Birkh\"{a}user Boston, Boston, MA, 1990.

\bibitem[FGK11]{FGK}
Kazuhiro Fujiwara, Ofer Gabber, and Fumiharu Kato.
\newblock On {H}ausdorff completions of commutative rings in rigid geometry.
\newblock {\em J. Algebra}, 332:293--321, 2011.

\bibitem[FK18]{FK}
Kazuhiro Fujiwara and Fumiharu Kato.
\newblock {\em Foundations of rigid geometry. {I}}.
\newblock EMS Monographs in Mathematics. European Mathematical Society (EMS),
  Z\"{u}rich, 2018.

\bibitem[Gab04]{GabberNotest}
Ofer Gabber.
\newblock Notes on some {$t$}-structures.
\newblock In {\em Geometric aspects of {D}work theory. {V}ol. {I}, {II}}, pages
  711--734. Walter de Gruyter, Berlin, 2004.

\bibitem[GR03]{GR}
Ofer Gabber and Lorenzo Ramero.
\newblock {\em Almost ring theory}, volume 1800 of {\em Lecture Notes in
  Mathematics}.
\newblock Springer-Verlag, Berlin, 2003.

\bibitem[Gro66]{EGAIV3}
A.~Grothendieck.
\newblock \'{E}l\'{e}ments de g\'{e}om\'{e}trie alg\'{e}brique. {IV}. \'{E}tude
  locale des sch\'{e}mas et des morphismes de sch\'{e}mas. {III}.
\newblock {\em Inst. Hautes \'{E}tudes Sci. Publ. Math.}, (28):255, 1966.

\bibitem[Han18]{H2}
David Hansen.
\newblock Remarks on nearby cycles of formal schemes.
\newblock 2018.

\bibitem[Han20]{H}
David Hansen.
\newblock Vanishing and comparison theorems in rigid analytic geometry.
\newblock {\em Compos. Math.}, 156(2):299--324, 2020.

\bibitem[Hub96]{Hub96}
Roland Huber.
\newblock {\em \'{E}tale cohomology of rigid analytic varieties and adic
  spaces}.
\newblock Aspects of Mathematics, E30. Friedr. Vieweg \& Sohn, Braunschweig,
  1996.

\bibitem[Hub98]{Hub98}
R.~Huber.
\newblock A finiteness result for the compactly supported cohomology of rigid
  analytic varieties.
\newblock {\em J. Algebraic Geom.}, 7(2):313--357, 1998.

\bibitem[ILO14]{ILO14}
Luc Illusie, Yves Laszlo, and Fabrice Orgogozo, editors.
\newblock {\em Travaux de {G}abber sur l'uniformisation locale et la
  cohomologie \'{e}tale des sch\'{e}mas quasi-excellents}.
\newblock Soci\'{e}t\'{e} Math\'{e}matique de France, Paris, 2014.
\newblock S\'{e}minaire \`a l'\'{E}cole Polytechnique 2006--2008. [Seminar of
  the Polytechnic School 2006--2008], With the collaboration of
  Fr\'{e}d\'{e}ric D\'{e}glise, Alban Moreau, Vincent Pilloni, Michel Raynaud,
  Jo\"{e}l Riou, Beno\^{\i}t Stroh, Michael Temkin and Weizhe Zheng,
  Ast\'{e}risque No. 363-364 (2014) (2014).

\bibitem[KS94]{KS}
Masaki Kashiwara and Pierre Schapira.
\newblock {\em Sheaves on manifolds}, volume 292 of {\em Grundlehren der
  Mathematischen Wissenschaften [Fundamental Principles of Mathematical
  Sciences]}.
\newblock Springer-Verlag, Berlin, 1994.
\newblock With a chapter in French by Christian Houzel, Corrected reprint of
  the 1990 original.

\bibitem[LO08]{LO}
Yves Laszlo and Martin Olsson.
\newblock The six operations for sheaves on {A}rtin stacks. {I}. {F}inite
  coefficients.
\newblock {\em Publ. Math. Inst. Hautes \'{E}tudes Sci.}, (107):109--168, 2008.

\bibitem[L{\"{u}}t93]{Lut93}
W.~L{\"{u}}tkebohmert.
\newblock Riemann's existence problem for a {$p$}-adic field.
\newblock {\em Invent. Math.}, 111(2):309--330, 1993.

\bibitem[Poi13]{P}
J\'{e}r\^{o}me Poineau.
\newblock Les espaces de {B}erkovich sont ang\'{e}liques.
\newblock {\em Bull. Soc. Math. France}, 141(2):267--297, 2013.

\bibitem[Qui]{QuillenCCNotes}
Daniel Quillen.
\newblock Homology of commutative rings.
\newblock preprint.

\bibitem[Roo06]{Roos}
Jan-Erik Roos.
\newblock Derived functors of inverse limits revisited.
\newblock {\em J. London Math. Soc. (2)}, 73(1):65--83, 2006.

\bibitem[Sai90]{SaitoDecompKahler}
Morihiko Saito.
\newblock Decomposition theorem for proper {K}\"{a}hler morphisms.
\newblock {\em Tohoku Math. J. (2)}, 42(2):127--147, 1990.

\bibitem[Sch12]{ScholzePerfectoid}
Peter Scholze.
\newblock Perfectoid spaces.
\newblock {\em Publications math{\'e}matiques de l'IH{\'E}S}, 116(1):245--313,
  2012.

\bibitem[Sch17]{Sch}
Peter Scholze.
\newblock \'etale cohomology of diamonds.
\newblock 2017.
\newblock preprint.

\bibitem[{Sta}18]{Stacks}
The {Stacks Project Authors}.
\newblock \textit{Stacks Project}.
\newblock \url{https://stacks.math.columbia.edu}, 2018.

\bibitem[SW20]{SWBerkeley}
Peter Scholze and Jared Weinstein.
\newblock Berkeley lectures on {$p$}-adic geometry.
\newblock 2020.
\newblock to appear in Annals of Math. Studies.

\bibitem[Tem18]{Temkin}
Michael Temkin.
\newblock Functorial desingularization over {$\bf Q$}: boundaries and the
  embedded case.
\newblock {\em Israel J. Math.}, 224(1):455--504, 2018.

\bibitem[Ver76]{VerdierHomology}
Jean-Louis Verdier.
\newblock Classe d'homologie associ\'{e}e \`a un cycle.
\newblock In {\em S\'{e}minaire de g\'{e}om\'{e}trie analytique (\'{E}cole
  {N}orm. {S}up., {P}aris, 1974-75)}, pages 101--151. Ast\'{e}risque, No.
  36--37. 1976.

\bibitem[War17]{W}
Evan Warner.
\newblock Adic moduli spaces.
\newblock 2017.
\newblock Stanford Ph.D. thesis, available at
  https://searchworks.stanford.edu/view/12135003.

\end{thebibliography}

\end{document}